\documentclass[a4paper]{amsart}
\usepackage{amssymb,amscd,color}
\usepackage{graphicx}
\usepackage[all]{xy}
\usepackage{tikz,tikz-cd,tikz-3dplot}
\usepackage%[backref=page]
{hyperref}

\input xypic

\newtheorem{theorem}{Theorem}[section]
\newtheorem{proposition}[theorem]{Proposition}
\newtheorem{lemma}[theorem]{Lemma}
\newtheorem{corollary}[theorem]{Corollary}

\theoremstyle{definition}

\newtheorem{example}[theorem]{Example}
\newtheorem{construction}[theorem]{Construction}

\theoremstyle{remark}

\numberwithin{equation}{section}
\def\R{\mathbb R}
\def\C{\mathbb C}

\def\Q{\mathbb Q}
\def\R{\mathbb R}

\def\Z{\mathbb Z}

\def\sK{\mathcal K}

\newcommand{\mb}[1]{{\textbf {\textit#1}}}

\renewcommand{\ge}{\geqslant}
\renewcommand{\le}{\leqslant}

\def\Ann{\mathop{\mathrm{Ann}}}
\def\Hom{\mathop{\mathrm{Hom}}}
\def\Ker{\mathop{\mathrm{Ker}}}
\def\Im{\mathop{\mathrm{Im}}}
\def\Re{\mathop{\mathrm{Re}}}

\def\cone{\mathop{\mathrm{cone}}}
\def\conv{\mathop{\mathrm{conv}}}
\def\int{\mathop{\mathrm{int}}}
\def\relint{\mathop{\mathrm{relint}}}

\def\rk{\mathcal R_{\mathcal K}}
\newcommand{\zk}{\mathcal Z_{\mathcal K}}

\begin{document}

\title[Exponential actions defined by vector configurations]{Exponential actions defined by vector configurations, Gale duality, and moment-angle manifolds}

\author{Taras Panov}
\address{Faculty of Mathematics and Mechanics, Lomonosov Moscow
State University, Leninskie gory, 119991 Moscow, Russia;\newline
%Steklov Mathematical Institute of Russian Academy of Sciences, Moscow;\newline
Institute for Information Transmission Problems, Russian Academy of Sciences, Moscow;\newline
National Research University Higher School of Economics, Moscow}
\email{tpanov@mech.math.msu.su}
\thanks{The work was funded within the framework of the HSE University Basic Research Program}
% 57S12      Toric topology
% 14M25  	Toric varieties, Newton polyhedra, Okounkov bodies
% 32J18  	Compact complex $n$-folds
% 32L05  	Holomorphic bundles and generalizations
% 32M25  	Complex vector fields, holomorphic foliations, $\mathbb{C}$-actions
% 32Q55  	Topological aspects of complex manifolds
% 37F75  	Dynamical aspects of holomorphic foliations and vector fields
% 52B35  	Gale and other diagrams
% 57R19  	Algebraic topology on manifolds and differential topology

\begin{abstract}
Exponential actions defined by vector configurations provide a universal framework for several constructions of holomorphic dynamics, non-K\"ahler complex geometry, toric geometry and topology. These include leaf spaces of holomorphic foliations, intersections of real and Hermitian quadrics, the quotient construction of simplicial toric varieties, LVM and LVMB manifolds, complex-analytic structures on moment-angle manifolds and their partial quotients, reviewed in this survey. In all cases, the geometry and topology of the appropriate quotient object can be described by combinatorial data including a pair of Gale dual vector configurations.
\end{abstract}

\maketitle
\tableofcontents

\section{Introduction and notation}
Let $V\cong\R^k$ be a $k$-dimensional real vector space, and let
$\Gamma=\{\gamma_1,\ldots,\gamma_m\}$ be a sequence (a \emph{configuration}) of $m$ vectors in the dual space~$V^*=\Hom(V,\R)$. We allow repetitions among $\gamma_1,\ldots,\gamma_m$
 and assume that they span the whole of~$V^*$. 

The \emph{exponential action} of $V$ on the standard real space $\R^m$ is given by
\begin{equation}\label{ract}
\begin{aligned}
  V\times\R^m&\longrightarrow\R^m\\
  (\mb v,\mb x)&\mapsto \mb v\cdot\mb x=\bigl(e^{\langle\gamma_1,\mb
  v\rangle}x_1,\ldots,e^{\langle\gamma_m,\mb v\rangle}x_m\bigr).
\end{aligned}
\end{equation}
This is a very classical dynamical system taking its origin in the works of Poincar\'e. 
At the same time, the action~\eqref{ract} and its holomorphic and algebraic versions arise in important modern constructions of algebraic geometry and topology, including

\begin{itemize}
\item[--] leaf spaces of holomorphic foliations, intersections of real and Hermitian quadrics (\emph{topology and holomorphic dynamics});

\item[--] the quotient construction of toric varieties (\emph{toric geometry});

\item[--] torus-invariant non-K\"ahler complex-analytic structures: LVM and LVMB manifolds (\emph{complex geometry});

\item[--] smooth and complex-analytic structures on moment-angle manifolds and their partial quotients (\emph{toric topology}).
\end{itemize}

There is a remarkable connection between the linear properties of the configuration $\Gamma$ and the topology of the foliation in $\mathbb R^m$ by the orbits of the action~\eqref{ract}, which is the subject of this survey. The connection is most effectively described by the \emph{Gale dual} configuration of vectors $\mathrm A=\{\mb a_1,\ldots,\mb a_m\}$ in the space $W^*\cong\mathbb R^{m-k}$.

From the geometric point of view, the spaces of \emph{nondegenerate leaves} of the foliation, i.\,e., subsets of $U\subset\mathbb R^m$ to which the restriction of the action~\eqref{ract} is free, are of interest. The nonfree orbits form a set of coordinate planes in $\R^m$ (a \emph{coordinate subspace arrangement}), and its complement is an open subset $U(\mathcal K)$ defined by a simplicial complex 
$\mathcal K$ on the set~$[m]$. If the action of $V$ on $U(\mathcal K)$ is not only free but also \emph{proper} (which implies the closedness of the orbits), then the leaf space, or the quotient, $U(\mathcal K)/V$ is a smooth manifold. In the case of  holomorphic or algebraic actions it is a complex-analytic manifold or algebraic variety, respectively.

The key topological fact is that the action of $V$ on $U(\mathcal K)$ is free and proper if and only if the data $\{\mathcal K;\mb a_1,\ldots,\mb a_m\}$ define a \emph{simplicial fan} $\Sigma$ in the space~$W^*$. Using Gale duality, the fan criterion can be formulated directly in terms of the original configuration~$\Gamma$ defining the action. Moreover, the orbit space $U(\mathcal K)/V$ is compact if and only if the fan~$\Sigma$ is complete. In this case $U(\mathcal K)/V$ is identified with the \emph{real moment-angle manifold} $\mathcal R_{\mathcal K}=(D^1,S^0)^{\mathcal K}$, and for the holomorphic action on~$\mathbb C^m$ with the \emph{moment-angle manifold} $\mathcal Z_{\mathcal K}=(D^2,S^1)^{\mathcal K}$. Finally, if the configuration $\Gamma$ generates a lattice (a discrete subgroup) in~$V^*$, then the fan $\Sigma$ is rational, and the action~\eqref{ract} extends to an algebraic torus action on $U(\mathcal K)\subset\mathbb C^m$ whose quotient space~is the toric variety~$X_\Sigma$. 
%(This is known as the Batyrev--Cox quotient construction.)

Also of interest is the class of exponential actions for which the corresponding fan $\Sigma$ is the normal fan of a convex polytope. There is a criterion for a fan to be normal in terms of the Gale dual cones. (In the rational case, this gives a criterion for the toric variety $X_\Sigma$ to be projective.) When $\Sigma$ is a normal fan, the orbit space $U(\mathcal K)/V$ is identified with the nondegenerate intersection of real quadrics and, in the case of a holomorphic action, with the intersection of Hermitian quadrics.

\medskip

In Section~\ref{secfree} we describe free orbits of the action~\eqref{ract} or, equivalently, nondegenerate leaves of the foliation of $\R^m$ by $V$-orbits. The stabiliser of a point $\mb x\in\R^m$ is trivial if and only if the subconfiguration of $\Gamma$ corresponding to the set of nonzero coordinates of $\mb x$ spans~$V^*$. The set of nonfree orbits is therefore an arrangement of coordinate planes in $\R^m$ corresponding to the universal simplicial complex $\mathcal K(\Gamma)=\{I\subset[m]\colon \Gamma_{\widehat I}\text{ spans }V^*\}$
associated with $\Gamma$. Any subcomplex $\mathcal K\subset\mathcal K(\Gamma)$ defines an open subset in $\R^m$, namely,
\[
  U(\mathcal K)=\R^m\setminus\bigcup_{\{i_1,\ldots,i_p\}\notin\mathcal K}
  \{\mb x\colon x_{i_1}=\cdots=x_{i_p}=0\},
\]  
such that the restriction of the exponential action~\eqref{ract} to $U(\mathcal K)$ is free.

The topology of the space of free $V$-orbits can be quite bad unless we consider subsets $U(\mathcal K)\subset\R^m$ such that the restriction of the $V$-action~\eqref{ract} to $U(\mathcal K)$ is \emph{proper}. In this case, the space of orbits $U(\mathcal K)/V$ is Hausdorff and has a smooth atlas of charts transverse to $V$-orbits. It is this quotient manifold $U(\mathcal K)/V$ that features in the constructions of algebraic geometry and topology mentioned above. The necessary and sufficient condition for the $V$-action on $U(\mathcal K)$ to be proper is the \emph{fan condition}, which is most readily formulated in terms of the Gale dual configuration.

We set up the notation for linear Gale duality in Section~\ref{galesec}.
From now on, we work with a pair of Gale dual configurations $\Gamma=\{\gamma_1,\ldots,\gamma_m\}$ in $V^*$ and $\mathrm A=\{\mb a_1,\ldots,\mb a _m\}$ in~$W^*$. The `fan condition' (Section~\ref{secfan}) is that the data $\{\mathcal K,\mathrm A\}$ define a (simplicial) fan~$\Sigma$. A pair $\{\mathcal K,\mathrm A\}$ satisfying the fan condition is known as a \emph{triangulated vector configuration}~\cite{d-r-s10}, and this concept was brought into the context of holomorphic foliations and LVMB manifolds by Battaglia and Zaffran in~\cite{ba-za15}.

The fan condition means that the cones $\cone\mathrm A_I=\R_\ge\langle\mb a_i\colon i\in I\rangle$ corresponding to $I\in\mathcal K$ do not overlap pairwise, that is, the relative interior of $\cone\mathrm A_I$ does not intersect the relative interior of $\cone\mathrm A_J$ for different $I,J$ in~$\mathcal K$. In terms of the Gale dual configuration $\Gamma$, this is equivalent to the following: the Gale dual cones $\cone\Gamma_{\widehat I}$ and $\cone\Gamma_{\widehat J}$ \emph{do overlap} for $I,J$ in~$\mathcal K$ (Theorem~\ref{galefangen}, Theorem~\ref{galefan}). This criterion appeared in several works on combinatorial convexity, often implicitly, and it is in the heart of the duality between fans and \emph{bunches} developed in~\cite[\S2.2]{a-d-h-l15}.

In Section~\ref{secproper} we prove that the exponential action of $V$ on $U(\sK)$ is proper if and only if $\{\mathcal K,\mathrm A\}$ is a fan data (Theorem~\ref{proper}). This property is of purely topological nature. On the other hand, it is the algebraic and holomorphic versions of this result that have become well known. The `fan condition' features in the quotient construction of toric varieties, it guarantees that the quotient variety is separated, i.\,e. Hausdorff in the usual topology, see~\cite[Theorem~5.1.11]{c-l-s11}. The Gale dual to the fan condition appeared as the `imbrication condition' in the work of Bosio~\cite{bosi01} as part of the definition of a class of non-K\"ahler complex manifolds arising as the leaf spaces of holomorphic foliations, now known as LVMB manifolds.  The fact that the `fan condition' and the `imbrication condition' are equivalent via Gale duality was used in several subsequent works relating complex geometry of LVMB manifolds to algebraic geometry of toric varieties~\cite{me-ve04,tamb12,batt13,ba-za15,ba-za17}. In the topological setting, the relationship between the fan condition and the separatedness of the quotient featured in the works~\cite{i-f-m13} and~\cite{ishi19}.

In Section~\ref{seccompl} we prove that the quotient $U(\sK)/V$ by a proper exponential action is compact if and only if the corresponding fan $\Sigma$ is complete (Theorem~\ref{comcom}). This property is well studied in toric geometry, but again it is of purely topological nature. When formulated in terms of the Gale dual configuration, it is known as \emph{substitute existence} in holomorphic dynamics and complex geometry of LVMB manifolds~\cite{bosi01,batt13}.

In Section~\ref{rksec} the quotient $U(\sK)/V$ corresponding to a complete simplicial fan $\Sigma$ is identified with the real moment-angle manifold $\mathcal R_{\mathcal K}$ or the polyhedral product $(D^1,S^0)^\sK$. Thus, the extensive information available on the topology of moment-angle manifolds, including the known descriptions of their cohomology rings and homotopy type, can be brought to bear on the study of the topology of foliations by orbits of exponential actions.
Topology of moment-angle manifolds and polyhedral products is currently a very active area of research, there is an extensive literature on this subject, see~\cite{bu-pa15,b-b-c20}.

Normal fans of polytopes and their associated presentations of the quotients $U(\sK)/V$ by intersections of quadrics are the subject of Section~\ref{normsec}. The criterion for a fan $\Sigma=\{\cone\mathrm A_I\colon I\in\mathcal C\}$ to be the normal fan of a polytope in terms of the Gale dual configuration~$\Gamma$ is stated in Theorem~\ref{galefannormgen}: instead of intersecting pairwise, the relative interiors of all Gale dual cones $\cone\Gamma_{\widehat I}$, $I\in\mathcal C$, must have a common intersection. This criterion plays an important role in algebraic geometry of toric varieties~\cite{od-pa91} and holomorphic dynamics and complex geometry of LVM manifolds~\cite{lo-ve97,meer00}. 
There are two ways to relate (the normal fans of) polytopes to intersections of quadrics, which are Gale dual to each other. 

The `polyhedra to quadrics' way starts with an $n$-dimensional polytope (more generally, a polyhedron) $P$ given by an intersection of halfspaces,
\[
  P=\{\mb w\in W\colon\langle \mb a_i,\mb w\rangle+b_i\ge0,\quad i=1,\ldots,m\}
\] 
and produces an intersection of quadrics defined by the Gale dual data:
\[
  \{(x_1,\ldots,x_m)\in\R^m\colon\gamma_1 x^2_1+ 
  \cdots+\gamma_mx^2_m=\delta\}.
\]
The intersection of quadrics is nondegenerate (and therefore smooth) if and only if the polytope is \emph{generic} (in particular, simple)~\cite[Theorem~6.1.3]{bu-pa15}. The nondegenerate intersection of quadrics is identified with the quotient $U(\sK)/V$ by the proper exponential action defined by the normal fan of~$P$ (Theorem~\ref{quadrics}). The proof mimics the reconstruction of a Hamiltonian toric manifold from its moment polytope via symplectic reduction~\cite{delz88,audi91,guil94}; the intersection of quadrics (the moment-angle manifold), appears here as the nondegenerate level set for the moment map. %The relationship with the original symplectic construction is reviewed in \S\ref{parsr}.

The `quadrics to polyhedra' way starts with a configuration $\Gamma$ and a choice of generic vector $\delta\in V^*$, satisfying what is known in holomorphic dynamics the `weak hypebolicity' condition. It produces a nondegenerate intersection of quadrics, which is identified with the quotient $U(\sK)/V$ of the exponential action defined by the normal fan of the associated polyhedron~$P$ (Theorem~\ref{quadrics1} and Proposition~\ref{ptopepos}). 

There is an important particular case when the vectors of the $\mathrm A$-configuration satisfy a linear relation with all positive coefficients. In this case, $\cone \mathrm A$ is the whole of $W^*$, while $\cone\Gamma$ is a strongly convex (pointed) cone.
The Gale dual vector configuration $\Gamma$ can be replaced by an affine point configuration $\Gamma'$ in an affine hyperplane of $V^*$ (the \emph{affine Gale transform} of~$\mathrm A$). The `fan condition' guaranteeing the properness of the exponential action can be restated in terms of the convex hulls of points $\conv\Gamma'_{\widehat I}$ instead of cones $\cone\Gamma_{\widehat I}$. This is precisely the setup used in the definition of LVM and LVMB manifolds, see Section~\ref{convlink}.

Holomorphic exponential actions are considered in Section~\ref{compzk}. Given a complex vector space $\widetilde V\cong\C^\ell$ and a vector configuration $\Gamma$ in $\widetilde V^*=\Hom(\widetilde V,\C)$, the holomorphic action is defined by
\[
\begin{aligned}
  \widetilde V\times\C^m&\longrightarrow\C^m\\
  (\mb v,\mb z)&\mapsto \mb v\cdot\mb z=\bigl(e^{\langle\gamma_1,\mb
  v\rangle_\C}z_1,\ldots,e^{\langle\gamma_m,\mb v\rangle_\C}z_m\bigr).
\end{aligned}
\]
The criterions for the properness of the holomorphic action and the compactness of the quotient $U_\C(\sK)/\widetilde V$ are the same as in the real case (Theorems~\ref{hproper} and~\ref{hcompact}). The quotient $U_\C(\sK)/\widetilde V$ of the proper holomorphic exponential action defined by a complete simplicial fan data $\{\mathcal K,\mathrm A\}$ is identified with the moment-angle manifold $\zk=(D^2,S^1)^\sK$ (Theorem~\ref{quotzk}). The latter therefore acquires a complex-analytic structure. Complex geometry of moment-angle manifolds is remarkably rich and peculiar~\cite{pa-us12,p-u-v16,i-k-p22} and is far from being understood completely.

When the configuration $\Gamma$ is \emph{rational}, that is, its $\Z$-span $L^*=\Z\langle\Gamma\rangle$ is a discrete subgroup (a lattice) in~$V^*$, both the real and holomorphic exponential actions extend to an action of the algebraic torus $\C^\times_L=L\otimes_\Z\C^\times\cong V\times T_L$ on $U_\C(\sK)$, where $T_L=L\times_\Z S^1$ is the compact torus defined by~$L$. If the algebraic action of $\C^\times_L$ on $U_\C(\sK)$ is proper, then the data $\{\sK,\mathrm A\}$ define a fan $\Sigma$ which is rational with respect to the Gale dual lattice $N=\Z\langle\mathrm  A\rangle$ in~$W^*$. The quotient $U_\C(\sK)/\C^\times_L$ is the toric variety $V_\Sigma$ defined by the rational fan~$\Sigma$. We obtain a holomorphic principal bundle $\zk\to V_\Sigma$ with fibre a holomorphic compact torus $\C^\times_L/\widetilde V\cong T_L$, known as the \emph{generalised Calabi--Eckmann fibration}~\cite{lo-ni96, lo-ve97, me-ve04}. Perturbing the vector configuration $\Gamma$ %(or~$\mathrm A$) 
destroys its rationality, the closed holomorphic tori in the fibres of the bundle $\zk\to X_\Sigma$ `open up', and the fibre bundle turns into a \emph{holomorphic foliation} $\mathcal F$ of $\zk$ with noncompact leaves. Holomorphic foliated manifolds $(\zk,\mathcal F)$ model \emph{irrational deformations} of toric varieties~\cite{ba-pr01, ba-za15, ra-zu19, k-l-m-v21, i-k-p22, ba-pr23}.

Thinking of $V$ as the space of linear relations between the vectors $\mb a_1,\ldots,\mb a_m$, we observe that $\Gamma$ is a rational configuration if and only if $V$ is generated by relations with integer coefficients. More generally, a
subspace $Q\subset V$ is \emph{rational} (with respect to~$\Gamma$) if it is generated by relations with integer coefficients (Construction~\ref{coras}). A rational subspace $Q\subset V$ contains a full-rank lattice
\[
  L=\{\mb q\in Q\colon\langle\gamma_k,\mb q\rangle\in\Z\;
  \text{ for }k=1,\ldots,m\}.
\]
The exponential action of $V$ on $\R^m$ extends to an action of the group $\C^\times_L\times_Q V\cong T_L\times V$ on $\C^m$, called the \emph{torus-exponential action} (Section~\ref{sectea}). Once again, the torus-exponential action $(\C^\times_L\times_Q V)\times U_\C(\sK)\to U_\C(\sK)$ is proper if and only if $\{\sK,\mathrm A\}$ is a fan data (Theorem~\ref{teproper}).
The quotient $U_\C(\sK)/(T_L\times V)$ by a proper torus-exponential action is identified with the quotient $\zk/T_L$ and is called a \emph{partial quotient} of a moment-angle manifold. There is also a holomorphic version of torus-exponential action $(\C^\times_L\times_{\widetilde Q}{\widetilde V})\times U_\C(\sK)\to U_\C(\sK)$, defined by a rational complex subspace $\widetilde Q\subset\widetilde V$. If $\{\sK,\mathrm A\}$ is a fan data, this defines a complex-analytic structure on the partial quotient $\zk/T_L$ as the quotient $U_\C(\sK)/(\C^\times_L\times_{\widetilde Q}{\widetilde V})$, with a holomorphic torus principal bundle $\zk\to\zk/T_L$.

Complex moment-angle manifolds $\zk$ are universal examples of \emph{complex manifolds with maximal torus action}. According to a classification result of Ishida~\cite{ishi19}, any such manifold is biholomorphic to a holomorphic partial quotient $\zk/T_L$. Two cases of holomorphic partial quotients $\zk/T_L$, corresponding to a one-dimensional rational subspace $Q\subset V$ and the whole of $V$ being rational, respectively, are of particular importance.

LVM manifolds are a class of non-K\"ahler complex-analytic manifolds introduced by Lopez de Medrano and Verjovsky~\cite{lo-ve97} and Meersseman~\cite{meer00} as the leaf spaces of exponential holomorphic foliations on the projectivisations of open subsets $U_\C(\sK)\subset\C^m$. This construction was further generalised by Bosio~\cite{bosi01} and became known as LVMB manifolds. In the original construction of~\cite{bosi01} the conditions ensuring that the leaf space is Hausdorff (or the action is proper) were formulated in terms of overlapping convex regions (see Section~\ref{LVMBsec}). These conditions by Bosio are equivalent to the fan condition on the Gale dual data~$\{\mathcal K,\mathrm A\}$ (Proposition~\ref{LVMBfan}). 

LVMB manifolds are therefore identified with projectivised moment-angle manifolds, that is, holomorphic partial quotients of $\zk$ by the diagonal circle (corresponding to a one-dimensional rational subspace $Q\subset V$). In these terms, LVM manifolds can be described as LVMB manifolds corresponding to the normal fans of polytopes, and are diffeomorphic to nondegenerate intersections of special Hermitian quadrics.

Toric manifolds or smooth complete toric varieties $X_\Sigma$ are partial quotients of moment-angle manifolds $\zk$ by a maximal freely acting subtorus. The corresponding holomorphic torus principal bundles $\zk\to X_\Sigma$ are described in Section~\ref{sectoric}.  Theorem~\ref{galefannormgen} gives a criterion for the rational fan $\Sigma=\{\cone\mathrm A_I\colon I\in\mathcal C\}$ of $X_\Sigma$ to 
be the normal fan of a polytope in terms of the Gale dual cones, which is equivalent to projectivity of~$X_\Sigma$. This criterion can be extended to a description of the nef and ample cones of~$X_\Sigma$ (\cite[Proposition~2.4.2.6]{a-d-h-l15} and Proposition~\ref{nefcone}). We expect a similar criterion to hold in the irrational setting, thereby providing a description of nef and ample divisors on irrational deformations of toric varieties. We finish by describing the construction of Hamiltonian toric manifolds via symplectic reduction and their connection with projective nonsingular toric varieties (Theorem~\ref{algsym}). Likewise, this connection extends to the irrational setting.
%\[
%  \mathop\mathrm{Nef}(X_\Sigma)=\bigcap_{I\in\mathcal C}\cone\Gamma_{\widehat I},\qquad
%  \mathop\mathrm{Ample}(X_\Sigma)=\bigcap_{I\in\mathcal C}
%  \relint\cone\Gamma_{\widehat I}.
%\]

%\section{Notation}
The following notation is used throughout the paper:
\begin{itemize}
\item
$[m]=\{1,2,\ldots,m\}$ an ordered set of $m$ elements.

\item
$I=\{i_1,\ldots,i_k\}\subset [m]$ a subset.

\item
$\widehat I=[m]\setminus I$ the complementary subset.

\item
$\mathcal K$ an (abstract) \emph{simplicial complex} on $[m]$. 
%That is, a collection $\sK$ of subsets $I\subset[m]$ such that if $I\in\sK$ and $J\subset I$, then $J\in\sK$. 
Every simplicial complex is assumed to contain the empty set~$\varnothing$.

\item
$I\in\sK$ is a \emph{simplex} of~$\sK$.

\item
A \emph{vertex} of $\sK$ is a one-element simplex $\{i\}\in\sK$.

\item
A \emph{ghost vertex} of $\sK$ is a one element subset $\{i\}\subset[m]$ that is not  in~$\sK$.

\item
$\R^m$ real vector space consisting of $m$-tuples $\mb x=(x_1,\ldots,x_m)$ of real numbers.

\item
$\C^m$ complex vector space consisting of $m$-tuples $\mb z=(z_1,\ldots,z_m)$ of complex numbers.

\item
$\R_\ge$ nonnegative real numbers.

\item
$\R_>$ multiplicative group of positive real numbers.

\item
$\C^\times$ multiplicative group of nonzero complex numbers.

\item
$U(\mathcal K)=\R^m\setminus\bigcup_{\{i_1,\ldots,i_p\}\notin\mathcal K}
  \{\mb x\colon x_{i_1}=\cdots=x_{i_p}=0\}$ the complement of the arrangement of real coordinate subspace defined by~$\sK$.
  
\item
$U_\C(\mathcal K)=\C^m\setminus\bigcup_{\{i_1,\ldots,i_p\}\notin\mathcal K}
  \{\mb z\colon x_{i_1}=\cdots=x_{i_p}=0\}$ the complement of the complex coordinate subspace arrangement defined by~$\sK$. 

\item
$V$ a $k$-dimensional real vector space.

\item
$V^*=\Hom_\R(V,\R)$ the dual space.

\item
$\langle\gamma,\mb v\rangle\in\R$ the inner product (pairing) $V^*\times V\to\R$.

\item
$\Gamma=\{\gamma_1,\ldots,\gamma_m\}$ a vector configuration in~$V^*$. Assume that $\Gamma$ spans $V^*$.

\item
$\mathrm A=\{\mb a_1,\ldots,\mb a_m\}$ the Gale dual vector configuration in $W^*$ (see Section~\ref{galesec}).

\item
$\Gamma_I=\{\gamma_i\colon i\in I\}$ a subconfiguration of $\Gamma$. Similarly, 
$\mathrm A_I$, $\Gamma_{\widehat I}=\{\gamma_i\colon i\notin I\}$, etc.

\item
$\R\langle\Gamma_I\rangle=\R\langle\gamma_i\colon i\in I\rangle$ the $\R$-span of a vector configuration (assume $\R\langle\Gamma\rangle=V^*$).

\item $\varGamma\colon\R^m\to V^*$ the linear map given by $\mb e_i\mapsto\gamma_i$, where $\mb e_i$ is the $i$th standard basis vector of~$\R^m$.

\item $\varGamma^*\colon V\to\R^m$ the linear map $\mb v\mapsto
(\langle\gamma_1,\mb v\rangle,\ldots,\langle\gamma_m,\mb v\rangle)$.

\item $\mathit A\colon\R^m\to W^*$ the linear map given by $\mb e_i\mapsto\mb a_i$.

\item $\mathit A^*\colon W\to\R^m$ the linear map $\mb w\mapsto
(\langle\mb a_1,\mb w\rangle,\ldots,\langle\mb a_m,\mb w\rangle)$.  

\item
$\cone\mathit A_I=\cone(\mb a_i\colon i\in I)$  the cone generated by (the nonnegative span of) a set of vectors. Similarly, $\cone\Gamma_{\widehat I}$, etc.

\item
$\relint C$ the relative interior (the interior in the affine span) of a convex set~$C$.

\item
$\Sigma$ a \emph{fan} (see Section~\ref{secfan}).

\item
$\Sigma_\sK=\{\cone(\mb e_i\colon i\in I)\text{ with }I\in\sK\}$ the \emph{coordinate fan} in $\R^m$ corresponding to a simplicial complex~$\sK$.

\item
$P$ a convex \emph{polyhedron} in $W^*$. A \emph{polytope} is a bounded polyhedron.

\item
$\Sigma_P$ the \emph{normal fan} of~$P$.

\item
$\Gamma'=\{\gamma'_1,\ldots,\gamma'_m\}$ a point configuration in an affine space.

\item
$\conv\Gamma'_I=\conv(\gamma'_i\colon i\in I)$ the convex hull of a point configuration.

\item
$\widetilde V$ the complex vector space of dimension $\ell$ obtained by endowing an even-dimensional real space $V$ ($k=2\ell$) with a complex structure~$\mathcal J$.

\item
$\langle\gamma,\mb v\rangle_\C$ the complex inner product (pairing) 
$\widetilde V^*\times\widetilde V\to\C$, it satisfies
$\langle\gamma,\mb v\rangle_\C=\langle\gamma,\mb v\rangle+
  i\langle\gamma,\mathcal J\mb v\rangle$.

\item
$V_\C$ the complexification of $V$ (a $k$-dimensional complex vector space).

\item
$(\mb X,\mb A)^\sK$ the \emph{polyhedral product} (see Section~\ref{rksec}).

\item
$\rk=(D^1,S^0)^\sK$ the \emph{real moment-angle complex} (see Section~\ref{rksec}).

\item
$\zk=(D^2,S^1)^\sK$ the \emph{moment-angle complex} (see Section~\ref{compzk}).

\item
$T^m$ the $m$-dimensional torus, $T^m=(S^1)^m$.

\item
$Q$ a rational subspace in $V$ (see Section~\ref{secgr}).

\item
$L=\{\mb q\in Q\colon\langle\gamma_k,\mb q\rangle\in\Z\;
  \text{ for }k=1,\ldots,m\}$ the full-rank lattice in $Q$.
  
\item
$L_I=\{\mb q\in Q\colon\langle\gamma_k,\mb q\rangle\in\Z\;\text{ for }k\notin I\}$ a subgroup of~$Q$.

\item
$\widehat\Gamma=\{\widehat\gamma_1,\ldots,\widehat\gamma_m\}$ a rational vector configuration in~$Q^*$.

\item
$\widehat{\mathrm A}=\{\widehat{\mb a}_1,\ldots,\widehat{\mb a}_m\}$ the Gale dual rational vector configuration in~$U^*$.

\item
$\C^\times_L$ the algebraic torus defined by a vector space $Q$ with a full-rank lattice $L$, namely $\C^\times_L=L\otimes_\Z\C^\times\cong Q_\C/(2\pi i L)$.

\item
$T_L$ the compact torus, $T_L=L\otimes_\Z S^1 \cong Q/(2\pi L)$.

\item
$X_\Sigma$ the toric variety corresponding to a rational fan~$\Sigma$ (Section~\ref{sectoric}).

\item
$N=M^*$ the full rank lattice spanned by or containing a rational vector configuration $\mathrm A=\{\mb a_1,\ldots,\mb a_m\}$ in~$W^*$.

\item 
$\mathop\mathrm{Nef}(X_\Sigma)$ the nef cone of $X_\Sigma$, generated by the classes of basepoint free divisors on $X_\Sigma$.

\item 
$\mathop\mathrm{Ample}(X_\Sigma)$ the ample cone of $X_\Sigma$ (the set of positive multiples of ample divisor classes on $X_\Sigma$).
\end{itemize}

The author is grateful to the anonymous referees for very helpful comments and suggestions.

\section{Free orbits (nondegenerate leaves)}\label{secfree}
We start with a description of free orbits of the exponential action~\eqref{ract}.

\begin{proposition}\label{freeorbit}
The stabiliser subgroup $V_{\mb x}=\{\mb v\in V\colon \mb v\cdot\mb x=\mb x\}$
of a point $\mb x=(x_1,\ldots,x_m)\in\R^m$ under the action~\eqref{ract} is given by
\[
  V_{\mb x}=\{\mb v\in V\colon\langle\gamma_i,\mb v\rangle=0\;
  \text{ whenever }x_i\ne0\}.
\]
In particular, $V_{\mb x}=\bf0$ if and only if
the subset $\{\gamma_i\colon x_i\ne0\}\subset\Gamma$ spans~$V^*$.
\end{proposition}
\begin{proof}
We have $\mb v\in V_{\mb x}$ if and only if
\[
  (x_1e^{\langle\gamma_1,\mb
  v\rangle},\ldots,x_me^{\langle\gamma_m,\mb v\rangle})=(x_1,\ldots,x_m),
\]
which holds precisely when $\langle\gamma_i,\mb v\rangle=0$ for $x_i\ne0$.

For the second statement, $V_{\mb x}\ne0$ if and only if there is a nonzero $\mb v\in V_{\mb x}$ such that $\langle\gamma_i,\mb v\rangle=0$ for $x_i\ne0$, which holds precisely when $\{\gamma_i\colon x_i\ne0\}\subset\Gamma$ does not span~$V^*$.
\end{proof}

Since $\Gamma$ itself spans $V^*$, the stabiliser of a generic point (with all coordinates nonzero) is zero, so the action is free on a dense open subset 
of~$\R^m$.

For a subset $I=\{i_1,\ldots,i_p\}\subset[m]$, denote 
$\Gamma_I=\{\gamma_i\colon i\in I\}$.
Let $\widehat I= [m]\setminus I$ denote the complementary subset. We set
\[
  \mathcal K(\Gamma)=\{I\subset[m]\colon \Gamma_{\widehat I}\text{ spans }V^*\}.
\]

A \emph{simplicial complex} on $[m]$ is a collection $\sK$ of subsets of $[m]$ such that
for any $I\in\sK$ all subsets of $I$ also belong to~$\sK$. We refer to $I\in\sK$ as a \emph{simplex} (or a \emph{face}) of~$\sK$. We
always assume that the empty set $\varnothing$ belongs to~$\sK$. We do not assume that $\sK$ contains all one-element subsets $\{i\}$. We refer to $\{i\}\in\sK$ as \emph{vertices}, and refer to $\{i\}\notin\sK$ as \emph{ghost vertices}. A simplicial complex is \emph{pure} if all its maximal simplices have the same dimension. (The dimension of a simplex is its cardinality minus one.)

\begin{proposition}
$\mathcal K(\Gamma)$ is a pure simplicial complex of dimension~$m-k-1$.
\end{proposition}
\begin{proof}
If $\Gamma_{\widehat I}$ spans $V^*$, then so does $\Gamma_{\widehat J}\supset\Gamma_{\widehat I}$ for any $J\subset I$. Hence, $\mathcal K(\Gamma)$ is a simplicial complex. Also, if $\Gamma_{\widehat I}$ spans $V^*$, then it contains a basis of $V^*$. Such a basis has the form $\Gamma_{\widehat L}$ for some $L$ with $I\subset L$ and $|L|=m-
|\Gamma_{\widehat L}|=m-k$. It follows that $\mathcal K(\Gamma)$ is pure $(m-k-1)$-dimensional.
\end{proof}

Given a simplicial complex $\mathcal K$ on $[m]$, define the following open subset in~$\R^m$ (the complement of an arrangement of coordinate subspaces):
\begin{equation}\label{UK}
  U(\mathcal K)=\R^m\setminus\bigcup_{\{i_1,\ldots,i_p\}\notin\mathcal K}
  \{\mb x\colon x_{i_1}=\cdots=x_{i_p}=0\}.
\end{equation}
It is easy to see that the complement to any set of coordinate planes in $\R^m$ has the form $U(\sK)$ for some~$\sK$.
For example, if $\sK=\{\varnothing\}$, then $U(\sK)=(\R^\times)^m$, where $\R^\times=\R\setminus\{0\}$, and if $\sK$ consists of all proper subsets of $[m]$, then $U(\sK)=\R^m\setminus\{\bf 0\}$.

Let
\begin{equation}\label{affui}
  U_I=\{(x_1,\ldots,x_m)\in\R^m\colon x_j\ne0\text{ for }j\notin I\}.
\end{equation}
Then $U(\sK)=\bigcup_{I\in\sK} U_I$.

\begin{proposition}\label{free}
If $\sK\subset\sK(\Gamma)$ is a simplicial subcomplex, then the restriction of the action~\eqref{ract} to $U(\sK)$ is free.
\end{proposition}
\begin{proof}
Take $\mb x=(x_1,\ldots,x_m)\in U(\sK)$. Then $I=\{i\colon x_i=0\}\in\sK$. As $\sK\subset\sK(\Gamma)$, we obtain that $\Gamma_{\widehat I}=\{\gamma_i\colon x_i\ne0\}$ spans~$V^*$. Then $V_{\mb x}=\bf0$ by Proposition~\ref{freeorbit}.
\end{proof}

We restate this by saying that $U(\sK)$ consists of \emph{nondegenerate leaves} of~\eqref{ract} for any $\sK\subset\sK(\Gamma)$.

\section{Linear Gale duality}\label{galesec}
The configuration $\Gamma$ defines a linear map $\varGamma\colon\R^m\to V^*$ that takes the $i$th standard basis vector $\mb e_i$ to~$\gamma_i$. Let $W=\Ker\varGamma$, so we have an exact sequence
\begin{equation}\label{1exseq}
  0\longrightarrow W\longrightarrow\R^m
  \stackrel{\varGamma}\longrightarrow V^*\longrightarrow0.
\end{equation}
Since $\Gamma$ spans $V^*$, the space $W^*$ has dimension $n:=m-k$. We identify the dual space of $\R^m$ with $\R^m$ via the
isomorphism that takes the standard basis to its dual basis, and write vectors in $\R^m$ by their coordinates in the standard basis. Consider the dual exact sequence
\[
  0\longrightarrow V\stackrel{\varGamma^*}{\longrightarrow}\R^m
  \stackrel{A}\longrightarrow W^*\longrightarrow0,
\]
where the map $\varGamma^*$ takes $\mb v$ to $(\langle\gamma_1,\mb v\rangle,\ldots,\langle\gamma_m,\mb v\rangle)$.
Define $\mb a_i=A(\mb e_i)$. The \emph{linear Gale transform} takes the vector configuration $\Gamma$ in $V^*$ to the vector configuration $\mathrm{A}=\{\mb a_1,\ldots,\mb a_m\}$ in $W^*$. It is an involutive procedure, so $\mathrm A$ is called the \emph{Gale dual}
of~$\Gamma$, and $\Gamma$ is the Gale dual of~$\mathrm{A}$. Then $\mathrm A$ spans~$W^*$ by duality. The linear function $\mb a_i$ takes a vector in $W$ (viewed as a subspace in~$\R^m$) to its $i$th coefficient.
%Whenever we talk about a pair of Gale dual configurations $\mathrm{A}$ and $\Gamma$, we assume that $\mathrm A$ spans~$W^*$ and $\Gamma$ spans~$V^*$.

Choosing bases in $V$ and $W$, we view $\varGamma$ as a $k\times m$-matrix with column vectors $\gamma_1,\ldots,\gamma_m$ and view $A$ as a $n\times m$-matrix with column vectors $\mb a_1,\ldots,\mb a_m$. The identity $A\varGamma^*=0$ implies that the rows of $A$ form a basis in the space of linear relations between the vectors $\gamma_1,\ldots,\gamma_m$. The latter space is just $W$, and we may think of $\mb a_i\in W^*$ as the linear function that takes a linear relation to its $i$th coefficient.

\begin{proposition}\label{galeindspan}
%Let $\Gamma$ and $\mathrm{A}$ be Gale dual configurations of $m$ vectors in $V^*$ and $W^*$, %respectively. 
For any $I\subset[m]$, the vectors in $\mathrm{A}_I=\{\mb a_i\colon i\in I\}$ are linearly independent in $W^*$ if and only if $\Gamma_{\widehat I}=\{\gamma_j\colon j\notin I\}$ spans~$V^*$.  
\end{proposition}
\begin{proof}
By duality, we think of $\gamma_j$ as the $j$th coefficient function on the space $V$ of linear relations between $\mb a_1,\ldots,\mb a_m$. If the vectors $\{\mb a_i$, $i\in I\}$ are linearly independent, then any linear relation between $\mb a_1,\ldots,\mb a_m$ has the $j$th coefficient nonzero for some $j\in\widehat I$. It follows that, for any vector in $V$, there is a linear function among $\Gamma_{\widehat I}=\{\gamma_j$, $j\in\widehat I\}$ that does not vanish on this vector. This implies that $\Gamma_{\widehat I}$ spans~$V^*$. Conversely, if
there is a linear relation between $\{\mb a_i$, $i\in I\}$ then for each $j\in\widehat I$ the $j$th coefficient function $\gamma_j$ vanishes on this relation. Hence, the vectors $\{\gamma_j$, $j\in\widehat I\}$ do not span~$V^*$.
\end{proof}

\section{Fans and triangulated configurations}\label{secfan}
%We collect basic facts about polyhedral cones and fans here, referring to~\cite{bron83}, \cite[Chapter~1]{fult93} and~\cite[Chapters~1 and~3]{c-l-s11} for a more detailed treatment.

A subset $\sigma$ of a real vector space $W^*$ is called a \emph{polyhedral cone}, or simply a \emph{cone}, if it consists of all linear combinations with nonnegative coefficients of a finite set of vectors $\mb a_1,\ldots,\mb a_p$ in~$W^*$. The notation is $\sigma=\cone(\mb a_1,\ldots,\mb a_p)$. The set of vectors may be empty, in which case the cone consists of the zero vector only, that is,
$\cone(\varnothing)=\mathbf{0}$. The dimension of $\sigma$ is the dimension of its linear span.

A \emph{generator set} of a cone $\sigma$ is a set of vectors $\mb a_1,\ldots,\mb a_p$ such that $\sigma=\cone(\mb a_1,\ldots,\mb a_p)$. A cone is \emph{strongly convex} or \emph{pointed} if it does not contain a line.  An inclusion-minimal generator set of a strongly convex cone is unique up to multiplication by positive numbers. A cone is \emph{simplicial} if its minimal set of generators is linearly independent.

A \emph{supporting hyperplane} of a cone $\sigma\subset W^*$ is a hyperplane $H$ such that $\sigma$ is contained in one of the two closed halfspaces defined by~$H$. A \emph{face} of a cone $\sigma$ is the intersection of $\sigma$ with a supporting hyperplane. We also regard $\sigma$ as a face of itself; faces which are different from~$\sigma$ are called~\emph{proper}. The complement of the union of all proper faces in~$\sigma$ is called the \emph{relative interior} of~$\sigma$ and is denoted by $\relint\sigma$; it coincides with the interior of $\sigma$ in its linear span.
Note that $\relint{\mathbf 0}=\mathbf 0$.

A nonzero vector $\mb w\in W$ in the dual space defines a hyperplane
\[
  H_{\mb w}=\{\mb a\in W^*\colon\langle\mb a,\mb w\rangle=0\}.
\]   
We denote by $H_{\mb w}^+$ and $H_{\mb w}^-$ its corresponding nonnegative and nonpositive closed halfspaces, respectively. %If $\mb u=0$, then $H_{\mb u}$ is the whole space~$L$ rather than a hyperplane.

\begin{proposition}\label{positive}
If $\sigma=\cone(\mb a_1,\ldots,\mb a_p)$ is a cone in~$W^*$, then $\mb a\in\relint\sigma$ if and only if $\mb a=\sum_{i=1}^p \lambda_i\mb a_i$ for some positive $\lambda_i\in\R$.
\end{proposition}
\begin{proof}
The case $\mb a=\mathbf0$ is trivial, so we may assume $\mb a\ne\mathbf 0$.

Let $\mb a=\sum_{i=1}^p \lambda_i\mb a_i$ with $\lambda_i>0$. Suppose that $\mb a\notin\relint\sigma$. Then $\mb a$ is contained in a proper face $\tau$ of $\sigma$, i.\,e. there is a supporting hyperplane $H_{\mb w}$ with $\mb w\in W$ such that 
$\sigma\subset H_{\mb w}^+$ and $\mb v\in H_{\mb w}\cap\sigma=\tau$. Hence, $\langle\mb a,\mb w\rangle=0$. On the other hand, $\langle\mb a,\mb w\rangle=\sum_{i=1}^p \lambda_i\langle\mb a_i,\mb w\rangle>0$, because each $\langle\mb a_i,\mb w\rangle$ is nonnegative and $\langle\mb a_i,\mb w\rangle>0$ for at least one~$i$, as $\tau$ is a proper face. A contradiction.

Now let $\mb a\in\relint\sigma$. It is easy to see that for any $\mb a'\in\sigma$ there exists $\mb a''\in\sigma$ such that $\mb a'+\mb a''=\mu\mb a$ for some $\mu>0$. Taking $\mb a'=\mb a_i$ we obtain
$\mb a_i+\mb a''_i=\mu_i\mb a$ for $\mb a''_i\in\sigma$ and $\mu_i>0$. Writing each $\mb a''_i$ as a nonnegative linear combination of $\mb a_1,\ldots,\mb a_p$ and summing up over $i$ we obtain
$\sum_{i=1}^p\nu_i\mb a_i=\sum_{i=1}^p\mu_i\mb a$ with positive $\nu_1,\ldots,\nu_p$. The result follows.
\end{proof}

\begin{lemma}[Separation Lemma]\label{separation}
The following conditions are equivalent for two different strongly convex cones $\sigma_1$ and $\sigma_2$ in~$V$:
\begin{itemize}
\item[(a)] the intersection of $\sigma_1$ and $\sigma_2$ is a proper face of each cone;

\item[(b)] the cones $\sigma_1$ and $\sigma_2$ are separated by a hyperplane, i.\,e. there exists a hyperplane $H$ such that 
\[
  \sigma_1\subset H^+,\quad  \sigma_2\subset H^-\quad\text{and}\quad
  H\cap\sigma_1=H\cap\sigma_2=\sigma_1\cap\sigma_2. 
\]
\end{itemize}
\end{lemma}

For the proof, see~\cite[\S1.2]{fult93} or~\cite[Lemma~1.2.13]{c-l-s11}.

A \emph{fan} is a finite set $\Sigma=\{\sigma_1,\ldots,\sigma_s\}$ of strongly convex cones in $W^*$ such that every face of a cone in $\Sigma$ belongs to $\Sigma$ and the intersection of any two cones in $\Sigma$ is a face of each. The \emph{support}  $|\Sigma|$ of a fan is the union of its cones: $|\Sigma|=\bigcup_{\sigma\in\Sigma}\sigma\subset W^*$. A fan is \emph{complete} if $|\Sigma|=W^*$ and has \emph{convex support} if $|\Sigma|$ is convex.

%Now we turn back to Gale dual vector configurations $\Gamma=\{\gamma_1,\ldots,\gamma_m\}$ and $\mathrm A=\{\mb a_1,\ldots,\mb a_m\}$ in $V^*$ and $W^*$, respectively. Recall that we assume that $\Gamma$ spans $V^*$ and $\mathrm A$ spans~$W^*$.

A fan $\Sigma$ in $W^*$ can be defined by specifying generators of its one-dimensional cones and the subsets of generators that span cones. This justifies the following formalism. Let $\mathrm A=\{\mb a_1,\ldots,\mb a_m\}$ be a spanning vector configuration in $W^*$ and let $\mathcal C$ be a collection of subsets 
of~$[m]$ containing~$\varnothing$. We say that $\mathcal C$ is \emph{$\mathrm A$-closed} if 
\begin{itemize}
\item
$\cone\mathrm A_I$ is strongly convex for any $I\in\mathcal C$,

\item if $I\in\mathcal C$, $J\subset I$ and $\cone\mathrm A_J$ is a face of $\cone\mathrm A_I$, then $J\in\mathcal C$.
\end{itemize}
If $\mathcal C$ is $\mathrm A$-closed, then all faces of a cone in the set 
$\{\cone \mathrm A_I\colon I\in\mathcal C\}$ also belong to this set. The set $\{\cone \mathrm A_I\colon I\in\mathcal C\}$ is a fan whenever any two cones $\cone\mathrm A_I$ and $\cone\mathrm A_J$ intersect in a common face (which has to be $\cone\mathrm A_{I\cap J}$). In this case we say that $\{\mathcal C,\mathrm A\}$ is a \emph{fan data}.

The ``fan condition'' can be stated in terms of the relative interiors of cones in either of the two Gale dual configurations $\mathrm A$ and~$\Gamma$; the next result also follows from the general correspondence between fans and \emph{bunches} in~\cite[Theorem~2.2.1.14]{a-d-h-l15}:

\begin{theorem}\label{galefangen}
Let $\mathrm A=\{\mb a_1,\ldots,\mb a_m\}$ and $\Gamma=\{\gamma_1,\ldots,\gamma_m\}$ be a pair of Gale dual vector configurations, and let $\mathcal C$ be an $\mathrm A$-closed collection of subsets of~$[m]$. The following conditions are equivalent:

\begin{itemize}
\item[(a)] $\{\cone \mathrm A_I\colon I\in\mathcal C\}$ is a fan;\\[-0.7\baselineskip]

\item[(b)]
$(\relint\cone\mathrm A_I)\cap(\relint\cone\mathrm A_J)=\varnothing$ for any $I,J\in\mathcal C$, $I\ne J$;\\[-0.7\baselineskip]

\item[(c)]
  $(\relint\cone\Gamma_{\widehat I})\cap(\relint\cone\Gamma_{\widehat J})\ne\varnothing$ for any $I,J\in\mathcal C$.
\end{itemize}
\end{theorem}

%\begin{remark}
%Condition (a) implies that $\relint\cone(\mathrm A_I)\cap\relint\cone(\mathrm A_J)=\varnothing$, 
%but not vice versa.
%\end{remark}

\begin{proof} The implication (a)$\Rightarrow$(b) is clear.

(b)$\Rightarrow$(a). Take any $I,J\in\mathcal C$. Set $X=\cone\mathrm A_I\cap
\cone\mathrm A_J$. Consider the smallest face of $\cone\mathrm A_I$ that contains $X$; this face is $\cone\mathrm A_{I'}$ for some $I'\subset I$, $I'\in\mathcal C$. Similarly, let $\cone\mathrm A_{J'}$ be the smallest face of $\cone\mathrm A_J$ containing~$X$. Then $\relint X\subset(\relint \cone\mathrm A_{I'})\cap(\relint \cone\mathrm A_{J'})$. By (b), the latter intersection is empty unless $I'=J'$. Therefore, $I'=J'$. Then $I'\subset I\cap J$, so $X\subset \cone\mathrm A_{I'}\subset\cone\mathrm A_{I\cap J}$. On the other hand, we obviously have $\cone\mathrm A_{I\cap J}\subset X$.  Hence, the intersection $\cone\mathrm A_I\cap\cone\mathrm A_J=\cone A_{I'}=\cone\mathrm A_{I\cap J}$ is a face of both $\cone\mathrm A_I$ and $\cone\mathrm A_J$.

(c)$\Rightarrow$(a).
Suppose $(\relint\cone\Gamma_{\widehat I})\cap(\relint\cone\Gamma_{\widehat J})\ne\varnothing$. By Proposition~\ref{positive},
\begin{equation}\label{intrel}
  \sum_{k\in\widehat I}r_k\gamma_k-\sum_{l\in\widehat J}s_l\gamma_l=\mathbf 0
\end{equation}
for some positive $r_k$ and $s_l$. This is a linear relation between $\gamma_1,\ldots,\gamma_m$, so it defines a vector $\mb w\in W$. %, or a linear function on~$W^*$. 
Recall that the value of $\mb a_i$ on a linear relation is its $i$th coefficient, so we have 
\begin{equation}\label{wacases}
  \langle\mb a_i,\mb w\rangle=\begin{cases}
  r_i-s_i,& i\in\widehat I\cap\widehat J=\widehat{I\cup J},\\
  r_i,& i\in \widehat I\setminus\widehat J=J\setminus (I\cap J),   \\
  -s_i,& i\in \widehat J\setminus\widehat I=I\setminus (I\cap J), \\
  0,& i\in \widehat{\widehat I\cup\widehat J}=I\cap J.
  \end{cases}
\end{equation}
It follows that 
\begin{equation}\label{Hw}
\begin{aligned}
&\cone\mathrm A_J\subset H^+_{\mb w},\quad 
\cone\mathrm A_I\subset H^-_{\mb w},\quad\text{and}\\
&\cone\mathrm A_{I\cap J}=H_{\mb w}\cap \cone\mathrm A_I= H_{\mb w}\cap\cone\mathrm A_J=\cone\mathrm A_I\cap
  \cone\mathrm A_J.
\end{aligned}
\end{equation}
Hence, $\cone\mathrm A_I$ and $\cone\mathrm A_J$ are separated by a hyperplane for each $I,J\in\mathcal C$, so (a) holds by Lemma~\ref{separation}.

(a)$\Rightarrow$(c).
Suppose that the intersection of $\cone\mathrm A_I$ and $\cone\mathrm A_J$ is a face of each cone. By Lemma~\ref{separation}, there exists a hyperplane~$H_{\mb w}$ satisfying~\eqref{Hw}. The vector $\mb w\in W$ satisfies~\eqref{wacases} for some positive~$r_i,s_i$. This $\mb w$ gives a linear relation~\eqref{intrel} between $\gamma_1,\ldots,\gamma_m$, which implies
$(\relint\cone\Gamma_{\widehat I})\cap(\relint\cone\Gamma_{\widehat J})\ne\varnothing$.
\end{proof}

A fan is \emph{simplicial} if each of its cones is generated by a linearly independent set of vectors. The fan $\Sigma$ defined by a fan data $\{\mathcal C,\mathrm A\}$ is simplicial if and only if $\mathcal C$ is a simplicial complex~$\mathcal K$. That is, a simplicial fan can be written as $\Sigma=\{\cone\mathrm A_I\colon I\in\mathcal K\}$, where $\mathcal K$ is a simplicial complex on~$[m]$ and $\mathrm A_I$ is linearly independent for any $I\in\sK$.  The simplicial complex $\sK$ is called the \emph{underlying complex} of a simplicial fan~$\Sigma$. 

A pair $\{\mathcal K,\mathrm A\}$ consisting of a simplicial complex $\mathcal K$ on $[m]$ and a spanning vector configuration $\mathrm A=\{\mb a_1,\ldots,\mb a_m\}$ in $W^*$ is called a \emph{simplicial fan data} or a 
\emph{triangulated configuration}~\cite{ba-za15} if $\Sigma=\{\cone\mathrm A_I\colon I\in\mathcal K\}$ is a fan in~$W^*$. This fan is automatically simplicial, since $\mathrm A_I$
is linear independent for any $I\in\sK$. (Indeed, if there is a linear dependency between $a_i$, $i\in I$, then there are $J,J'\subset I$ such that $\cone\mathrm A_J$ and $\cone\mathrm A_{J'}$ overlap.) 

A simplicial complex $\sK$ defines the \emph{coordinate fan} $\Sigma_\sK$ in $\R^m$ consisting of coordinate cones $\cone(\mb e_i\colon i\in I)$ with $I\in\sK$, where $\mb e_i$ is the $i$th coordinate vector. Then $\{\sK,\mathrm A\}$ is a triangulated configuration if and only if the restriction of the map $A\colon\R^m\to W^*$, $\mb e_i\mapsto\mb a_i$, 
to the union of cones of~$\Sigma_\sK$ is a bijection. In this case, $A(\Sigma_\sK)$ is precisely the fan $\Sigma$ defined by~$\{\mathcal K,\mathrm A\}$.

We do allow ghost vertices in $\sK$; they do not affect the fan~$\Sigma$. The vector $\mb a_i$ corresponding to a ghost vertex $\{i\}$ can be zero as it does not correspond to a one-dimensional cone of~$\Sigma$.

We restate Theorem~\ref{galefangen} in the case of simplicial fans for future reference.

\begin{theorem}\label{galefan}
Let $\sK$ be a simplicial complex on $[m]$, let $\mathrm A=\{\mb a_1,\ldots,\mb a_m\}$ be a vector configuration in
$W^*$, % such that for any simplex $I\in\sK$ the subset $\mathrm A_I$ is linearly independent, 
and let $\Gamma=\{\gamma_1,\ldots,\gamma_m\}$ be the Gale dual vector configuration. The following conditions are equivalent:

\begin{itemize}
\item[(a)] $\{\sK,\mathrm A\}$ is a triangulated configuration;\\[-0.7\baselineskip]

\item[(b)]
$(\relint\cone\mathrm A_I)\cap(\relint\cone\mathrm A_J)=\varnothing$ for any $I,J\in\sK$, $I\ne J$;\\[-0.7\baselineskip]

\item[(c)]
  $(\relint\cone\Gamma_{\widehat I})\cap(\relint\cone\Gamma_{\widehat J})\ne\varnothing$ for any $I,J\in\sK$.
\end{itemize}
\end{theorem}

Note that $\mathrm A_I$ is linearly independent for $I\in\sK$ in the case of simplicial fans, so
$\cone\Gamma_{\widehat I}$ is a full-dimensional cone in $V^*$ and its relative interior is the interior.

\section{Proper actions}\label{secproper}
A continuous map $f\colon X\to Y$ of topological spaces is \emph{proper} if $f^{-1}(C)$ is compact for any compact $C\subset Y$. If $X$ is Hausdorff and $Y$ is locally compact Hausdorff, then $f\colon X\to Y$ is proper if and only if it is \emph{universally closed}, where the latter means that for any map $Z\to Y$ the pullback $X\times_Y Z\to Z$ is a closed map. 

A continuous action $G\times X\to X$, $(g,x)\mapsto g\, x$, of a topological group $G$ on a locally compact Hausdorff space $X$ is called \emph{proper} if the \emph{shear map} 
\[
  h\colon G\times X\to X\times X, \quad (g,x)\mapsto (g\,x,x)
\] 
is proper. 
Equivalently, a $G$-action on a locally compact Hausdorff $X$ is proper if the subset 
$\{g\in G\colon gC\cap C\ne\varnothing\}$ is compact for any compact $C\subset X$.

Properness is a key property for noncompact group actions ensuring that the quotient is a well-behaved space. Namely, we have
\begin{itemize}
\item for a proper action of $G$ on $X$, the orbits are closed, the isotropy subgroups are compact, and the quotient $X/G$ is Hausdorff;

\item for a smooth, free and proper action of a Lie group $G$ on a smooth manifold~$M$, the quotient $M/G$ has a unique smooth structure with the property that the quotient map $M\to M/G$ is a smooth submersion.
\end{itemize}
We refer to~\cite[Chapter~21]{lee13} for the details.

%In our case there is the following combinatorial criterion of properness.

\begin{theorem}\label{proper}
Let $\Gamma=\{\gamma_1,\ldots,\gamma_m\}$ and $\mathrm A=\{\mb a_1,\ldots,\mb a_m\}$ be a pair of Gale dual vector configurations in $V^*$ and $W^*$, respectively, and let $\sK$ be a simplicial complex on~$[m]$. Then
\begin{itemize}
\item[(1)] the action $V\times U(\sK)\to U(\sK)$ obtained by restricting~\eqref{ract} to $U(\sK)$ is free if and only $\mathrm A_I$ is linearly independent for any $I\in\sK$;

\item[(2)] the action $V\times U(\sK)\to U(\sK)$ is proper if and only if 
$\{\sK,\mathrm A\}$ is a triangulated configuration, in which case the action is also free.
\end{itemize}
\end{theorem}
\begin{proof}
Statement (1) is Proposition~\ref{free}.

To prove (2), first assume that $\{\sK,\mathrm A\}$ is a triangulated configuration. We need to check that the preimage $h^{-1}(C)$ of a compact subset $C\subset U(\sK)\times U(\sK)$ under the shear map $h\colon V\times U(\sK)\to U(\sK)\times U(\sK)$ is compact. Since $V$ and $U(\sK)$ are metric spaces, it suffices to check that any infinite sequence $\{(\mb v^{(k)},\mb x^{(k)})\colon k=1,2,\ldots\}$ of points in $h^{-1}(C)$ contains a convergent subsequence. Since $C$ is compact, by passing to a subsequence we may assume that the sequence
\[
  \{h(\mb v^{(k)},\mb x^{(k)})\}=\{(\mb v^{(k)}\cdot \mb x^{(k)},\mb x^{(k)})\}
\]
has a limit in $U(\sK)\times U(\sK)$. We set $\mb y^{(k)}=\mb v^{(k)}\cdot\mb x^{(k)}$  
and let
\[
  \mb x=\lim_{k\to\infty}\mb x^{(k)},\quad \mb y=\lim_{k\to\infty}\mb y^{(k)},
\]
where $\mb x,\mb y\in U(\sK)$. We need to show that a subsequence of $\{\mb v^{(k)}\}$ has  a limit in~$V$.

By passing to a subsequence we may assume
that each numerical sequence $\{\langle\gamma_i,\mb v^{(k)}\rangle\}$, \ $i=1,\ldots,m$, has a
limit in $\R\cup{\pm\infty}$. Let
\[
  I_+=\bigl\{i\colon\lim_{k\to\infty}\langle\gamma_i,\mb v^{(k)}\rangle= +\infty\bigr\},
\quad
  I_-=\bigl\{i\colon\lim_{k\to\infty}\langle\gamma_i,\mb v^{(k)}\rangle= -\infty\bigr\}.
\]
As both sequences $\{\mb x^{(k)}\}$, $\{\mb y^{(k)}=\mb v^{(k)}\cdot\mb x^{(k)}\}$ converge, formula~\eqref{ract}  implies $x_i=0$ for $i\in I_+$ and $y_i=0$ for $i\in I_-$. Then it follows from the definition of $U(\sK)$~\eqref{UK} that $I_+$ and $I_-$ are disjoint simplices of~$\sK$. Now condition~(c) of Theorem~\ref{galefan} gives
\[
  (\relint\cone\Gamma_{\widehat I_+})\cap(\relint\cone\Gamma_{\widehat I_-})\ne\varnothing.
\]
Hence,
\[
  0=\sum_{i\in\widehat I_-}r_i\gamma_i  -\sum_{i\in\widehat I_+}s_i\gamma_i=
  \sum_{i\in I_+}r_i\gamma_i-\sum_{i\in I_-}s_i\gamma_i
  +\sum_{i\notin I_+\sqcup I_-}(r_i-s_i)\gamma_i
\]
for some positive $r_i$ and $s_i$. This implies that both $I_+$ and $I_-$ are empty, as otherwise
\[
  0=\lim_{k\to\infty}\Bigl(\sum_{i\in I_+}r_i\langle\gamma_i,\mb v^{(k)}\rangle
  -\sum_{i\in I_-}s_i\langle\gamma_i,\mb v^{(k)}\rangle+
  \sum_{i\notin I_+\sqcup I_-}(r_i-s_i)\langle\gamma_i,\mb v^{(k)}\rangle\Bigr)=+\infty,
\]
giving a contradiction. Therefore, each sequence $\{\langle\gamma_i,\mb v^{(k)}\rangle\}$, \ $i=1,\ldots,m$, has a finite limit. This implies that $\{\mb v^{(k)}\}$ converges in~$V$, because $\gamma_1,\ldots,\gamma_m$ span~$V^*$. Thus, the action $V\times U(\sK)\to U(\sK)$ is proper.

\smallskip

Now assume that the data $\{\sK,\mathrm A\}$ do not define a fan. 
Then condition (b) of Theorem~\ref{galefan} fails, and we have
\[
  (\relint\cone\mathrm A_I)\cap(\relint\cone\mathrm A_J)\ne\varnothing
\]
for some $I,J\in\sK$, $I\ne J$.
%In this case $I\cup J\notin\sK$, as otherwise both $\cone\mathrm A_I$ and $\cone\mathrm A_J$ are faces of $\cone\mathrm A_{I\cup J}$, and their relative interiors have empty intersection. 
It follows that
\[
  0=\sum_{i\in I}r_i\mb a_i  -\sum_{i\in J}s_i\mb a_i=
  \sum_{i\in I\setminus J}r_i\mb a_i-\sum_{i\in J\setminus I}s_i\mb a_i+
  \sum_{i\in I\cap J}(r_i-s_i)\mb a_i
\]
for some positive $r_i$ and $s_i$. By Gale duality, this linear relation between $\mb a_1,\ldots,\mb a_m$ gives a vector  $\mb v\in V$ satisfying
\[
  \langle\gamma_i,\mb v\rangle=\begin{cases}
  r_i-s_i,& i\in I\cap J,\\
  r_i,& i\in I\setminus J,   \\
  -s_i,& i\in J\setminus I, \\
  0,& i\notin I\cup J.
  \end{cases}
\]
Now consider the sequence of points $\{\mb x^{(k)}\}$ in $\R^m$ with coordinates
\[
  x^{(k)}_i=\begin{cases}0,&i\in I\cap J,\\
  e^{-k\langle\gamma_i,\mb v\rangle},&i\in I\setminus J,\\1,&i\notin I\end{cases}
  =\begin{cases}0,&i\in I\cap J,\\
  e^{-kr_i},&i\in I\setminus J,\\1,&i\notin I.\end{cases}
\]
Observe that $\lim_{k\to\infty}\mb x^{(k)}=\mb x$, where
\[
  x_i=\begin{cases}0,&i\in I,\\1,&i\notin I,\end{cases}
\]
and both $\mb x^{(k)}$ and $\mb x$ lie in~$U(\sK)$. Now define $\mb v^{(k)}=k\mb v$ and $\mb y^{(k)}=\mb v^{(k)}\cdot\mb x^{(k)}$, so that
\[
  y^{(k)}_i=e^{k\langle\gamma_i,\mb v\rangle}x^{(k)}_i=\begin{cases}0,&i\in I\cap J,\\
  1,&i\in I\setminus J,\\e^{k\langle\gamma_i,\mb v\rangle},&i\notin I\end{cases}
  =\begin{cases}0,&i\in I\cap J,\\
  e^{-ks_i},&i\in J\setminus I,\\1,&i\notin J.\end{cases}
\]
We have $\lim_{k\to\infty}\mb y^{(k)}=\mb y$, where
\[
  y_i=\begin{cases}0,&i\in J,\\1,&i\notin J,\end{cases}
\]
so both $\mb y^{(k)}$ and $\mb y$ lie in~$U(\sK)$. On the other hand, no subsequence of $\mb v^{(k)}$ converges in~$V$. Thus, the action $V\times U(\sK)\to U(\sK)$ is not proper.

Finally, if $\{\sK,\mathrm A\}$ is a triangulated configuration, then $\mathrm A_I$ is linearly independent for any $I\in\sK$, so a proper action $V\times U(\sK)\to U(\sK)$ is free by statement~(1).
\end{proof}

\begin{example}\label{1dimfan}
Consider the action of $V=\R$ on $\R^2$ given by 
\[
  (v,(x_1,x_2))\mapsto(e^{v}x_1,e^{v}x_2).
\]
We have $\Gamma=(\gamma_1,\gamma_2)=(1,1)$, $\mathrm A=(\mb a_1,\mb a_2)=(1,-1)$. Let $\sK=\sK(\Gamma)=\{\varnothing,\{1\},\{2\}\}$, so that $U(\sK)=\R^2\setminus\{\mathbf 0\}$.
The vectors $\mb a_1,\mb a_2$ form a one-dimensional fan in~$\R$, so the above action of $\R$ on $U(\sK)=\R^2\setminus\{\mathbf 0\}$ is proper and the quotient is homeomorphic to a circle (a smooth manifold).

Now consider the action of $V=\R$ on $\R^2$ given by 
\[
  (v,(x_1,x_2))\mapsto(e^{v}x_1,e^{-v}x_2).
\]
This time we have $\Gamma=(\gamma_1,\gamma_2)=(1,-1)$, $\mathrm A=(\mb a_1,\mb a_2)=(1,1)$.
The relative interiors of the cones generated by $\mb a_1$ and $\mb a_2$ have a non-empty intersection (they coincide), so the action of $\R$ on $U(\sK)=\R^2\setminus\{\mathbf 0\}$ is not proper and the quotient is a non-Hausdorff space.
\end{example}

\section{Completeness and compactness}\label{seccompl}
Assume given data $\{\sK,\mb a_1,\ldots,\mb a_m\}$ defining a simplicial fan 
$\Sigma=\{\sigma_1,\ldots,\sigma_s\}$ in $W^*$, as in Section~\ref{secfan}. A fan $\Sigma$ in $W^*$ is complete if the union of its cones $\sigma_i$ is the whole of~$W^*$. The underlying simplicial complex $\sK$ of a complete simplicial fan defines a simplicial subdivision (a \emph{triangulation}) of the unit sphere in~$W^*$.

Let $\Gamma=\{\gamma_1,\ldots,\gamma_m\}$ be the Gale dual configuration in~$V^*$. The restriction of the exponential action~\eqref{ract} to $U(\sK)$ is free and proper by Theorem~\ref{proper}.
We have the following criterion of compactness for the quotient of the action $V\times U(\sK)\to U(\sK)$.

\begin{theorem}\label{comcom}
Assume that $\{\mathcal K,\mathrm A\}$ is a triangulated configuration defining a simplicial fan~$\Sigma$.
The quotient smooth manifold $U(\sK)/V$ is compact if and only if the fan $\Sigma$ is complete.
\end{theorem}
\begin{proof}
Assume that $\Sigma$ is a complete fan.
Because $U(\sK)/V$ is metrisable, it is enough to prove that any sequence of points in $U(\sK)/V$ contains a convergent subsequence.

Take a sequence of points in $U(\sK)/V$ and let $\{\mb x^{(n)},n=1,2,\ldots\}$ be its lift to the total space, $\mb x^{(n)}=(x_1^{(n)} , \dots , x_m^{(n)}) \in U(\sK)$ . We shall change the sequence $\{\mb x^{(n)}\}$ to $\{\widetilde{\mb x}^{(n)}\}$ so that $\mb x^{(n)}$ and $\widetilde{\mb x}^{(n)}$ lie in the same $V$-orbit for each $n=1,2,\ldots$ and $\{\widetilde{\mb x}^{(n)}\}$ has a convergent  subsequence  in $U(\sK)$. Its projection will be the required subsequence in $U(\sK)/V$. 

By passing to a subsequence we may assume that  all elements of the sequence $\{\mb x^{(n)}\}$ have the same sets of zero coordinates. That is, there is $J\in\sK$ such that $x^{(n)}_j=0$ for $j\in J$ and $x^{(n)}_j\ne0$ for $j\notin J$, for all~$n$.

We argue by induction on $\dim W^*=m-k$. Let $\dim W^*=1$.  There is only one complete one-dimensional fan (corresponding to $\sK=\{\varnothing,\{1\},\{2\}\}$), so the theorem holds in this case (see Example~\ref{1dimfan}).

If $J\ne\varnothing$, then $(\mb v\cdot\mb x^{(n)})_j=0$ for any $j\in J$, $\mb v\in V$ and $n=1,2,\ldots$ Consider the link subcomplex $\mathop\mathrm{link}_{\sK}J=\{I\in\sK\colon I\cup J\in\sK,\,I\cap J=\varnothing\}$. Then the data 
$\{\mathop\mathrm{link}_{\sK}J,\mb a_j\colon j\notin J\}$ define a complete fan in the quotient space $W^*/\langle\mb a_j\colon j\in J\rangle$, which has dimension less than~$W^*$. Therefore, we can apply the induction assumption in this case.

We therefore may assume $J=\varnothing$, that is, $x_j^{(n)}\ne0$ for any $n$ and $j\in[m]$.

We denote $\R^\times=\R\setminus\{0\}$. Given $\mb x=(x_1,\ldots,x_m)\in(\R^\times)^m$, define 
\[
  \ell(\mb x)=\sum_{i=1}^m\mb a_i\log|x_i|\in W^*.
\]
%We think of $\ell(\mb x)$ as a linear function on $W$, whose value on $\mb w$ is 
%$\sum_{i=1}^m\langle\mb a_i,\mb w\rangle\log|x_i|$. 
For any $\mb v\in V$ we have
\begin{multline}\label{lconst}
\ell(\mb v\cdot\mb x)=\sum_{i=1}^m\mb a_i\log\bigl(e^{\langle\gamma_i,\mb v\rangle}|x_i|\bigr)
\\=\sum_{i=1}^m\mb a_i\langle\gamma_i,\mb v\rangle+\sum_{i=1}^m\mb a_i\log|x_i|=\sum_{i=1}^m\mb a_i\log|x_i|=\ell(\mb x),
\end{multline}
where the second-to-last equality follows by Gale duality.

Now we get back to the sequence $\{\mb x^{(n)}\}$ and consider $-\ell(\mb x^{(n)})\in W^*$. As $\Sigma$ is a complete fan in $W^*$, we can find $I\in\sK$ with $|I|=\dim W^*$ such that infinitely many terms of the sequence $\{-\ell(\mb x^{(n)})\}$ lie in $\cone\mathrm A_I$. Hence, by passing to a subsequence, we may assume that
$\{-\ell(\mb x^{(n)})\}\in\cone\mathrm A_I$ for $n=1,2,\ldots$.

By Proposition~\ref{galeindspan}, the vector configuration $\Gamma_{\widehat I}$ is a basis of $V^*$, so we can find $\mb v^{(n)}\in V$ satisfying
$$
   \langle \gamma_i ,\mb v^{(n)}\rangle = -\log|x_{i}^{(n)}|,\quad i \notin I.
$$
and set
\[
  \widetilde{\mb x}^{(n)}=\mb v^{(n)}\cdot{\mb x}^{(n)}.
\]
We have
\[
  |\widetilde x_i^{(n)}| = |e^{\langle \gamma_i,\mb v^{(n)}\rangle} x_i^{(n)}|=1,\quad i\notin I.
\]
This together with~\eqref{lconst} implies
\[
-\sum\limits_{i \in I}\mb a_i \log |\widetilde x_i^{(n)}| =-\ell(\widetilde{\mb x}^{(n)})=  -\ell(\mb x^{(n)}) \in \cone\mathrm A_I.
\]
It follows that $\log |\widetilde x_i^{(n)}|\le0$ for any $i\in[m]$ and $n=1,2,\ldots$. Therefore, 
$\{\widetilde{\mb x}^{(n)}\}$ is a bounded sequence, so its subsequence has a limit it $\R^m$. As $|\widetilde x_i^{(n)}|=1$ for $i\notin I$, the limit may have zero coordinates $x_i$ only for $i\in I$, so the limit belongs to $U(\sK)$, as needed.

\smallskip

Now assume that $U(\sK)/V$ is compact. Take $\mb a\in W^*$. To prove that $\Sigma$ is complete, we need to show that $\mb a$ belongs to a cone of~$\Sigma$. Write 
\begin{equation}\label{aalpha}
  \mb a=\sum_{i=1}^m\alpha_i\mb a_i, \quad \alpha_i\in\R.
\end{equation}
For any $x>0$ consider
\[
  (x^{\alpha_1},\ldots,x^{\alpha_m})\in (\R^\times)^m\subset U(\sK)
\]
and let $[x^{\alpha_1},\ldots,x^{\alpha_m}]$ denote the image in $U(\sK)/V$. Setting $x=x^{(n)}$, $n=1,2,\ldots$, where $x^{(n)}>0$ converges to zero as $n\to\infty$ and using the fact that $U(\sK)/V$ is compact, we observe that there exists a limit
%$[x^{\alpha_1},\ldots,x^{\alpha_m}]$ must have a limit in $U(\sK)/V$ as $x\to0$. Let
\[
  \lim_{x\to0}[x^{\alpha_1},\ldots,x^{\alpha_m}]=[y_1,\ldots,y_m]\in U(\sK)/V.
\]
Let 
\[
  I=\{i\in[m]\colon y_i=0\}\in\sK.
\]
This is well-defined as the set of zero coordinates of a point in $U(\sK)$ is constant along its $V$-orbit. Furthermore, $\mb y=(y_1,\ldots,y_m)\in U_I$, where $U_I\in U(\sK)$ is given by~\eqref{affui}.

For any $\mb w\in W$ and $\mb x= (x_1,\ldots,x_m)\in (\R^\times)^m$, define
\[
  f_{\mb w}(x_1,\ldots,x_m)=\prod_{i=1}^m|x_i|^{\langle\mb a_i,\mb w\rangle}.
\]
Then $f_{\mb w}$ defines a function $(\R^\times)^m/V\to\R$, because $f_{\mb w}(\mb x)=e^{\langle\ell(\mb x),\mb w\rangle}$ for $\mb x\in(\R^\times)^m$ and $\ell(\mb x)$ is constant on $V$-orbits (see above). Consider the dual cone
\[
  (\cone \mathrm A_I)^\mathsf{v}=
  \{\mb w\in W\colon\langle\mb a_i,\mb w\rangle\ge 0\text{ for }i\in I\}.
\]
If $\mb w\in(\cone \mathrm A_I)^\mathsf{v}$, then $f_{\mb w}$ extends to a continuous  function on~$U_I/V$.

Now $\lim_{x\to0}[x^{\alpha_1},\ldots,x^{\alpha_m}]=[\mb y]$ in $U_I/V$, implying that 
\[
  \lim_{x\to0}f_{\mb w}(x^{\alpha_1},\ldots,x^{\alpha_m})=f_{\mb w}(\mb y)
\]  
 for $\mb w\in(\cone \mathrm A_I)^\mathsf{v}$. On the other hand,
\[
  \lim_{x\to0}f_{\mb w}(x^{\alpha_1},\ldots,x^{\alpha_m})  
  =\lim_{x\to 0} x^{\langle\sum_{i=1}^m\alpha_i\mb a_i,\mb w\rangle}
  =\lim_{x\to 0} x^{\langle\mb a,\mb w\rangle}.
\]
The latter limit exists if and only if $\langle\mb a,\mb w\rangle\ge0$. Hence, $\langle\mb a,\mb w\rangle\ge0$ for any $\mb w\in(\cone \mathrm A_I)^\mathsf{v}$. This implies  that $\mb a\in\cone \mathrm A_I$, and therefore the fan $\Sigma$ is complete.
\end{proof}

The condition of completeness (or compactness of the quotient) can be formulated in terms of the triangulated configuration $(\mathcal K,\mathrm A)$ as follows.

\begin{proposition}\label{galecomp}
Assume that $\{\mathcal K,\mathrm A\}$ is a triangulated configuration defining a simplicial fan~$\Sigma$ in an $n$-dimensional space~$W^*$. Then, the $\Sigma$ is complete if and only if one of the following two equivalent conditions is satisfied:
\begin{itemize}
\item[(a)] $\sK$ contains an $n$-element subset (simplex), and each $(n-1)$-element simplex of $\sK$ is contained in two $n$-element simplices;
\item[(b)] the collection of subsets $\mathcal E=\{J\subset [m]\colon|J|=m-n,\widehat J\in\sK\}$ is nonempty, and for any $J\in\mathcal E$ and $i\in[m]$ there exists $j\in J$ such that $(J\setminus\{j\})\cup\{i\}\in\mathcal E$. 
\end{itemize}
\end{proposition}
\begin{proof}
The equivalence of (a) and (b) is clear.

Let $\Sigma$ be a complete simplicial fan. Then $\Sigma$ contains an $n$-dimensional cone, so there is $I\in\sK$ with $|I|=n$. Suppose there is an $(n-1)$-element subset of $\sK$ that is contained in at most one $n$-element subset. That is, there is an $(n-1)$-dimensional cone $\tau$ of $\Sigma$ that is a face of at most one $n$-dimensional cone of~$\Sigma$. Then, there is a hyperplane $H$ containing $\tau$ such that one of the halfspaces defined by $H$ contains points that do not lie in any cone of~$\Sigma$. A contradiction.

Now let condition (a) be satisfied. Suppose $\Sigma$ is not complete. Then $C=\bigcup_{\sigma\in\Sigma}\sigma$ is a closed proper subset of~$W^*$. Its topological boundary $\partial C$ contains an 
$(n-1)$-dimensional cone of~$\Sigma$, say~$\tau$. Then $\tau$ is a face of at most one $n$-dimensional cone of~$\Sigma$. A contradiction.
\end{proof}

Condition (b) in Proposition~\ref{galecomp} is the \emph{substitute existence (SE)} condition of Bosio~\cite{bosi01}. It features in the construction of LVMB manifolds, see Section~\ref{LVMBsec}.

\section{Real moment-angle manifolds}\label{rksec}

When the quotient $U(\sK)/V$ is compact, its topology can be described via the polyhedral product decomposition.

\begin{construction}[polyhedral product]\label{polpr}
Assume given a simplicial complex  $\sK$ on~$[m]$ and a sequence
\[
  (\mb X,\mb A)=\{(X_1,A_1),\ldots,(X_m,A_m)\}
\]
of $m$ pairs of pointed topological spaces. For each
subset $I\subset[m]$ we set
\begin{equation}\label{XAI}
  (\mb X,\mb A)^I=\bigl\{(x_1,\ldots,x_m)\in
  \prod_{k=1}^m X_k\colon\; x_k\in A_k\quad\text{for }k\notin I\bigl\}
\end{equation}
and define the \emph{polyhedral product} of $(\mb X,\mb A)$
corresponding to $\sK$ as
\[
  (\mb X,\mb A)^{\sK}=\bigcup_{I\in\mathcal K}(\mb X,\mb A)^I=
  \bigcup_{I\in\mathcal K}
  \Bigl(\prod_{i\in I}X_i\times\prod_{i\notin I}A_i\Bigl)\subset\prod_{i=1}^m X_i.
\]
In the case when all pairs $(X_i,A_i)$ are the same, i.\,e.
$X_i=X$ and $A_i=A$ for $i=1,\ldots,m$, we use the notation
$(X,A)^\sK$ for $(\mb X,\mb A)^\sK$. 
%Also, if each $A_i=\pt$, then we use the abbreviated notation $\mb X^\sK$ for $(\mb X,\pt)^\sK$,
%and $X^\sK$ for $(X,\pt)^\sK$.
\end{construction}

We have $(\mb X,\mb A)^{\sK}=\prod_{i=1}^m A_i$ when 
$\sK=\{\varnothing\}$, and
$(\mb X,\mb A)^{\sK}=\prod_{i=1}^m X_i$ when $\sK$ is the simplex 
$\varDelta[m]$ on $m$ vertices. Here are more examples.

\begin{example}\label{expp}\ 

\noindent {\bf 1.} It is easy to see that $U(\sK)=(\R,\R^\times)^\sK$, where $\R^\times=\R\setminus\{0\}$, see~\eqref{UK}.

\noindent {\bf 2.} Let $(X,A)=(D^1,S^0)$, where $D^1$ is a closed
interval $[-1,1]$ and $S^0$ is
its boundary, consisting of two points. The polyhedral product
$(D^1,S^0)^\sK$ is known as the \emph{real moment-angle
complex}~\cite[\S3.5]{bu-pa00},~\cite{bu-pa15}:
%\begin{equation}\label{rk}
\[
  \rk=(D^1,S^0)^\sK=\bigcup_{I\in\sK}(D^1,S^0)^I.
\]
%\end{equation}
It is a cubic subcomplex in the $m$-cube $(D^1)^m=[-1,1]^m$. When
$\sK$ consists of $m$ disjoint vertices, $\rk$ is the 1-dimensional
skeleton of~$[-1,1]^m$. When
$\sK=\partial\varDelta[m]$, we have $\rk=\partial[-1,1]^m$. In general, if $\{i_1,\ldots,i_k\}$ is a face of
$\sK$, then $\rk$ has $2^{m-k}$ cubic faces of dimension $k$
which lie in the $k$-dimensional planes parallel to the
$\{i_1,\ldots,i_k\}$th coordinate plane.
\end{example}

\begin{theorem}\label{quotrk}
Let $V\times U(\sK)\to U(\sK)$ be the exponential action defined by a complete simplicial fan $\Sigma=\{\sK,\mathrm A\}$. Then there is a homeomorphism 
\[
  U(\sK)/V\cong\rk.
\]
\end{theorem}
\begin{proof}
We have an inclusion of pairs $(D^1,S^0)\hookrightarrow(\R,\R^\times)$, which induces an inclusion 
$\rk=(D^1,S^0)^\sK\hookrightarrow(\R,\R^\times)^\sK=U(\sK)$. So the required homeomorphism will follow from the fact that the $V$-orbit of any point $\mb x=(x_1,\ldots,x_m)\in U(\sK)$ intersects $\rk$ at a single point. 

We may assume by induction that $x_i\ne0$ for $i\in[m]$, as in the proof of Theorem~\ref{comcom}. Consider $\ell(\mb x)=\sum_{i=1}^m\mb a_i\log|x_i|\in W^*$. As $\Sigma$ is complete, there is $I\in\sK$ with $|I|=\dim W^*$ such that $-\ell(\mb x)\in\cone\mathrm A_I$. Then $\Gamma_{\widehat I}=\{\gamma_i\colon i\notin I\}$ is a basis of~$V^*$. By solving the equations $\langle \gamma_i ,\mb v\rangle = -\log|x_i|$, $i \notin I$, we find $\widetilde{\mb x}=\mb v\cdot\mb x$ in the same $V$-orbit satisfying $|\widetilde x_i|=1$ for $i\notin I$. Since $-\ell(\widetilde{\mb x})=-\ell(\mb x)\in\cone\mathrm A_I$, we have $\log|\widetilde x_i|\le0$ for $i\in I$. Hence, $|\widetilde x_i|\le1$ for $i\in I$. These conditions imply that $\widetilde{\mb x}\in\rk$. Therefore, every $V$-orbit intersects~$\rk$.

Now suppose that $\mb x,\mb x'\in\rk$ are in the same $V$-orbit. By definition of $\rk$, there is $I\in\sK$ such that $|x_i|=1$ for $i\notin I$ and $|x_i|\le1$ for $i\in I$, and $J\in\sK$ such that 
$|x'_j|=1$ for $j\notin J$ and $|x'_j|\le1$ for $j\in J$. Since $\ell$ is constant on $V$-orbits by~\eqref{lconst}, we obtain
\[
  \sum_{i\in I}\mb a_i\log|x_i|=\ell(\mb x)=\ell(\mb x')=\sum_{j\in J}\mb a_j\log|x'_j|. 
\]
It follows that the expression above lies in $-\cone\mathrm A_I\cap\cone\mathrm A_J=-\cone\mathrm A_{I\cap J}$. As the latter is a simplicial cone, we obtain $|x_i|=|x'_i|$ for $i\in[m]$. Since the $V$-action preserves the sign of coordinates, we actually have $\mb x=\mb x'$, as needed.
\end{proof}

\begin{corollary}
Let $\sK$ be the underlying complex of a complete simplicial fan. Then
\begin{itemize}
\item[(1)] the real moment-angle manifold $\rk$ has a structure of a smooth compact manifold;
\item[(2)] $U(\sK)$ and $\rk$ have the same homotopy type.
\end{itemize}
\end{corollary}

In~\cite[Theorem~4.7.5]{bu-pa15}, it is proved that there is a homotopy equivalence $U(\sK)\simeq\rk$ for arbitrary~$\sK$. By virtue of Theorem~\ref{quotrk}, this homotopy equivalence can be realised as the quotient by a free action of  $V\cong\R^k$ when $\sK$ is the underlying complex of a complete simplicial fan.

\section{Polyhedra, normal fans, and intersections of quadrics}\label{normsec}
As we have seen in the previous sections, the quotient $U(\sK)/V$ is a smooth compact manifold if and only if $\{\sK,\mathrm A\}$ is a complete simplicial fan data, see Theorems~\ref{galefan},~\ref{proper} and~\ref{comcom}. There is an important class of complete fans arising from convex polytopes, as described next.

\subsection{Normal fans}
A \emph{convex polyhedron} in $W$ is defined as the nonempty intersection of finitely many halfspaces:
\begin{equation}\label{ptope}
  P=\{\mb w\in W\colon\langle \mb a_i,\mb w\rangle+b_i\ge0,\quad i=1,\ldots,m\},
\end{equation}
where $\mb a_i\in W^*$ and $b_i\in\R$. 

By referring to a polyhedron $P$ we imply a particular system of linear inequalities, not only the set of its common solutions. An inequality $\langle\mb a_i,\mb w\rangle+b_i\ge0$ in~\eqref{ptope} is 
\emph{redundant} if removing it from the system does not change the set of common solutions; otherwise the inequality is \emph{irredundant}.

A \emph{face} of~$P$ is the intersection of $P$ with a supporting hyperplane. Every face is specified by turning some of the defining inequalities into equalities. A face is therefore determined by a subset $I\subset[m]$ and is given by
\[
  F_I=\{\mb w\in P\colon \langle\mb a_i,\mb w\rangle+b_i=0 \quad\text{for }i\in I\}.
\]
Note that $F_\varnothing=P$. A \emph{vertex} is a $0$-dimensional face, and a \emph{facet} is a face of dimension one less the dimension of~$P$. We denote $F_{\{i\}}$ by~$F_i$.

A polyhedron $P$ is \emph{full-dimensional} if the affine hull of $P$ is the whole of~$W$.
If $P$ is full-dimensional and without redundant inequalities, then each $F_{i}$ is a facet with normal $\mb a_i\in W^*$, so there are $m$ facets in total.

A bounded polyhedron is called a \emph{polytope}.

Given a face $Q$ of $P$, the \emph{polyhedral angle at $Q$} is the cone generated by all vectors ${\mb w'-\mb w}$ pointing from $\mb w\in Q$ to $\mb w'\in P$. The cone dual to the polyhedral angle at~$Q$ is given by
\[
  \sigma_Q=\{\mb a\in W^*\colon
  \langle\mb a,\mb w\rangle\le\langle\mb a,\mb w'\rangle
  \text{ for all $\mb w\in Q$ and $\mb w'\in P$}\}.
\]
and is called the \emph{normal cone} of~$Q$. 
If we write $Q=F_I$, where $I\subset[m]$ is maximal among all subsets defining~$Q$, then
\begin{equation}\label{sigmaQAI}
  \sigma_Q=\cone\mathrm A_I=\cone(\mb a_i\colon 
  \langle\mb a_i,\mb w\rangle+b_i=0\text{ for }\mb w\in Q).
\end{equation}
If $P$ is full-dimensional, then $\sigma_Q$ is a strongly convex cone; it is generated by the normals to the facets containing~$Q$.  Then
\[
  \Sigma_P=\{\sigma_Q\colon Q\text{ is a face of }P\}
\]
is a fan in $W^*$, referred to as the \emph{normal fan} of the polyhedron~$P$. If $P$ is not full-dimensional, then the cones $\sigma_Q$ are not strictly convex, and $\Sigma_P$ is a \emph{generalised fan}~\cite[Definition~6.2.2]{c-l-s11}

The support $|\Sigma_P|$ of a normal fan is convex (it coincides with the dual of the \emph{recession cone} of~$P$, see~\cite[Theorem~7.1.6]{c-l-s11}). In particular, $\Sigma_P$ is complete if and only if $P$ is a polytope.

\medskip

%\begin{construction}[generic polyhedron]
A polyhedron~\eqref{ptope} is \emph{generic} if the hyperplanes defined by the equations $\langle\mb a_i,\mb w\rangle+b_i=0$ are
in general position at every point of~$P$ (that is, for any $\mb
w\in P$ the normals $\mb a_i$ of the hyperplanes containing
$\mb w$ are linearly independent). A generic polyhedron is full-dimensional. A generic polyhedron may have redundant inequalities, but,
for any such inequality, the intersection of the corresponding
hyperplane with $P$ is empty. For a generic $P$, each $F_i$ is either a facet or empty. A generic polytope $P$ is \emph{simple}, that is, exactly $n=\dim W$ facets meet at each vertex of~$P$.

The normal fan $\Sigma_P$ is simplicial if and only if $P$ is a generic polyhedron. In this case, the cones of $\Sigma_P$ are generated by those sets $\{\mb a_{i_1},\ldots,\mb
a_{i_k}\}$ for which the intersection $F_{i_1}\cap\cdots\cap
F_{i_k}$ is nonempty. Define the \emph{dual simplicial complex}
\[
  \mathcal K_P=\bigl\{I=\{i_1,\ldots,i_k\}\colon F_{i_1}\cap\cdots\cap
  F_{i_k}\ne\varnothing \bigr\}.
\]
Then the simplicial fan $\Sigma_P$ is defined by the data $\{\mathcal K_P;\mb a_1,\ldots,\mb a_m\}$. The redundant inequalities in a generic polytope~\eqref{ptope} correspond to the ghost vertices of~$\mathcal K_P$. 
%\end{construction}

%Assume now that $\mathbf 0$ is contained
%in the interior of $P$, which is equivalent to the condition $b_i>0$ for $i=1,\ldots,m$  %in~\eqref{ptope}. Then $\Sigma_P$ may be also described as
%the set of cones over the faces of the polar polytope
%\[
%  P^*=\{\mb a\in W^*\colon\langle\mb a,\mb w\rangle+1\ge0\:\text{for all}\:\mb w\in P\}.
%\]
%Furthermore, if $b_i>0$ for $i=1,\ldots,m$, then $P^*=\mathop\mathrm{conv}(\mb a_1,\ldots,\mb %a_m)$.

Our next goal is to describe the normal fan $\Sigma_P$ in terms of the Gale dual configuration $\Gamma=\{\gamma_1,\ldots,\gamma_m\}$. %To do this we must assume that $\mathrm A=\{\mb a_1,\ldots,\mb a_m\}$ spans the whole of~$W$, which implies that $P$ contains a vertex and does not contain a line.
In what follows, whenever we refer to the normal fan and Gale duality, we assume that the polyhedron~\eqref{ptope} is full-dimensional (so that $\Sigma_P$ is a fan) and $\mathrm A=\{\mb a_1,\ldots,\mb a_m\}$ spans~$W$ (so that the Gale dual configuration is defined).

\begin{lemma}\label{normal1}
The normal fan $\Sigma_P$ of a polyhedron~\eqref{ptope} consists of the cones 
$\cone\mathrm A_I$ such that
$\sum_{i=1}^m b_i\gamma_i\in\relint\cone\Gamma_{\widehat I}$.
\end{lemma}
\begin{proof}
Take $\sigma_Q\in\Sigma_P$ and write $Q=F_I$, where $I\subset[m]$ is maximal among all subsets defining the face~$Q$. Then $\sigma_Q=\cone\mathrm A_I$, see~\eqref{sigmaQAI}. Take $\mb w\in\relint Q$. We have $\langle\mb a_i,\mb w\rangle+b_i=0$ for $i\in I$ and $r_i:=\langle\mb a_i,\mb w\rangle+b_i>0$ for $i\in\widehat I$. Hence,
\[
  \sum_{i=1}^m b_i\gamma_i=\sum_{i=1}^m(\langle\mb a_i,\mb w\rangle+b_i)\gamma_i=
  \sum_{i\in\widehat I}r_i\gamma_i.
\]
It follows that $\sum_{i=1}^m b_i\gamma_i\in\relint\cone\Gamma_{\widehat I}$, by Proposition~\ref{positive}.

Now consider $\cone\mathrm A_I$ such that $\sum_{i=1}^m b_i\gamma_i\in\relint\cone\Gamma_{\widehat I}$. Then $\sum_{i=1}^m b_i\gamma_i=\sum_{i\in\widehat I}r_i\gamma_i$ for some positive $r_i$, $i\in\widehat I$. Denote $\mb b=(b_1,\ldots,b_m)\in\R^m$ and $\mb r=(r_1,\ldots,r_m)\in\R^m$, where $r_i=0$ for $i\in I$. Then $\varGamma(\mb r-\mb b)=0$. Hence, $\mb r-\mb b$ is in the image of $A^*$, see~\eqref{1exseq}. In other words, $\langle\mb a_i,\mb w\rangle=r_i-b_i$ for some $\mb w\in W$. This $\mb w$ satisfies $\langle\mb a_i,\mb w\rangle+b_i=0$ for $i\in I$ and  $\langle\mb a_i,\mb w\rangle+b_i>0$ for $i\in\widehat I$. Hence, $Q=F_I$ is a face of~$P$, and $\mb w\in\relint Q$. The normal cone $\sigma_Q$ is precisely $\cone\mathrm A_I$.
\end{proof}

%\begin{remark}
%If $P$ is not full-dimensional, then $\mathrm A=\{\mb a_1,\ldots,\mb a_m\}$ does not span $W^*$ and $\Sigma_P$ is a generalised fan. However, Lemma~\ref{normal1} holds in this setting, with the same proof.
%\end{remark}

\begin{proposition}\label{normal2}
Let $\mathrm A=\{\mb a_1,\ldots,\mb a_m\}$ and $\Gamma=\{\gamma_1,\ldots,\gamma_m\}$ be a pair of Gale dual vector configurations. Take
$\delta\in\relint\cone\Gamma$, let 
\[
  \mathcal C=\{I\subset[m]\colon\delta\in\relint\cone\Gamma_{\widehat I}\}
\]
and write $\delta=\sum_{i=1}^mb_i\gamma_i$. Then $\Sigma=\{\cone \mathrm A_I\colon I\in\mathcal C\}$ is the normal fan of the polyhedron $P$ given by~\eqref{ptope}.
\end{proposition}
\begin{proof}
Since $\delta\in\relint\cone\Gamma$, we may assume that all $b_i>0$ in the expression
$\delta=\sum_{i=1}^mb_i\gamma_i$. Then~\eqref{ptope} implies that $\mathbf 0$ is in the interior of~$P$, so $P$ is full-dimensional. The rest follows from Lemma~\ref{normal1}.
\end{proof}

Now we can state a criterion for a fan to be the normal fan of a polyhedron. It is originally due to Shephard~\cite{shep71}.

\begin{theorem}\label{galefannormgen}
Let $\mathrm A=\{\mb a_1,\ldots,\mb a_m\}$ and $\Gamma=\{\gamma_1,\ldots,\gamma_m\}$ be a pair of Gale dual vector configurations. Let $\Sigma=\{\cone \mathrm A_I\colon I\in\mathcal C\}$ be a fan with convex support (respectively, a complete fan). Then $\Sigma$ is the normal fan of a polyhedron (respectively, the normal fan of a polytope) if and only if
\[
  \bigcap_{I\in\mathcal C}\relint\cone\Gamma_{\widehat I}\ne\varnothing.
\]  
\end{theorem}
\begin{proof}
Assume that $\Sigma=\{\cone \mathrm A_I\colon I\in\mathcal C\}$ is the normal fan of a polyhedron $P$ given by~\eqref{ptope}. Set $\delta=\sum_{i=1}^m b_i\gamma_i$. Then $\delta\in\bigcap_{I\in\mathcal C}\relint\cone\Gamma_{\widehat I}$ by Lemma~\ref{normal1}.

Conversely, let $\bigcap_{I\in\mathcal C}\relint\cone\Gamma_{\widehat I}\ne\varnothing$.
Take $\delta\in\bigcap_{I\in\mathcal C}\relint\cone\Gamma_{\widehat I}$. Write
$\delta=\sum_{i=1}^m b_i\gamma_i$ for some $b_i\in\R$, $i=1,\ldots,m$, and define a polyhedron $P$ by~\eqref{ptope}. Then each $\cone \mathrm A_I$ with $I\in\mathcal C$ belongs to the normal fan~$\Sigma_P$ by Lemma~\ref{normal1}, so that $\Sigma$  is a subfan of~$\Sigma_P$. If $\Sigma$ is a complete fan, then this implies that $\Sigma=\Sigma_P$ and we are done. If $\Sigma$ has convex support, then it is easy to see that a subfan of $\Sigma_P$ with convex support is the normal fan of a polyhedron~$P'$. Indeed, define
\[
    P'=\{\mb w\in W\colon\langle \mb a_i,\mb w\rangle+b_i\ge0\quad\text{for } 
    \{i\}\in\mathcal C\}.
\]
Then $P\subset P'$ and $\Sigma_{P'}$ is a subfan of~$\Sigma_P$. For $\{j\}\in\mathcal C$, we have that $F_j=\{\mb w\in P'\colon \langle\mb a_j,\mb w\rangle+b_j=0\}$ is either a facet of $P'$, or the corresponding vector $\mb a_j$ is in the cone generated by $\mb a_i$, $i\in\mathcal C$, $i\ne j$. It follows that
\[
  |\Sigma_{P'}|=\cone(\mb a_i\colon\{i\}\in\mathcal C)=|\Sigma|,
\] 
where the second identity holds because $\Sigma$ has convex support. Since both 
$\Sigma_{P'}$ and $\Sigma$ are subfans of~$\Sigma_P$, we conclude that 
$\Sigma_{P'}=\Sigma$.
%
%Suppose $\Sigma_P$ also contains some $\cone \mathrm A_J$ with $J\notin\mathcal C$. Then, since $\Sigma$ has convex support,  there is  $\cone \mathrm A_I$ with $I\in\mathcal C$ such that 
%\[
%  (\relint\cone\mathrm A_I)\cap(\relint\cone\mathrm A_J)\ne\varnothing.
%\]
%Then the equivalence of (b) and (c) in Theorem~\ref{galefangen} implies that
%\[
%  (\relint\cone\Gamma_{\widehat I})\cap(\relint\cone\Gamma_{\widehat J})=\varnothing,
%\]
%so that $\delta\notin\relint\cone\Gamma_{\widehat J}$. This contradicts Lemma~\ref{normal1}. Thus, $\Sigma=\Sigma_P$.
\end{proof}

We see how the condition for a normal fan is stronger than the condition for a fan: in the former case, the relative interiors of the cones $\cone\Gamma_{\widehat I}$, $I\in\mathcal C$, must have a common point, while in the latter case only pairwise intersections must be nontrivial.

A choice of $\delta\in \bigcap_{I\in\mathcal C}\relint\cone\Gamma_{\widehat I}$ in Theorem~\ref{galefannormgen} defines a polyhedron $P$ (whose normal fan is $\Sigma$) up to a translation in $W$: one needs to express $\delta$ as a linear combination $\sum_{i=1}^mb_i\gamma_i$ to get a polyhedron~\eqref{ptope}, and different expressions of the same $\delta$ produce polytopes which are different by a translation. Choosing a different $\delta'\in \bigcap_{I\in\mathcal C}\relint\cone\Gamma_{\widehat I}$ results in a polyhedron $P'$ with the same normal fan $\Sigma_P$; such polyhedra $P$ and $P'$ are called \emph{analogous}.

\subsection{From polyhedra to quadrics}\label{secqua}
Now we turn to the quotient (leaf space) $U(\sK)/V$. As we have seen in Theorem~\ref{quotrk}, the quotient $U(\sK)/V$ is identified with the real moment-angle manifold $\rk$ when the data $\{\sK,\mb a_1,\ldots,\mb a_m\}$ define a complete simplicial fan~$\Sigma$.  When $\Sigma$ is the normal fan of a polytope, there is another realisation of $U(\sK)/V$, which has an intrinsic smooth structure as a nondegenerate intersection of real quadratic surfaces.

Consider the smooth map
\[
  \mu_\Gamma\colon\R^m\to V^*, \quad (x_1,\ldots,x_m)\mapsto x^2_1\gamma_1+ 
  \cdots+x^2_m\gamma_m.
\]

\begin{theorem}\label{quadrics}
Let 
\[
  P=\{\mb w\in W\colon\langle \mb a_i,\mb w\rangle+b_i\ge0,\quad i=1,\ldots,m\}
\]  
be a generic polyhedron with normal simplicial fan $\Sigma_P=\{\cone \mathrm A_I\colon I\in\mathcal K_P\}$. Let $\Gamma=\{\gamma_1,\ldots,\gamma_m\}$ be the Gale dual configuration of $\mathrm A=\{\mb a_1,\ldots,\mb a_m\}$. Put $\delta=\sum_{i=1}^mb_i\gamma_i$. Then 
\begin{itemize}
\item[(1)] $\delta\in\bigcap_{I\in\mathcal K_P}\relint\cone\Gamma_{\widehat I}$;

\item[(2)] $\mu_\Gamma^{-1}(\delta)\subset U(\sK_P)$;

\smallskip

\item[(3)] $\delta$ is a regular value of~$\mu_\Gamma$;

\smallskip

\item[(4)] the action of $V$ on $U(\sK_P)$ is free and proper, and the quotient $U(\sK_P)/V$ is diffeomorphic to the level set $\mu_\Gamma^{-1}(\delta)$, where the latter is given as a nondegenerate intersection of $k$ quadratic surfaces in~$\R^m$:
\[
  U(\sK_P)/V\cong\{(x_1,\ldots,x_m)\in\R^m\colon\gamma_1 x^2_1+ 
  \cdots+\gamma_mx^2_m=\delta\}.
\]
\end{itemize}
\end{theorem}

This result has well-known interpretations in the theory of Hamiltonian torus actions in symplectic geometry and in the theory of toric varieties (see~\cite[Proposition~3.1.1]{audi91} and~\cite[Appendix~1, Theorems~1.2 and~1.4]{guil94}, and also~\S\ref{parsr} below). However, it is purely topological in nature.

\begin{proof}[Proof of Theorem~\ref{quadrics}] Statement (1) follows from Lemma~\ref{normal1}. 

To prove (2), take $\mb x=(x_1,\ldots,x_m)\in\mu_\Gamma^{-1}(\delta)$, so that $\gamma_1 x^2_1+ 
\cdots+\gamma_mx^2_m=\delta$. Let $I_{\mb x}=\{i\colon x_i=0\}\subset[m]$. Then
\[
  \delta=\sum_{i\in\widehat I_{\mb x}}x_i^2\gamma_i\subset
  \relint\cone\Gamma_{\widehat I_{\mb x}}.
\]  
This implies, by Lemma~\ref{normal1}, that $\cone\mathrm A_{I_{\mb x}}\in\Sigma_P$. Hence, $I_{\mb x}\in\sK_P$ and $\mb x\in U(\sK_P)$ by definition of $U(\sK_P)$, see~\eqref{UK}.

\smallskip

To prove (3), observe that the differential of $\mu_\Gamma$ at a point $\mb x$ is given by
\[
  (d\mu_\Gamma)_{\mb x}\colon\R^m\to V^*,\quad
  (y_1,\ldots,y_m)\mapsto 2x_1\gamma_1y_1+\cdots+2x_m\gamma_my_m.
\]
Now for any $\mb x\in\mu_\Gamma^{-1}(\delta)$ we have $I_{\mb x}\in\sK_P$. The differential at such $\mb x$ is given by 
\[
  (y_1,\ldots,y_m)\mapsto\sum_{i\in\widehat I_{\mb x}}2x_i\gamma_iy_i.
\]
Since $P$ is generic, $\cone\mathrm A_{I_{\mb x}}$ is a simplicial cone. Hence, the set $\mathrm A_{I_{\mb x}}$ is linearly independent. By Proposition~\ref{galeindspan}, the set $\Gamma_{\widehat I_{\mb x}}$ spans~$V^*$. Therefore, the set 
$\{x_i\gamma_i\colon i\in\widehat I_{\mb x} \}$ also spans~$V^*$, which implies that the differential $(d\mu_\Gamma)_{\mb x}$ is surjective at any $\mb x\in\mu_\Gamma^{-1}(\delta)$. Thus, $\delta$ is a regular value of~$\mu_\Gamma$.

\smallskip

Since $\delta$ is a regular value, the intersection of quadrics in~(4) is nondegenerate, 
i.\,e. the $k$ gradients of the quadratic surfaces are linearly independent at each point of $\mu_\Gamma^{-1}(\delta)$. In particular, $\mu_\Gamma^{-1}(\delta)$ is a smooth submanifold of dimension $n$ in~$\R^m$. 

To prove the rest of~(4), we need to verify that each $V$-orbit of $U(\sK_P)$ intersects $\mu_\Gamma^{-1}(\delta)$ in one and only one point. Since the $V$-action on $U(\sK_P)$ is free, restricting the map $\mu_\Gamma$ to an orbit $V\cdot\mb x$ gives a map
\begin{equation}\label{Lmap}
  L\colon V\to V^*,\quad
  \mb v\mapsto e^{2\langle\gamma_1,\mb v\rangle}x_1^2\gamma_1+\cdots
  +e^{2\langle\gamma_m,\mb v\rangle}x_m^2\gamma_m.
\end{equation}
By Lemma~\ref{image} below, $L$ maps the orbit diffeomorphically onto $\relint\cone\Gamma_{\widehat I_{\mb x}}$. We have $\delta\in\relint\cone\Gamma_{\widehat I_{\mb x}}$, since $I_{\mb x}\in\sK_P$. Hence, $\delta$ is in the image of $L$, so that the orbit $V\cdot\mb x$ intersects $\mu_\Gamma^{-1}(\delta)$. Finally, since $L$ is injective, there is only one point of intersection.
%We first prove the latter. Suppose $\mu_\Gamma(\mb x)=\mu_\Gamma(\mb x')=\delta$, where $\mb x'=\mb v\cdot\mb x$,  $\mb v\in V$. Then
%\[
%  \sum_{i=1}^m\gamma_i x_i^2=\sum_{i=1}^m\gamma_i(x'_i)^2=
%  \sum_{i=1}^m\gamma_ie^{2\langle\gamma_i,\mb v\rangle}x_i^2.
%\]
%Consider the function $h(t)=\frac12\sum_{i=1}^me^{2\langle\gamma_i,t\mb v\rangle}%x_i^2$. Then
%\begin{align*}
%  h'(t)&=\sum_i\langle\gamma_i,\mb v\rangle e^{2\langle\gamma_i,t
%  \mb v\rangle}x_i^2,\\
%  h'(0)&=\sum_i\langle\gamma_i,\mb v\rangle x_i^2=
%  \Bigl\langle\sum_i\gamma_ix_i^2,\mb v\Bigr\rangle,\\
%  h'(1)&=\sum_i\langle\gamma_i,\mb v\rangle e^{2\langle\gamma_i,
%  \mb v\rangle}x_i^2,
%\end{align*}
\end{proof}

It remains to prove the key fact about the image of~$L$. The argument below is based on the properties of the Legendre transform and is taken from~\cite[Appendix~1]{guil94}.

\begin{lemma}
\label{image}
The restriction~\eqref{Lmap} of the map $\mu_\Gamma$ to the $V$-orbit of a point $\mb x\in U(\sK_P)$ maps this orbit diffeomorphically onto $\relint\cone(\gamma_i\colon x_i\ne0)$.
\end{lemma}
\begin{proof}
Consider the function $\frac12\|\mb x\|^2$, where 
$\|\mb x\|^2=x_1^2+\cdots+x_m^2$, and restrict it to the $V$-orbit of a given $\mb x=(x_1,\ldots,x_m)$. We obtain the function
\begin{equation}\label{fsum}
  f\colon V\to\R,\quad f(\mb v)={\textstyle\frac12}\bigl(e^{2\langle\gamma_1,\mb 
  v\rangle}x_1^2+\cdots +e^{2\langle\gamma_m,\mb v\rangle}x_m^2\bigr).
\end{equation}
Then $(df)_{\mb v}=L(\mb v)$, i.\,e. $L$ is the (multidimensional) \emph{Legendre transform} of~$f$. The second differential
\[
  (d^2f)_{\mb v}=2\bigl(e^{2\langle\gamma_1,\mb v\rangle}x_1^2
  \gamma_1\otimes\gamma_1+
  \cdots +e^{2\langle\gamma_m,\mb v\rangle}   
  x_m^2\gamma_m\otimes\gamma_m\bigr)
\]
is positive definite, since $\{\gamma_i\colon x_i\ne0\}$ generate~$V^*$. This implies that $f$ is strictly convex, and also that $L$ is a local diffeomorphism (its differential is nondegenerate).

Given $\delta\in V^*$, consider the function $f_\delta(\mb v)=f(\mb v)-\langle\delta,\mb v\rangle$. This function is also strictly convex, because $d^2f_\delta=d^2f$. For a strictly convex function $f$ the following three conditions are equivalent:
\begin{itemize}
\item[(a)] $(df)_{\mb v}=0$ at some point $\mb v\in V$;
\item[(b)] $f$ has a unique local minimum;
\item[(c)] $f(\mb v)$ tends to $+\infty$ as $\mb v$ tends to infinity in~$V$.
\end{itemize}
Now $(df_\delta)_{\mb v}=(df)_{\mb v}-\delta=L(\mb v)-\delta$, so $f_\delta$ has a critical point if and only if $\delta$ is in the image of~$L$. Since a critical point of $f_\delta$ us unique, the map $L$ is injective, so it is a diffeomorphism onto its image. 

It remains to identify the image of~$L$. Clearly, $\Im L\subset\R_>\langle\gamma_i\colon x_i\ne 0\rangle=\relint\cone(\gamma_i\colon x_i\ne0)$. The opposite inclusion can be proved directly in the case when $I_{\mb x}$ is largest possible, that is $|I_{\mb x}|=n$. Then $\{\gamma_i\colon x_i\ne 0\}$ is a basis of~$V^*$. Take $\delta\in\R_>\langle\gamma_i\colon x_i\ne 0\rangle$ and write 
$\delta=\sum_{i\in\widehat I_{\mb x}}r_i\gamma_i$, $r_i>0$. The equation $L(\mb v)=\delta$ takes the form
\[
  \sum_{i\in\widehat I_{\mb x}}e^{2\langle\gamma_i,\mb v\rangle} x_i^2\gamma_i
  =\sum_{i\in\widehat I_{\mb x}}r_i\gamma_i\quad\Leftrightarrow\quad
  \langle\gamma_i,\mb v\rangle=\log\frac{\sqrt r_i}{|x_i|},\quad i\in\widehat I_{\mb x}.
\]
The latter system of equations has a (unique) solution $\mb v$, as needed.

In the general case, the function~\eqref{fsum} can be written as the sum of functions corresponding to $\mb x$ with $|I_{\mb x}|=n$. So the result follows from the following fact:

\smallskip

\noindent\textit{Let $L_i$ be the Legendre transform of a strictly convex function~$f_i$, $i=1,2$. Then 
$\Im(L_1+L_2)=\Im L_1+\Im L_2$, where the latter is the Minkowski sum of two convex sets in~$V^*$.}

\smallskip

\noindent Indeed, the inclusion $\Im(L_1+L_2)\subset\Im L_1+\Im L_2$ is clear. For the opposite inclusion, take $\delta_1+\delta_2\in\Im L_1+\Im L_2$. Then $\delta_i=L(\mb v_i)$, so that $\mb v_i$ is the unique critical point of~$(f_i)_{\delta_i}$. Hence, both 
$(f_1)_{\delta_1}$ and $(f_2)_{\delta_2}$ satisfy property~(c) above. Then their sum 
$(f_1)_{\delta_1}+(f_2)_{\delta_2}=(f_1+f_2)_{\delta_1+\delta_2}$ is strictly convex and also satisfies~(c). Therefore, $(f_1+f_2)_{\delta_1+\delta_2}$ has a (unique) critical point, so $\delta_1+\delta_2$ is in the image of $L_1+L_2$, as claimed.
\end{proof}

\begin{corollary}\label{rkquad}
Suppose that $\sK$ is the underlying simplicial complex of the normal fan of a generic polytope~\eqref{ptope}. Then the real moment-angle manifold $\rk=(D^1,S^0)$ is homeomorphic to the
intersection of quadrics given by
\[
  \{(x_1,\ldots,x_m)\in\R^m\colon\gamma_1 x^2_1+ 
  \cdots+\gamma_mx^2_m=\delta\},
\] 
where $\Gamma=(\gamma_1,\ldots,\gamma_m)$ is the Gale dual configuration to $\mathrm A=\{\mb a_1,\ldots,\mb a_m\}$ and $\delta=\sum_{i=1}^m b_i\gamma_i$.
\end{corollary}

\begin{proof}
This follows by combining Theorem~\ref{quotrk} with Theorem~\ref{quadrics}~(4), as both 
$\rk$ and the intersection of quadrics are homeomorphic to the quotient $U(\sK)/V$ (see also~\cite[Theorem~6.2.4]{bu-pa15} for a direct argument). Note that the fan is assumed to be complete in Theorem~\ref{quotrk}, so we require that~\eqref{ptope} is bounded, i.\,e. a polytope. 
\end{proof}

\subsection{From quadrics to polyhedra}
The input of Theorem~\ref{quadrics} is a generic polyhedron~$P$, that is, a vector configuration $\mathrm A$ in $W^*$ and a choice of generic parameters $b_1,\ldots,b_m$. There is a Gale dual version of this result, in which the input is a vector configuration $\Gamma$ in $V^*$ and a choice of generic $\delta\in V^*$:

\begin{theorem}\label{quadrics1}
Let $\Gamma=\{\gamma_1,\ldots,\gamma_m\}$ be a spanning vector configuration in~$V^*\cong\R^{k}$. Take $\delta\in V^*$ satisfying the two conditions:
\begin{itemize}
\item[(a)] $\delta\in\cone\Gamma$;\\[-0.7\baselineskip]

\item[(b)]
if $\delta\in\cone\Gamma_J$, then $|J|\ge k$,
\end{itemize}
and let
\[
  \sK=\{I\in[m]\colon \delta\in\cone\Gamma_{\widehat I}\}.
\]
Then the action~\eqref{ract} of $V$ on $U(\sK)$ is free and proper, and the quotient $U(\sK)/V$ is diffeomorphic to the nondegenerate intersection of $k$ quadratic surfaces in~$\R^m$:
\[
  U(\sK)/V\cong\{(x_1,\ldots,x_m)\in\R^m\colon\gamma_1 x^2_1+ 
  \cdots+\gamma_mx^2_m=\delta\}.
\]
Furthermore,  the data $\{\mathcal K,\mathrm A\}$ define the normal fan of a generic polyhedron $P$ given by~\eqref{ptope}, where $\mathrm A=\{\mb a_1,\ldots,\mb a_m\}$ is the Gale dual configuration in~$W^*$ and
 the numbers $b_i$ are defined by the expression $\delta=\sum_{i=1}^mb_i\gamma_i$.
\end{theorem}
\begin{proof}
By definition, $\sK$ is a pure simplicial complex of dimension~$m-k-1$. 
Condition (b) implies that if $\delta\in\cone\Gamma_J$, then $\delta\in\relint\cone\Gamma_J$, $\cone\Gamma_J$ is a full-dimensional cone (of dimension~$k$) and its relative interior is the interior. Indeed, suppose $\delta\notin\relint\cone\Gamma_J$, then $\delta$ lies in a face of 
$\cone\Gamma_J$, which is a cone of dimension~$<k$. By the Carath\'eodory Theorem~\cite[Proposition~1.15]{zieg95}, this cone is generated by subset of $\Gamma$ of cardinality $<k$, in contradiction with~(b). 

Now
\[
  \Sigma=\{\cone \mathrm A_I\colon \delta\in\relint\cone\Gamma_{\widehat I}\}
\]
is the normal fan of~$P$ by Proposition~\ref{normal2}. Furthermore, each $\cone \mathrm A_I$ in $\Sigma$ is simplicial, because $\Gamma_{\widehat I}$ is a spanning configuration by the argument above. Hence, $P$ is generic and the result follows from Theorem~\ref{quadrics}.
\end{proof}

We refer to $P$ as the \emph{associated polyhedron} of the intersection of quadrics 
$\gamma_1 x^2_1+\cdots+\gamma_mx^2_m=\delta$. It can be alternatively described as follows.

\begin{proposition}\label{ptopepos}
Let $\gamma_1 x^2_1+\cdots+\gamma_mx^2_m=\delta$ be the nondegenerate intersection of quadrics defined by $\Gamma=\{\gamma_1,\ldots,\gamma_m\}$ and $\delta$ as in Theorem~\ref{quadrics1}. The affine map
\[
  W\to\R^m, \quad \mb w\mapsto 
  (\langle\mb a_1,\mb w\rangle+b_1,\ldots,\langle\mb a_m,\mb w\rangle+b_m)
\]
takes the associated polyhedron $P\subset W$ isomorphically onto the subset of~$\R^m$ given as the intersection of the nonnegative cone $\R^m_\ge$ with an $n$-dimensional affine plane:  
\[
  \{(y_1,\ldots,y_m)\in\R_\ge^m\colon
  \gamma_1y_1+\cdots+\gamma_m y_m=\delta.\}
\]
\end{proposition}
\begin{proof}
The map $W\to\R^m$ above takes $W$ isomorphically onto the $n$-dimensional plane 
$\gamma_1y_1+\cdots+\gamma_m y_m=\delta$, because
\[
  \gamma_1(\langle\mb a_1,\mb w\rangle+b_1)+\cdots+
  \gamma_m(\langle\mb a_m,\mb w\rangle+b_m)= \gamma_1b_1+\cdots+
  \gamma_mb_m=\delta.
\]
The image of $P$ consists of points with nonnegative coordinates by~\eqref{ptope}.
\end{proof}

\begin{construction}[GKZ decomposition]
Let $\Gamma$ be a spanning vector configuration in $V^*$, and let $\mathrm A$ be the Gale dual configuration. For any $\delta\in \cone\Gamma$ set
\[
  C(\delta)=\bigcap_{I\subset[m]\colon\delta\in\cone\Gamma_{\widehat I}}
  \cone\Gamma_{\widehat I}=
  \bigcap_{I\subset[m]\colon\delta\in\relint\cone\Gamma_{\widehat I}}
  \cone\Gamma_{\widehat I}\,.
\]

\begin{proposition}
\begin{itemize}
\item[(a)] $\Sigma_\Gamma=\{C(\delta)\colon\delta\in\cone\Gamma\}$ is a fan with support $\cone\Gamma$;

\item[(b)] if $C(\delta)$ is a full-dimensional cone, then $\Sigma=\{\cone\mathrm A_I\colon\delta\in\relint\cone\Gamma_{\widehat I}\}$ is the normal simplicial fan of a generic polyhedron.
\end{itemize}
\end{proposition}
\begin{proof}
For any $\delta\in\cone\Gamma$ we have $\delta\in\relint C(\delta)$. Furthermore, if $\delta'\in\relint C(\delta)$, then $C(\delta')=C(\delta)$. It follows that the cones $C(\delta)$ corresponding to different $\delta\in\cone\Gamma$ either coincide or have disjoint relative interiors. Then $\{C(\delta)\colon\delta\in\cone\Gamma\}$ is a fan by Theorem~\ref{galefan}, proving~(a).

If $C(\delta)$ is full-dimensional, then $\delta$ satisfies the conditions of Theorem~\ref{quadrics1}, so that the fan $\{\cone\mathrm A_I\colon\delta\in\relint\cone\Gamma_{\widehat I}\}$ is normal and simplicial, proving~(b). 
\end{proof}

The decomposition of $\cone\Gamma$ into cones (or \emph{chambers}) $C(\delta)$ is called the \emph{GKZ decomposition} (after Gelfand, Kapranov and Zelevinsky), and $\Sigma_\Gamma$
is called the \emph{GKZ fan}, or the \emph{secondary fan} of~$\Gamma$, see~\cite[Chapter~7]{g-k-m91} and also~\cite{od-pa91}, \cite[\S2.2.2]{a-d-h-l15}.
It encodes generalised normal fans with cone generators chosen among the Gale dual configuration~$\mathrm A$. The relative interiors of maximal cones of 
$\Sigma_\Gamma$ correspond to the normal fans of generic polyhedra. Points in 
the interior of $\cone\Gamma$ correspond to the normal fans of full-dimensional polyhedra (see Proposition~\ref{normal2}), and points on the boundary of $\cone\Gamma$ correspond to generalised fans. We illustrate this next.
 \end{construction}

\begin{example}\label{GKZ}
Let $\mathrm A=\{\mb a_1,\mb a_2,\mb a_3,\mb a_4\}$ be the configuration in $\R^2$ given by the column vectors of the matrix $A=\begin{pmatrix}1&0&-1&-1\\0&1&0&-1\end{pmatrix}$. 
There are two basis relations $\mb a_1+\mb a_3=\mathbf 0$ and $\mb a_1+\mb a_2+\mb a_4=\mathbf 0$,  so
the Gale dual configuration $\Gamma=\{\gamma_1,\gamma_2,\gamma_3,\gamma_4\}$ is given
by the column vectors of the matrix $\varGamma=\begin{pmatrix}1&0&1&0\\1&1&0&1\end{pmatrix}$. Note that $\gamma_2=\gamma_4$.

Here $\cone\Gamma$ is the positive quadrant, and its GKZ decomposition has two maximal cones separated by the ray containing $\gamma_1=\begin{pmatrix}1\\1\end{pmatrix}$, see Fig.~\ref{fig2dim}.
\begin{figure}[h]
\begin{center}
\begin{picture}(100,100)
  {\thicklines
  \put(0,0){\vector(1,0){40}}
  \put(0,0){\vector(0,1){40}}
  \put(0,0){\vector(1,1){40}}  
  }
  \put(0,0){\line(1,0){100}}
  \put(0,0){\line(0,1){100}}
  \put(0,0){\line(1,1){100}}
  \put(38,35){$\gamma_1$}
  \put(-9,39){$\gamma_2,\gamma_4$}
  \put(38,-3){$\gamma_3$}
  \put(80,40){\circle*{1.5}}
  \put(40,80){\circle*{1.5}}
  \put(40,40){\circle*{1.5}}
  \put(40,00){\circle*{1.5}}
  \put(00,40){\circle*{1.5}}
  \put(0,0){\circle*{1.5}}
  \put(41,78){$\delta=\begin{pmatrix}1\\2\end{pmatrix}$}
  \put(13,55){\line(1,0){24}}
  \put(15,53){\line(0,1){44}}
  \put(35,53){\line(0,1){24}}
  \put(13,97){\line(1,-1){24}}
  {\thicklines
  \put(15,55){\line(1,0){20}}
  \put(15,55){\line(0,1){40}}
  \put(35,55){\line(0,1){20}}
  \put(15,95){\line(1,-1){20}}
  }
  \put(17,93){\small$-x_1-x_2+2\ge0$}
  \put(5,82){\small$x_1\ge0$}
  \put(36,62){\small$-x_1+1\ge0$}
  \put(20,52){\small$x_2\ge0$}
  \textcolor{red}{
  \put(25,72){\vector(1,0){20}}
  \put(25,72){\vector(0,1){20}}
  \put(25,72){\vector(-1,0){20}}
  \put(25,72){\vector(-1,-1){20}}
  }
  \put(41,39){$\delta=\begin{pmatrix}1\\1\end{pmatrix}$}
  \put(13,20){\line(1,0){24}}
  \put(15,18){\line(0,1){24}}
  \multiput(35,18)(0,5){5}{\line(0,1){3}}
  \put(13,42){\line(1,-1){24}}
  {\thicklines
  \put(15,20){\line(1,0){20}}
  \put(15,20){\line(0,1){20}}
  \put(15,40){\line(1,-1){20}}
  }
  \put(81,39){$\delta=\begin{pmatrix}2\\1\end{pmatrix}$}
  \put(58,20){\line(1,0){24}}
  \put(60,18){\line(0,1){24}}
  \multiput(100,18)(0,5){5}{\line(0,1){3}}
  \put(58,42){\line(1,-1){24}}
  {\thicklines
  \put(60,20){\line(1,0){20}}
  \put(60,20){\line(0,1){20}}
  \put(60,40){\line(1,-1){20}}
  }
  \put(68,32){\small$-x_1-x_2+1\ge0$}
  \put(50,26){\small$x_1\ge0$}
  \put(82,23){\small$-x_1+2\ge0$}
  \put(65,17){\small$x_2\ge0$}
  \textcolor{red}{
  \put(65,30){\vector(1,0){20}}
  \put(65,30){\vector(0,1){20}}
  \multiput(64,30)(-5,0){3}{\line(-1,0){3}}
  \put(49,30){\vector(-1,0){4}}
  \put(65,30){\vector(-1,-1){20}}
  }
\end{picture}  
\end{center}
\caption{GKZ decomposition and the corresponding normal fans}
\label{fig2dim}
\end{figure}
Then $\delta$ satisfies conditions (a) and (b) of Theorem~\ref{quadrics1} if and only if it belongs to the interior of one of these two cones.
Choosing $\delta$ in different cones gives different polyhedra and different normal fans.

Let $\delta=\begin{pmatrix}1\\2\end{pmatrix}$. Then $\delta$ belongs to (the relative interior of) $\cone(\gamma_1,\gamma_2)$, $\cone(\gamma_2,\gamma_3)$, $\cone(\gamma_3,\gamma_4)$, $\cone(\gamma_4,\gamma_1)$. The corresponding complete fan $\Sigma$ consists of the cones $\cone(\mb a_3,\mb a_4)$, $\cone(\mb a_4,\mb a_1)$, $\cone(\mb a_1,\mb a_2)$, $\cone(\mb a_2,\mb a_3)$ and their faces; it is shown in red next to $\delta=\begin{pmatrix}1\\2\end{pmatrix}$ in Fig.~\ref{fig2dim}. This fan is the normal fan of a polytope given by the inequalities~\eqref{ptope} in which the numbers $b_i$ are the coefficients in an expression $\delta=\sum b_i\gamma_i$, for example,
\[
  \begin{pmatrix}1\\2\end{pmatrix}=0\begin{pmatrix}1\\1\end{pmatrix}+
  0\begin{pmatrix}0\\1\end{pmatrix}+1\begin{pmatrix}1\\0\end{pmatrix}+
  2\begin{pmatrix}0\\1\end{pmatrix}.
\] 
The resulting polytope is a trapezoid shown next to $\delta=\begin{pmatrix}1\\2\end{pmatrix}$ in Fig.~\ref{fig2dim} (in thick black lines) together with its defining inequalities. It is a generic polytope without redundant inequalities. 

The intersection of quadrics $\gamma_1x_1^2+\cdots+\gamma_mx_m^2=\delta$ is given by
\[
  \biggl\{\,\begin{aligned}
  x_1^2&&+x_3^2&&=1,\\
  x_1^2&+x_2^2&&+x_4^2&=2,
  \end{aligned}
\]
and is diffeomorphic to $S^1\times S^1$.

Now let $\delta=\begin{pmatrix}2\\1\end{pmatrix}$. Then $\delta$ belongs to 
$\cone(\gamma_1,\gamma_3)$, $\cone(\gamma_2,\gamma_3)$, $\cone(\gamma_3,\gamma_4)$. The corresponding complete fan $\Sigma$ consists of the cones $\cone(\mb a_2,\mb a_4)$, $\cone(\mb a_4,\mb a_1)$, $\cone(\mb a_1,\mb a_2)$ and their faces; it is shown in red next to $\delta=\begin{pmatrix}2\\1\end{pmatrix}$ in Fig.~\ref{fig2dim}. This time we write
\[
  \begin{pmatrix}2\\1\end{pmatrix}=0\begin{pmatrix}1\\1\end{pmatrix}+
  0\begin{pmatrix}0\\1\end{pmatrix}+2\begin{pmatrix}1\\0\end{pmatrix}+
  1\begin{pmatrix}0\\1\end{pmatrix}.
\] 
The resulting polytope is a triangle shown next to $\delta=\begin{pmatrix}2\\1\end{pmatrix}$ in Fig.~\ref{fig2dim}. It is still generic, but with one redundant inequality $-x_1+2\ge0$.

The intersection of quadrics $\gamma_1x_1^2+\cdots+\gamma_mx_m^2=\delta$ is now given by
\[
  \biggl\{\,\begin{aligned}
  x_1^2&&+x_3^2&&=2,\\
  x_1^2&+x_2^2&&+x_4^2&=1,
  \end{aligned}
\]
and is diffeomorphic to $S^2\times S^0=S^2\sqcup S^2$. 

Let us also see what happens when $\delta$ is chosen in smaller cones of the GKM-decomposition, so that it does not satisfy conditions (a) and (b) of Theorem~\ref{quadrics1}. Let 
$\delta=\begin{pmatrix}1\\1\end{pmatrix}$. Since $\delta\in\relint\cone\Gamma$, Proposition~\ref{normal2} still applies, giving the fan consisting of $\cone(\mb a_2,\mb a_3,\mb a_4)$, $\cone(\mb a_4,\mb a_1)$, $\cone(\mb a_1,\mb a_2)$. This is the same fan as in the previous case of $\delta=\begin{pmatrix}1\\2\end{pmatrix}$, but this time it is obtained as the normal fan of a non-generic polyhedron. It is also shown in Fig.~\ref{fig2dim}.

For $\delta=\begin{pmatrix}0\\1\end{pmatrix}$, the polyhedron is given by the inequalities $x_1\ge0$, $x_2\ge0$, $-x_1\ge0$ and $-x_1-x_2+1\ge0$. It is a vertical segment, so it is not full-dimensional. This reflects the fact that $\delta\notin\relint\cone\Gamma$. The generalised normal fan consists of the horizontal line and two halfplanes defined by it.

Finally, for $\delta=\begin{pmatrix}1\\0\end{pmatrix}$ or 
$\delta=\begin{pmatrix}0\\0\end{pmatrix}$, the polyhedron is a point and its generalised normal fan is the whole plane.
\end{example}

%The decomposition of $\cone\Gamma$ into chambers corresponding to normal fans, illustrated in Fig.~\ref{fig2dim} and Example~\ref{GKZ}, is known as the \emph{Gelfand--Kapranov--Zelevinsky decomposition} (\emph{GKZ-decomposition}). It plays an important role in toric geometry, see~\cite{od-pa91}, \cite[\S14.4]{c-l-s11}, \cite[\S2.2.2]{a-d-h-l15}.

In general, the topology of real moment-angle manifolds $\rk$ and, in particular, intersections of quadrics of the form $\gamma_1x_1^2+\cdots+\gamma_mx_m^2=\delta$ is quite complicated. See~\cite[Chapters~4 and~6]{bu-pa15}.

\section{Convex hulls and links of special quadrics}\label{convlink}
There is an important special class of vector configurations and the corresponding foliations~\eqref{ract}, which are studied in holomorphic dynamics~\cite{meer00, me-ve04, bosi01, bo-me06} and feature in the definition of complex LVM-manifolds. We review how this  construction fits the general framework of the previous section.

\begin{proposition}\label{posit}
Let $\mathrm A=\{\mb a_1,\ldots,\mb a_m\}$ and $\Gamma=\{\gamma_1,\ldots,\gamma_m\}$ be a pair of Gale dual vector configurations in $W^*$ and $V^*$, respectively. The following conditions are equivalent:

\begin{itemize}
\item[(a)] there is a linear relation $\sum_{i=1}^m r_i\mb a_i=\mathbf 0$ with $r_i>0$ for $i=1,\ldots,m$;\\[-0.7\baselineskip]

\item[(b)]
There is $\mb v_1\in V$ such that $\langle\gamma_i,\mb v_1\rangle>0$ for $i=1,\ldots,m$,  that is, all $\gamma_i$ lie in the positive halfspace defined by~$\mb v_1$;\\[-0.7\baselineskip]

\item[(c)] 
$\cone\Gamma$ is a strictly convex (pointed) cone.\\[-0.7\baselineskip]

\item[(d)]
$\cone\mathrm {A}$ is the whole of $W^*$.
\end{itemize}
\end{proposition}

\begin{proof}
(a)$\Rightarrow$(b) This follows from the fact that $V$ can be identified with the space of linear relations between $\mb a_1,\ldots,\mb a_m$, and $\gamma_i$ is the linear function on $V$ that takes a linear relation to its $i$th coefficient.

\smallskip

(b)$\Rightarrow$(c) Indeed, $\mathbf 0$ is a face of $\cone\Gamma$, defined by the supporting hyperplane~$\mb v_1$. 

\smallskip

(c)$\Rightarrow$(d) $\cone\Gamma$ and its face
$\{\mathbf 0\}=\cone\Gamma_\varnothing$ have no common relative interior points. By Theorem~\ref{galefangen}, this implies that 
$\cone\mathrm{A}_\varnothing=\{\mathbf 0\}$ is in the relative interior of 
$\cone\mathrm {A}$. Since $\mathrm{A}$ is a spanning configuration, we get 
$\cone\mathrm{A}=W^*$.

(d)$\Rightarrow$(a) If $\cone\mathrm {A}=W^*$, then $\mathbf 0$ is in the relative interior of $\cone\mathrm {A}$. Now (a) follows from Proposition~\ref{positive}.
\end{proof}

If $\mathrm A=\{\mb a_1,\ldots,\mb a_m\}$ is a set of generators of a complete fan, then $\mathrm A$ satisfies the conditions of Proposition~\ref{posit}.

If $\Gamma$ satisfies the conditions of Proposition~\ref{posit}, then by multiplying each vector $\gamma_i$ by the corresponding positive number, we can achieve that 
$\langle\gamma_i,\mb v_1\rangle=1$ for $i=1,\ldots,m$. Consider the hyperplane
\[
  \Ann\mb v_1=\{\gamma\in V^*\colon \langle\gamma,\mb v_1\rangle=0\},
\]
choose arbitrary $\delta\in V^*$ such that $\langle\delta,\mb v_1\rangle=1$, and write 
\[
  \gamma_i=\gamma'_i+\delta,\quad \gamma'_i\in \Ann\mb v_1,\quad i=1,\ldots,m.
\]
Then we have $V^*=\Ann\mb v_1\oplus\R\langle\delta\rangle$.
% and $\Ann\mb v$ is the annihilator of~$\mb v$. 
Similarly, $V=\R\langle\mb v_1\rangle\oplus V'$, where $V'=\Ann\delta$. Writing 
$\gamma_i=\gamma'_i+\delta$ and $\mb v=t\mb v_1+\mb v'$, the action~\eqref{ract} takes the form
\begin{equation}\label{ract'}
\begin{aligned}
  (\R_>\times V')\times\R^m&\longrightarrow\R^m\\
  (r,\mb v',\mb x)&\mapsto \bigl(re^{\langle\gamma'_1,\mb
  v'\rangle}x_1,\ldots,re^{\langle\gamma'_m,\mb v'\rangle}x_m\bigr),
\end{aligned}
\end{equation}
where $r=e^t$, $t\in\R$. This is the action considered in~\cite{bo-me06}.

Note that multiplying $\gamma_i$ by $s>0$ is compensated by an automorphism 
of~$\R^m$ (which consists in multiplying $x_i$ by $e^{-s}$), and it does not affect the foliation.

The freeness and properness conditions for the action~\eqref{ract'} can be formulated in terms of the convex hulls of $\gamma'_i$ in $(V')^*\cong\R^{m-n-1}$ instead of the cones generated by $\gamma_i$ in~$V^*\cong\R^{m-n}$. This can be visualised by choosing a basis in $V^*$ in which $\delta$ is the last basis vector. In this basis, each $\gamma_i$ is written as a column 
$\begin{pmatrix}\gamma'_i\\1\end{pmatrix}$, so that the endpoints of all $\gamma_i$ lie in the affine hyperplane of points with the last coordinate~$1$.  

We therefore refer to $\Gamma'=\{\gamma'_1,\ldots,\gamma'_m\}$ as a \emph{point configuration} in the affine space $(V')^*$, as opposed to the vector configuration $\Gamma=\{\gamma_1,\ldots,\gamma_m\}$ in~$V^*$.
Given a subset $I\subset[m]$, we denote by $\conv\Gamma'_I$ the convex hull of  the points $\gamma'_i$, $i\in I$, in the affine space~$(V')^*$.

\begin{theorem}\label{proper'}
Let $\Gamma'=\{\gamma'_1,\ldots,\gamma'_m\}$ be a point configuration in an $(m-n-1)$-dimensional affine space~$(V')^*$, and let $\sK$ be a simplicial complex on~$[m]$. Then
\begin{itemize}
\item[(1)] the action $(\R_>\times V')\times U(\sK)\to U(\sK)$ obtained by restricting~\eqref{ract'} to $U(\sK)$ is free if and only if,  for any $I\in\sK$,
the affine span of $\Gamma'_{\widehat I}$ is the whole~$(V')^*$;

\item[(2)] the action $(\R_>\times V')\times U(\sK)\to U(\sK)$ is proper if and only if
$\conv\Gamma'_{\widehat I}$ and $\conv\Gamma'_{\widehat J}$ have a common interior point for any $I,J\in\sK$, in which case the action is also free.
\end{itemize}
\end{theorem}
\begin{proof}
We let $V=V'\oplus\R$ and $\gamma_i=\begin{pmatrix}\gamma'_i\\1\end{pmatrix}$ as above and convert the action~\eqref{ract'} into the exponential action of $V$ given by~\eqref{ract} with the additional condition that each $\gamma_i$ belongs to a specific affine hyperplane $H$ where the last coordinate is~$1$. Then the affine span of $\Gamma'_{\widehat I}$ is $(V')^*$ if and only if the linear span of $\Gamma_{\widehat I}$ is~$V^*$. Therefore, the first statement follows from Theorem~\ref{proper}~(1).

For the second statement, observe that the interior of $\conv\Gamma'_{\widehat I}$ is the intersection of the interior of $\cone\Gamma_{\widehat I}$ with the hyperplane~$H$. Hence, there result follows from from Theorem~\ref{proper}~(2).
\end{proof}

Condition~(2) in Theorem~\ref{proper'} is precisely the \emph{imbrication condition} of Bosio, see~\cite[Lemma~1.1, Definition~1.8]{bosi01}. As a further specification, we obtain the following version of Theorem~\ref{quadrics1}:

\begin{theorem}\label{quadrics2}
Let $\Gamma'=\{\gamma'_1,\ldots,\gamma'_m\}$ be a point configuration in an affine space $(V')^*$ of dimension $k-1$ satisfying the two conditions:
\begin{itemize}
\item[(a)] $\mathbf{0}\in\conv\Gamma'$;\\[-0.7\baselineskip]

\item[(b)]
if $\mathbf{0}\in\conv\Gamma'_J$, then $|J|\ge k$,
\end{itemize}
and let
\begin{equation}\label{Kgamma}
  \sK=\{I\in[m]\colon \mathbf{0}\in\conv\Gamma_{\widehat I}\}.
\end{equation}
Then the action~\eqref{ract'} of $\R_>\times V'$ on $U(\sK)$ is free and proper, and the quotient $U(\sK)/(\R_>\times V')$ is diffeomorphic to the following nondegenerate intersection of quadratic surfaces  in~$\R^m$:
\begin{equation}\label{link}
  \left\{\begin{array}{lrcl}
  (x_1,\ldots,x_m)\in\R^m\colon&
  \gamma'_1x_1^2+\cdots+\gamma'_m x_m^2&=&\mathbf{0},\\[1mm]
  &x_1^2+\cdots+x_m^2&=&1.
  \end{array}\right\}
\end{equation}
\end{theorem}
\begin{proof}
Let $V=V'\oplus\R$ and $\gamma_i=\begin{pmatrix}\gamma'_i\\1\end{pmatrix}$; then $\Gamma$ is a spanning vector configuration in~$V^*$ with the additional condition that each $\gamma_i$ belongs to the hyperplane $H$ where the last coordinate is~$1$. Then $\delta:=\begin{pmatrix}\mathbf{0}\\1\end{pmatrix}$ also belongs to $H$ and satisfies the conditions of Theorem~\ref{quadrics1}. The equations 
$\gamma_1x_1^2+\cdots+\gamma_mx_m^2=\delta$ take the form~\eqref{link}.
\end{proof}

Note that whenever the vectors $\gamma_i$ in $\Gamma$ belong to the hyperplane $\Ann\mb v_1$ we may take 
$\delta=\begin{pmatrix}\delta'\\1\end{pmatrix}$ in Theorem~\ref{quadrics1} and choose an affine coordinate system in the hyperplane so that $\delta'=\mathbf 0$. Then the equations 
$\gamma_1x_1^2+\cdots+\gamma_mx_m^2=\delta$ take the form~\eqref{link}.

Conditions (a) and (b) in Theorem~\ref{quadrics2} are known as the \emph{Siegel condition} and the \emph{weak hyperbolicity condition}, see e.\,g.~\cite[\S I]{meer00}. The manifold given by~\eqref{link} is called the \emph{link of a system of special real quadrics}~\cite[Definition~0.1]{bo-me06}. It is diffeomorphic to the real moment-angle manifold $\mathcal R_{\mathcal K}$, where $\sK$ is dual to the simple (generic) polytope $P$ given by 
\[
   \left\{\begin{array}{lrcl}
  (y_1,\ldots,y_m)\in\R_\ge^m\colon&
  \gamma'_1y_1+\cdots+\gamma'_m y_m&=&\mathbf{0},\\[1mm]
  &y_1+\cdots+y_m&=&1.
  \end{array}\right\}
\]
(see Corollary~\ref{rkquad}).

\begin{example}\label{5gon}
Consider the vector configuration $\mathrm A=\{\mb a_1,\ldots,\mb a_5\}$ in $\R^2$ given by the columns of the matrix 
\[
  A=\begin{pmatrix}2&0&-2&-1&1\\1&4&1&-3&-3\end{pmatrix},
\]  
see Fig.~\ref{fig5gon}, left. Note that $\mb a_1+\cdots+\mb a_5=\mathbf 0$, so Proposition~\ref{posit} applies.
\begin{figure}[h]
\begin{center}
\begin{picture}(100,45)
  \put(10,20){\vector(2,1){10}}
  \put(10,20){\vector(0,1){20}}
  \put(10,20){\vector(-2,1){10}}
  \put(10,20){\vector(-1,-3){5}}
  \put(10,20){\vector(1,-3){5}}
  \put(20,25){\circle*{1}}
  \put(10,40){\circle*{1}}
  \put(0,25){\circle*{1}}
  \put(5,5){\circle*{1}}
  \put(15,5){\circle*{1}}
  \put(21,25){$\mb a_1$}
  \put(11,40){$\mb a_2$}
  \put(-5,25){$\mb a_3$}
  \put(2,2){$\mb a_4$}
  \put(15,2){$\mb a_5$}
  \put(50,30){\circle*{1}}
  \put(50,20){\circle*{1}}
  \put(60,25){\circle*{1}}
  \put(55,35){\circle*{1}}
  \put(45,35){\circle*{1}}
  \put(40,25){\circle*{1}}
  \put(40,25){\line(1,2){5}}
  \put(40,25){\line(2,-1){10}}
  \put(45,35){\line(1,0){10}}
  \put(55,35){\line(1,-2){5}}
  \put(50,20){\line(2,1){10}}
  \put(40,25){\line(3,2){15}}
  \put(40,25){\line(1,0){20}}
  \put(45,35){\line(3,-2){15}}
  \put(50,20){\line(1,3){5}}
  \put(50,20){\line(-1,3){5}}
  \put(49,27){\small$\mathbf 0$}
  \put(56,35){$\gamma_1$}
  \put(35.5,25){$\gamma_5$}
  \put(61,25){$\gamma_4$}
  \put(40.5,35){$\gamma_3$}
  \put(48.5,17){$\gamma_2$}
  \put(95,30){\circle*{1}}
  \put(90,20){\circle*{1}}
  \put(100,25){\circle*{1}}
  \put(95,35){\circle*{1}}
  \put(85,35){\circle*{1}}
  \put(80,25){\circle*{1}}
  \put(80,25){\line(1,2){5}}
  \put(80,25){\line(2,-1){10}}
  \put(85,35){\line(1,0){10}}
  \put(95,35){\line(1,-2){5}}
  \put(90,20){\line(2,1){10}}
  \put(80,25){\line(3,2){15}}
  \put(80,25){\line(1,0){20}}
  \put(85,35){\line(3,-2){15}}
  \put(90,20){\line(1,3){5}}
  \put(90,20){\line(-1,3){5}}
  \put(95.5,28){\small$\mathbf 0$}
  \put(96,35){$\gamma_1$}
  \put(75.5,25){$\gamma_5$}
  \put(101,25){$\gamma_4$}
  \put(80.5,35){$\gamma_3$}
  \put(88.5,17){$\gamma_2$}
\end{picture}  
\end{center}
\caption{Five-vector configuration and its Gale dual}
\label{fig5gon}
\end{figure}

The Gale dual configuration $\Gamma=(\gamma_1,\ldots,\gamma_5)$ is given by the columns of
\[
  \varGamma=\begin{pmatrix}1&0&-1&2&-2\\1&-2&1&-1&-1\\
  1&1&1&1&1\end{pmatrix}
\]   
up to a change of coordinates in~$\R^3$. The vectors
$\gamma_i=\begin{pmatrix}\gamma'_i\\1\end{pmatrix}$ lie in the affine plane~$x_3=1$, and the corresponding $2$-dimensional point configuration $\Gamma'=(\gamma'_1,\ldots,\gamma'_5)$ is shown in Fig.~\ref{fig5gon}, middle.
%$\varGamma'=\begin{pmatrix}1&0&-1&2&-2\\1&-2&1&-1&-1\end{pmatrix}$ 
It satisfies the conditions of Theorem~\ref{quadrics2}, so the corresponding system of quadrics~\eqref{link} is nondegenerate. It is given by
\begin{equation}\label{link5}
  \left\{\begin{array}{cccccccccc}
  x_1^2&&&-&x_3^2&+&2x_4^2&-&2x_5^2&=0,\\[2pt]
  x_1^2&-&2x_2^2&+&x_3^2&-&x_4^2&-&x_5^2&=0,\\[2pt]
  x_1^2&+&x_2^2&+&x_3^2&+&x_4^2&+&x_5^2&=1,
  \end{array}\right.
\end{equation}
The simplicial complex $\sK$~\eqref{Kgamma} here is the boundary of pentagon with $1$-simplices $\{1,2\}$, $\{2,3\}$, $\{3,4\}$, $\{4,5\}$ and $\{5,1\}$. The intersection of quadrics~\eqref{link5} is homeomorphic to the real moment-angle manifold $\mathcal R_{\mathcal K}$, which is an oriented surface of genus~$5$ (see~\cite[Proposition~4.1.8]{bu-pa15}). The associated polytope is given by~\eqref{ptope}, where $\sum_{i=1}^5b_i\gamma_i=\begin{pmatrix}\mathbf 0\\1\end{pmatrix}$. One choice is $b_1=b_3=b_4=b_5=\frac14$ and $b_2=0$. It is easy to see that $P$ is a pentagon.

In the notation of Theorem~\ref{quadrics1}, we have $k=3$ and the cones 
$\cone\Gamma_J$ with $|J|=2$ split $\cone\Gamma$ into $11$ chambers, corresponding to the maximal cones of the GKZ decomposition. The intersection of the GKZ fan with the plane $x_3=1$ is the pentagon $\conv\Gamma'$ split into $11$ two-dimensional chambers as shown in Fig.~\ref{fig5gon}, middle. Conditions (a) and (b) of Theorem~\ref{quadrics1} imply that if $\delta$ is chosen in the interior of one of the $11$ chambers, then the corresponding intersection of quadrics is nondegenerate, but its topological type changes when $\delta$ crosses one of the walls $\cone\Gamma_J$ with $|J|=2$.

The origin $\delta'=\begin{pmatrix}0\\0\end{pmatrix}$ lies in the central chamber in Fig.~\ref{fig5gon}, middle. On the other hand, $\delta'=\begin{pmatrix}1\\0\end{pmatrix}$ lies in the interior of another chamber, see Fig.~\ref{fig5gon}, right. To fit this choice into the notation of Theorem~\ref{quadrics2} we must change the affine coordinate system on the hyperplane $x_3=1$, so that the new $\delta'$ becomes the origin. The corresponding intersection of quadrics is given by
\[
  \left\{\begin{array}{cccccccccc}
  x_1^2&&&-&x_3^2&+&2x_4^2&-&2x_5^2&=1,\\[2pt]
  x_1^2&-&2x_2^2&+&x_3^2&-&x_4^2&-&x_5^2&=0,\\[2pt]
  x_1^2&+&x_2^2&+&x_3^2&+&x_4^2&+&x_5^2&=1,
  \end{array}\right.
\]
or, equivalently,
\[
  \left\{\begin{array}{cccccccccc}
  &-&x_2^2&-&2x_3^2&+&x_4^2&-&3x_5^2&=0,\\[2pt]
  x_1^2&-&2x_2^2&+&x_3^2&-&x_4^2&-&x_5^2&=0,\\[2pt]
  x_1^2&+&x_2^2&+&x_3^2&+&x_4^2&+&x_5^2&=1.
  \end{array}\right.
\]
This time the simplicial complex~\eqref{Kgamma} has $1$-simplices $\{2,3\}$, 
$\{3,5\}$, $\{2,5\}$ and ghost vertices $\{1\}$ and $\{4\}$, the associated polytope $P$ is a triangle with two redundant inequalities, and the intersection of quadrics is homeomorphic to a disjoint union of four two-dimensional spheres.
\end{example}

\begin{example}\label{prism1}
Consider the vector configuration $\mathrm A=\{\mb a_1,\ldots,\mb a_6\}$ in $\R^3$ given by the columns of the matrix 
\[
  A=\begin{pmatrix}1&0&-1&1&0&-1\\0&1&-1&0&1&-1\\
  1&1&1&-1&-1&-1\end{pmatrix}.
\]   
We triangulate the boundary of the triangular prism $\conv(\mb a_1,\ldots,\mb a_6)$
as shown in Fig.~\ref{fig6gon}, left, and consider the complete simplicial fan 
$\Sigma$ with apex at $\bf 0$ over the faces of this triangulation. Then $\Sigma$ has $8$ three-dimensional cones.  Let $\mathcal K$ be its underlying simplicial complex.

Note that $\mb a_1+\cdots+\mb a_6=\mathbf 0$, so Proposition~\ref{posit} applies.

The Gale dual configuration $\Gamma=(\gamma_1,\ldots,\gamma_6)$ is given by the columns of
%Old Gamma configuration
%\[
%  \varGamma=\begin{pmatrix}-1&2&2&3&0&0\\2&-1&2&0&3&0\\
%  1&1&1&1&1&1\end{pmatrix},
%\]
%New Gamma configuration
\[
  \varGamma=\begin{pmatrix}1&-1&0&-1&1&0\\0&1&-1&0&-1&1\\
  1&1&1&1&1&1\end{pmatrix},
\]   
up to a change of coordinates in~$\R^3$. The vectors
$\gamma_i=\begin{pmatrix}\gamma'_i\\1\end{pmatrix}$ lie in the affine plane~$x_3=1$, and the corresponding $2$-dimensional point configuration 
$\Gamma'=(\gamma'_1,\ldots,\gamma'_6)$ is shown in Fig.~\ref{fig6gon}, right.
\begin{figure}[h]
\begin{tikzpicture}[scale=0.8]
  \draw[thick] (0,2)--(0,5)--(6,5)--(6,2)--(1.5,0.5)--cycle;
  %\draw[thick] (1.5,0.5)--(0,5)--(1.5,3.5)--cycle;
  \draw[thick] (0,2)--(1.5,3.5)--(0,5);
  %\draw[thick] (6,2)--(1.5,3.5)--(6,5);
  \draw[thick] (1.5,0.5)--(1.5,3.5)--(6,5)--cycle;
  %\draw[dashed] (6,2)--(0,2)--(6,5);
  \draw[dashed] (0,2)--(6,2)--(0,5);
  \draw[fill] (1.5,3.5) circle (1.6pt) node[anchor=south]{$\ \mb a_1$};
  \draw[fill] (6,5) circle (1.6pt) node[anchor=west]{$\mb a_2$};
  \draw[fill] (0,5) circle (1.6pt) node[anchor=east]{$\ \mb a_3$};
  \draw[fill] (1.5,0.5) circle (1.6pt) node[anchor=north]{$\mb a_4$};
  \draw[fill] (6,2) circle (1.6pt) node[anchor=west]{$\mb a_5$};
  \draw[fill] (0,2) circle (1.6pt) node[anchor=east]{$\mb a_6$};
  \draw[fill] (2.5,3) circle (1.6pt) node[anchor=north]{$\bf0$};
%
% Old Gamma configuration  
%  \draw[thick] (8,4.5)--(9.5,6)--(12.5,4.5)--(14,1.5)--(12.5,0)--(9.5,1.5)--cycle;
%%456  
%%  \draw[fill=red!20] (14,1.5)--(9.5,6)--(9.5,1.5)--cycle;
%%123  
%%  \draw[fill=blue!20] (8,4.5)--(12.5,0)--(12.5,4.5)--cycle;
%%356
%  \draw[fill=red!40, opacity=0.5] (12.5,4.5)--(9.5,6)--(9.5,1.5)--cycle;
%%136
%  \draw[fill=green!40, opacity=0.5] (8,4.5)--(12.5,4.5)--(9.5,1.5)--cycle;
%%146
%  \draw[fill=blue!40, opacity=0.5] (8,4.5)--(14,1.5)--(9.5,1.5)--cycle;
%%124
%  \draw[fill=cyan!40, opacity=0.5] (8,4.5)--(12.5,0)--(14,1.5)--cycle;
%%245
%  \draw[fill=magenta!40, opacity=0.5] (12.5,0)--(14,1.5)--(9.5,6)--cycle;
%%235
%  \draw[fill=yellow!40, opacity=0.5] (12.5,0)--(12.5,4.5)--(9.5,6)--cycle;
%%256
%%  \draw[fill=magenta!40, opacity=0.5] (12.5,0)--(9.5,6)--(9.5,1.5)--cycle;
%%125
%%  \draw[fill=yellow!40, opacity=0.5] (8,4.5)--(12.5,0)--(9.5,6)--cycle;  
%%
%  \draw[fill] (8,4.5) circle (1.6pt) node[anchor=east]{$\gamma_1$};
%  \draw[fill] (12.5,0) circle (1.6pt) node[anchor=north]{$\gamma_2$};
%  \draw[fill] (12.5,4.5) circle (1.6pt) node[anchor=west]{$\gamma_3$};
%  \draw[fill] (14,1.5) circle (1.6pt) node[anchor=west]{$\gamma_4$};
%  \draw[fill] (9.5,6) circle (1.6pt) node[anchor=south]{$\gamma_5$};
%  \draw[fill] (9.5,1.5) circle (1.6pt) node[anchor=east]{$\gamma_6$};  
%  
%New Gamma configuration
  \draw[thick] (8,3)--(8,6)--(11,6)--(14,3)--(14,0)--(11,0)--cycle;
%356
  \draw[fill=red!40, opacity=0.5] (11,0)--(14,0)--(11,6)--cycle;
%136
  \draw[fill=green!40, opacity=0.5] (14,3)--(11,0)--(11,6)--cycle;
%146
  \draw[fill=blue!40, opacity=0.5] (14,3)--(8,3)--(11,6)--cycle;
%124
  \draw[fill=cyan!40, opacity=0.5] (14,3)--(8,6)--(8,3)--cycle;
%245
  \draw[fill=magenta!40, opacity=0.5] (8,6)--(8,3)--(14,0)--cycle;
%235
  \draw[fill=yellow!40, opacity=0.5] (8,6)--(11,0)--(14,0)--cycle;
  \draw[fill] (14,3) circle (1.6pt) node[anchor=west]{$\gamma_1$};
  \draw[fill] (8,6) circle (1.6pt) node[anchor=south east]{$\gamma_2$};
  \draw[fill] (11,0) circle (1.6pt) node[anchor=north]{$\gamma_3$};
  \draw[fill] (8,3) circle (1.6pt) node[anchor=east]{$\gamma_4$};
  \draw[fill] (14,0) circle (1.6pt) node[anchor=north west]{$\gamma_5$};
  \draw[fill] (11,6) circle (1.6pt) node[anchor=south]{$\gamma_6$};  
\end{tikzpicture}
\caption{Non-polytopal fan}
\label{fig6gon}
\end{figure}

The eight triangles $\conv\Gamma'_{\widehat I}$ with $I\in\sK$ are shown in Fig.~\ref{fig6gon}, right, in different colours. Each pair of these triangles has a common interior point, which reflects the fact that $\Sigma$ is a fan, see Theorem~\ref{proper'}. However, the eight triangles altogether do not have a common interior point, so the fan $\Sigma$ is non-polytopal.

Replacing the diagonal $(\mb a_2,\mb a_4)$ of a rectangular face in Fig.~\ref{fig6gon}, left, with another diagonal $(\mb a_1,\mb a_5)$ results in a polytopal fan, which can also be seen from the Gale dual configuration. 
\end{example}

%\section{Norm minima}

\section{Holomorphic exponential actions and complex-analytic structures on moment-angle manifolds}\label{compzk}

To define the holomorphic version of the exponential action, assume that $\dim V=k$ is even, $k=2\ell$, and choose a complex structure $\mathcal J\colon V\to V$, $\mathcal J^2=-\mathop\mathrm{id}$. We denote the resulting $\ell$-dimensional complex space $(V,\mathcal J)$ by~$\widetilde V$.

The real dual space $V^*=\Hom(V,\R)$ has a complex structure $\mathcal J'$ defined by
\[
  \langle\mathcal J'\gamma,\mb v\rangle=\langle\gamma,\mathcal J\mb v\rangle
  \quad\text{for }\gamma\in V^*,\mb v\in V.
\]
We denote the resulting complex space $(V^*,\mathcal J')$ by~$\widetilde V^*$; there is a canonical $\C$-linear isomorphism $\widetilde V^*\cong\Hom(\widetilde V,\C)$ given by $\gamma\mapsto\langle\gamma,\,\cdot\,\rangle-i\langle\gamma,\mathcal J\,\cdot\,\rangle$.

We denote the complex pairing $\widetilde V^*\times\widetilde V\to\C$ by 
$\langle\gamma,\mb v\rangle_\C$ and the real pairing $V^*\times V\to\R$ by 
$\langle\gamma,\mb v\rangle_\R$ when it is necessary to distinguish between them. We have
\[
  \langle\gamma,\mb v\rangle_\C=\langle\gamma,\mb v\rangle_\R-
  i\langle\gamma,\mathcal J\mb v\rangle_\R.
\]

A vector configuration $\Gamma=\{\gamma_1,\ldots,\gamma_m\}$ in~$V^*$ can be viewed as a vector configuration in the complex space~$\widetilde V^*$. We continue to assume that the \emph{real} span of $\Gamma$ is the whole of~$V^*$. 

The \emph{holomorphic exponential action} of $\widetilde V$ on $\C^m$ is given by
\begin{equation}\label{hact}
\begin{aligned}
  \widetilde V\times\C^m&\longrightarrow\C^m\\
  (\mb v,\mb z)&\mapsto \mb v\cdot\mb z=\bigl(e^{\langle\gamma_1,\mb
  v\rangle_\C}z_1,\ldots,e^{\langle\gamma_m,\mb v\rangle_\C}z_m\bigr).
\end{aligned}
\end{equation}
It is related to the real exponential action~\eqref{ract} via the commutative diagram
\[
\begin{tikzcd}
  \widetilde V\times\C^m \ar{r} 
  \ar{d}{\mathop\mathrm{id}\times|\cdot|} & \C^m \ar{d}{|\cdot|}\\
  V\times\R^m \ar{r} & \R^m
\end{tikzcd}
\]
where $|\cdot|$ denotes the map $(z_1,\ldots,z_m)\mapsto(|z_1|,\ldots,|z_m|)$.

Free orbits of the action~\eqref{hact} can be detected by looking at the real spans of subsets of $\Gamma$, similarly to Proposition~\ref{freeorbit}:

\begin{proposition}\label{Cfreeorbit}
The stabiliser $\widetilde V_{\mb z}$ of a point $\mb z=(z_1,\ldots,z_m)\in\C^m$ under the action~\eqref{hact} is zero if $\R\langle\gamma_i\colon z_i\ne0\rangle=V^*$.
\end{proposition}
\begin{proof}
Suppose $\widetilde V_{\mb z}\ne\bf0$, so there exists $\mb v\ne0$ such that
\[
  (z_1e^{\langle\gamma_1,\mb
  v\rangle_\C},\ldots,z_me^{\langle\gamma_m,\mb v\rangle_\C})=(z_1,\ldots,z_m).
\]
Then $\Re\langle\gamma_i,\mb v\rangle_\C=0$ for $z_i\ne0$, which implies that the vectors $\gamma_i$ with $z_i\ne0$ belong to the real hyperplane $\Re\langle\,\cdot\,,\mb v\rangle_\C=0$ and do not span~$V^*$.
\end{proof}

Unlike the situation with the real exponential action, the converse statement to Proposition~\ref{Cfreeorbit} does not hold, as can be seen from Example~\eqref{2gamma} below.

For any simplicial complex $\mathcal K$ on $[m]$, there is the corresponding complement of an arrangement of \emph{complex} coordinate subspaces 
in~$\C^m$:
\begin{equation}\label{UKC}
  U_\C(\mathcal K)=\C^m\setminus\bigcup_{\{i_1,\ldots,i_p\}\notin\mathcal K}
  \{\mb z\colon z_{i_1}=\cdots=z_{i_p}=0\}.
\end{equation}
There is a polyhedral product decomposition $U_\C(\sK)=(\C,\C^\times)^\sK$, where $\C^\times=\C\setminus\{0\}$, similar to Example~\ref{expp}.1.

Recall the universal simplicial complex $\mathcal K(\Gamma)=\{I\subset[m]\colon \R\langle\Gamma_{\widehat I}\rangle=V^*\}$. We have the following analogue of Proposition~\ref{free}, with the same proof:

\begin{proposition}\label{Cfree}
If $\sK\subset\sK(\Gamma)$ is a simplicial subcomplex, then the restriction of the action~\eqref{hact} to $U_\C(\sK)$ is free.
\end{proposition}

The properness criterion for the holomorphic exponential action is similar to that of Theorem~\ref{proper}:

\begin{theorem}\label{hproper}
Let $\Gamma=\{\gamma_1,\ldots,\gamma_m\}$ and $\mathrm A=\{\mb a_1,\ldots,\mb a_m\}$ be a pair of Gale dual vector configurations in $V^*$ and $W^*$, respectively, and let $\sK$ be a simplicial complex on~$[m]$. The holomorphic exponential action $\widetilde V\times U_\C(\sK)\to U_\C(\sK)$ is proper if and only if $\{\sK,\mathrm A\}$ is a triangulated configuration, in which case the action is also free.
\end{theorem}
\begin{proof}
Consider the commutative diagram of shear maps
\[
\begin{tikzcd}
  \widetilde V\times U_\C(\sK) \ar{r}{h_\C} 
  \ar{d}{\mathop\mathrm{id}\times p} & 
  U_\C(\sK)\times U_\C(\sK) \ar{d}{p\times p}\\
  V\times U(\sK) \ar{r}{h} & U(\sK)\times U(\sK)
\end{tikzcd}
\]
where $p=|\cdot|$. The result will follow from Theorem~\ref{proper} if we prove that $h_\C$ is proper if and only if $h$ is proper.

Suppose $h$ is proper. The preimage $h_\C^{-1}(C)$ of a compact subset $C\subset U_\C(\sK)\times U_\C(\sK)$ is a closed subset in $W=((\mathop\mathrm{id}\times p)\circ h)^{-1}((p\times p)(C))$. The map $p$ can be identified with the quotient projection for an action of a compact group (a torus), so it is proper. Hence, $W$ is a compact subset in $\widetilde V\times U_\C(\sK)$ and  $h_\C^{-1}(C)$ is compact as a closed subset in~$W$.

Suppose $h_\C$ is proper. If $R\subset U(\sK)\times U(\sK)$ is a compact subset, then $h^{-1}(R)=(\mathop\mathrm{id}\times p)
\bigl(((p\times p)\circ h_\C)^{-1}(R)\bigr)$ is compact, because both $h_\C$ and $p$ are proper.

The fact that the action $\widetilde V\times U_\C(\sK)\to U_\C(\sK)$ defined by a triangulated configuration $\{\sK,\mathrm A\}$ is free follows from Proposition~\ref{Cfree}.
\end{proof}

\begin{example}\label{2gamma}
Consider the exponential action of $\widetilde V\cong\C$ on $\C^2$ defined by a configuration of two vectors $\{\gamma_1,\gamma_2\}$ in $\widetilde V^*$:
\[
   \widetilde V\times\C^2\to\C^2,\quad (v,z_1,z_2)\mapsto
   (e^{\gamma_1v}z_1,e^{\gamma_2v}z_2).
\]

Let $\sK=\{\varnothing\}$ on the two-element set $[2]=\{1,2\}$, so that
$U_\C(\sK)=(\C^\times)^2$.
If $\R\langle\gamma_1,\gamma_2\rangle=\R^2$, then the restriction of the action to $(\C^\times)^2$ is free and proper by Theorem~\ref{hproper} . The quotient is given by
\[
  (\C^\times)^2/\widetilde V\cong \C/(\Z\langle\gamma_1,\gamma_2\rangle)
  \cong T^2,
\]
it is a one-dimensional complex torus.

If $\gamma_1=\alpha\gamma_2$ with $\alpha\in\R\setminus\Q$, then the action 
$\widetilde V\times(\C^\times)^2\to(\C^\times)^2$ is still free, so the converse of Proposition~\ref{Cfreeorbit} does not hold. However, this action is not proper, as the orbit of the point $(1,1)\in(\C^\times)^2$ is not closed. The quotient 
$(\C^\times)^2/\widetilde V$ is non-Hausdorff.

Finally, if $\gamma_1=\alpha\gamma_2$ with $\alpha\in\Q$, then the action $\widetilde V\times(\C^\times)^2\to(\C^\times)^2$ is not free, as any $v\in\widetilde V$ such that 
$\gamma_1v=2\pi ik_1$ and $\gamma_2v=2\pi ik_2$ with integer $k_1,k_2$ acts trivially. If at least one of $\gamma_1,\gamma_2$ is nonzero, the quotient $(\C^\times)^2/\widetilde V$ is~$\C^\times$. 
\end{example}

The criterion for compactness of the quotient is also the same as in the real case, with the same proof:

\begin{theorem}\label{hcompact}
Let $\widetilde V\times U_\C(\sK)\to U_\C(\sK)$ be a proper holomorphic exponential action.
The quotient $U_\C(\sK)/\widetilde V$ is compact if and only if the corresponding fan $\Sigma$ is complete.
\end{theorem}

The \emph{moment-angle complex} $\zk$ is the polyhedral product 
$(D^1,S^0)^\sK$, where $D^2$ is a closed unit disc in $\C$ and $S^1$ is
its boundary circle:
\begin{equation}\label{madec}
  \zk=(D^2,S^1)^\sK=\bigcup_{I\in\sK}(D^2,S^1)^I.
\end{equation}
It is a CW-subcomplex in the product of $m$ discs $(D^2)^m$ with the induced action of $m$-torus $T^m=(S^1)^m$. When $\sK$ is a simplicial subdivision of an $(n-1)$-sphere, i.\,e. $|\sK|\cong S^{n-1}$, the moment-angle complex $\zk$ is a closed topological $(m+n)$-dimensional manifold, called the \emph{moment-angle manifold}~\cite[Theorem~4.1.4]{bu-pa15}. The topology of moment-angle manifolds can be remarkably complicated, see~\cite[Chapter~4]{bu-pa15}.

If $\{\sK,\mathrm A\}$ is a triangulated configuration, then the moment-angle manifold $\zk$ can be identified with the quotient $U_\C(\sK)/V$ by the \emph{real} exponential action
\begin{equation}\label{ractz}
\begin{aligned}
  V\times\C^m&\longrightarrow\C^m\\
  (\mb v,\mb z)&\mapsto \mb v\cdot\mb z=\bigl(e^{\langle\gamma_1,\mb
  v\rangle_\R}z_1,\ldots,e^{\langle\gamma_m,\mb v\rangle_\R}z_m\bigr).
\end{aligned}
\end{equation}
This is proved in exactly the same way as Theorem~\ref{quotrk}. (Note also that the moment-angle manifold $\zk$ is homeomorphic to the real moment-angle manifold $\mathcal R_{D(\sK)}$, where~$D(\sK)$ is the \emph{double} of~$\sK$, see~\cite{b-b-c-g15} or~\cite[Corollary~4.8.14]{bu-pa15}.)

Provided that $\dim V=k$ is even and a complex structure is chosen on~$V$, the quotient of the \emph{holomorphic} exponential action~\eqref{hact} can also be identified with the moment-angle manifold $\zk$, thereby endowing the latter with a complex structure:

\begin{theorem}\label{quotzk}
Assume that $\{\sK,\mathrm A\}$ is a triangulated configuration defining a complete simplicial fan $\Sigma$, and let $\widetilde V\times U_\C(\sK)\to U_\C(\sK)$ be the corresponding holomorphic exponential action. Then there is a homeomorphism 
\[
  U_\C(\sK)/\widetilde V\cong\zk.
\]
\end{theorem}
\begin{proof}
We have an inclusion of pairs $(D^2,S^1)\hookrightarrow(\C,\C^\times)$, which induces an inclusion 
$\zk=(D^2,S^1)^\sK\hookrightarrow(\C,\C^\times)^\sK=U_\C(\sK)$. So the required homeomorphism will follow from the fact that the $\widetilde V$-orbit of any point $\mb z=(z_1,\ldots,z_m)\in U_\C(\sK)$ intersects $\zk$ at a single point.

We may assume that $z_i\ne0$ for $i\in[m]$, as in the proof of Theorem~\ref{comcom}. Given $\mb z=(z_1,\ldots,z_m)\in(\R^\times)^m$, define 
\[
  \ell(\mb z)=\sum_{i=1}^m\mb a_i\log|z_i|\in W^*.
\]
The map $\ell$ is constant on $\widetilde V$-orbits, since for any $\mb v\in\widetilde V$ we have
\[
  \ell(\mb v\cdot\mb z)=\sum_{i=1}^m\mb a_i
  \log\bigl(e^{\langle\gamma_i,\mb v\rangle_\R}|z_i|\bigr)=
  \sum_{i=1}^m\mb a_i\langle\gamma_i,\mb v\rangle_\R+
  \sum_{i=1}^m\mb a_i\log|z_i|=\ell(\mb z),
\]
where the last equality follows by Gale duality.
As $\Sigma$ is complete, there is $I\in\sK$ with $|I|=\dim W^*$ such 
that $-\ell(\mb z)\in\cone\mathrm A_I$. Then $\Gamma_{\widehat I}=\{\gamma_i\colon i\notin I\}$ is a (real) basis of~$V^*$. By solving the equations 
$\langle \gamma_i ,\mb v\rangle_\R = -\log|z_i|$, $i \notin I$, we find 
$\widetilde{\mb z}=\mb v\cdot\mb z$ in the same $\widetilde V$-orbit satisfying 
$|\widetilde z_i|=1$ for $i\notin I$. Since $-\ell(\widetilde{\mb z})=-\ell(\mb z)\in\cone\mathrm A_I$, we have $\log|\widetilde z_i|\le0$ for $i\in I$. Hence, 
$|\widetilde z_i|\le1$ for $i\in I$. These conditions imply that $\widetilde{\mb z}\in\zk$. Therefore, every $\widetilde V$-orbit intersects~$\zk$.

Now suppose that $\mb z,\mb z'\in\zk$ are in the same $\widetilde V$-orbit, i.\,e. 
$\mb z'=\mb v\cdot\mb z$. By definition of 
$\zk$, there is $I\in\sK$ such that $|z_i|=1$ for $i\notin I$ and $|z_i|\le1$ for $i\in I$, and $J\in\sK$ such that 
$|z'_j|=1$ for $j\notin J$ and $|z'_j|\le1$ for $j\in J$. We have
\[
  \sum_{i\in I}\mb a_i\log|z_i|=\ell(\mb z)=\ell(\mb z')=
  \sum_{j\in J}\mb a_j\log|z'_j|. 
\]
It follows that the expression above lies in $-\cone\mathrm A_I\cap\cone\mathrm A_J=-\cone\mathrm A_{I\cap J}$. As the latter is a simplicial cone, we obtain $|z_i|=|z'_i|=e^{\langle\gamma_i,\mb v\rangle_\R}|z_i|$ for $i\in[m]$. This implies $\langle\gamma_i,\mb v\rangle_\R=0$ for $i\in[m]$. Hence, $\mb v=0$ and 
$\mb z=\mb z'$, as needed.
\end{proof}

\begin{corollary}\label{zkfancplx}
Suppose $\sK$ is the underlying complex of a complete $n$-dimensional simplicial fan $\Sigma$ and $m-n=2\ell$ is even. Then
\begin{itemize}
\item[(1)] the moment-angle manifold $\zk$ has a structure of a compact complex manifold;
\item[(2)] $U_\C(\sK)$ and $\zk$ have the same homotopy type.
\end{itemize}
\end{corollary}

\begin{construction}[adding a ghost vertex]\label{addgv}
Given a simplicial complex $\sK$ on $[m]$, let $\sK_+$ be the same collection of subsets, but viewed as a simplicial complex on $[m+1]$. That is, $\sK_+$ is obtained from $\sK$ by adding the ghost vertex $\{m+1\}$. Then $\mathcal Z_{\sK_+}=\zk\times S^1$ and $U_\C(\sK_+)=U_\C(\sK)\times\C^\times$. 

Given a spanning vector configuration $\mathrm A=\{\mb a_1,\ldots,\mb a_m\}$ in~$W^*$, consider a configuration $\mathrm A_+=\{\mb a_1,\ldots,\mb a_m,\mb a_{m+1}\}$ obtained by appending an arbitrary vector $\mb a_{m+1}\in W^*$. The space of linear relations between the vectors of $\mathrm A_+$ is $V_+=V\oplus\R\langle\mb v_+\rangle$, where $V$ is the space of relations between $\mathrm A$ and $\mb v_+$ corresponds to a new relation $r_1\mb a_1+\cdots+r_m\mb a_m+\mb a_{m+1}=\bf0$. (For instance, if $\mb a_{m+1}=\bf0$, then we may take $r_1=\cdots=r_m=0$.) The Gale dual vector configuration $\Gamma_+$ in $V_+^*$ is
\[
  \Gamma_+=\{\gamma^+_1,\ldots,\gamma^+_m,\gamma^+_{m+1}\}=
  \Bigl\{\Bigl(\begin{matrix}\gamma_1\\r_1\end{matrix}\Bigr),\ldots,
  \Bigl(\begin{matrix}\gamma_m\\r_m\end{matrix}\Bigr),
  \Bigl(\begin{matrix}\bf0\\1\end{matrix}\Bigr)\Bigr\}.
\]
The exponential action~\eqref{ractz} of $V_+$ on $\C^{m+1}$ is given by
\[
\begin{aligned}
  V_+\times\C^{m+1}&\longrightarrow\C^{m+1}\\
  (\mb v,t,\mb z)&\mapsto (\mb v,t)\cdot\mb z=\bigl(e^{\langle\gamma_1,\mb
  v\rangle+r_1t}z_1,\ldots,
  e^{\langle\gamma_m,\mb v\rangle+r_mt}z_m,e^tz_{m+1}\bigr).
\end{aligned}
\]
Clearly, if $\{\sK,\mathrm A\}$ is a triangulated configuration defining an $n$-dimensional fan~$\Sigma$, then $\{\sK_+,\mathrm A_+\}$ defines the same fan. In this case, the exponential action of $V_+$ above restricted to $U_\C(\sK_+)=U_\C(\sK)\times\C^\times$ is free and proper. When the fan is complete, the quotient $U_\C(\sK_+)/V_+$ is homeomorphic to  $\mathcal Z_{\sK_+}=\zk\times S^1$.

Now if $m-n$ is odd, then $m+1-n$ is even, and therefore $\mathcal Z_{\sK_+}=\zk\times S^1$ has a structure of a compact complex manifold by Corollary~\ref{zkfancplx}.
\end{construction}

Elaborating on this further, we get the following generalisation of Example~\ref{2gamma}.

\begin{example}[holomorphic tori]\label{exhtorus}
Let $\sK=\{\varnothing\}$ be the simplicial complex  on $[m]$ consisting of 
$\varnothing$ only, that is, with $m$ ghost vertices.  Let $\mathrm A$ be a configuration of $m$ zero vectors in $W^*=\{\bf 0\}$. The Gale dual configuration 
$\Gamma=\{\gamma_1,\ldots,\gamma_m\}$ is a basis in $V^*\cong\R^m$. The quotient of the real exponential action~\eqref{ractz} of $V$ on $U_\C(\sK)=(\C^\times)^m$ is $\zk=T^m$, the compact $m$-torus. 

Now consider holomorphic exponential actions. We need to specify a complex structure $\mathcal J$ on $V\cong\R^m$, so that $m$ must be even, $m=2\ell$. Then $\widetilde V=(V,\mathcal J)$ has complex dimension~$\ell$. In order to describe the quotient $(\C^\times)^m/\widetilde V$ of the holomorphic exponential action, consider the exact sequence of $\C$-linear maps
\[
  0\longrightarrow\widetilde V\stackrel\iota\longrightarrow
  \C^m\cong\widetilde V\oplus\overline{\widetilde V}
  \stackrel p\longrightarrow \C^m/\widetilde V\cong  \overline{\widetilde V}
  \longrightarrow 0
\]
Here $\iota$ denotes the map $\mb v\mapsto(\langle\gamma_1,\mb v\rangle_\C,\ldots,\langle\gamma_m,\mb v\rangle_\C)$, and $p$ denotes the quotient projection. We have $\widetilde V_\C\cong\widetilde V\oplus\overline{\widetilde V}\cong\C^m$, where 
$\overline{\widetilde V}=(V,-\mathcal J)$, the map 
$\iota$ is identified through this isomorphism with the inclusion of the first summand, and $p$ with the second projection. For 
$\Gamma=\{\gamma_1,\ldots,\gamma_m\}$,  let $\mb g_1,\ldots,\mb g_m$ denote the $\R$-dual basis of~$V$, that is, 
$\langle\gamma_k,\mb g_l\rangle_\R=\delta_{kl}$. Then $\mb v=
\sum_{k=1}^m\langle\gamma_k,\mb v\rangle_\R\mb g_k$, and the map $p$ is given on the standard basis vectors by $\mb e_k\mapsto\mb g_k$, $i\mb e_k\mapsto-\mathcal J\mb g_k$. Now we have
\[
  (\C^\times)^m/\widetilde V%=(\exp\C^m)/\widetilde V 
  \cong(\C^m/2\pi i\Z^m)/\widetilde V\cong \overline{\widetilde V}/p(2\pi i\Z^m)
  \cong\overline{\widetilde V}/\Z\langle\mathcal J\mb g_1,\ldots,
  \mathcal J\mb g_m\rangle.
\]
This is a holomorphic torus, the quotient of an $\ell$-dimensional complex space by the lattice generated by the $\R$-basis $\mathcal J\mb g_1,\ldots,\mathcal J
\mb g_{2\ell}$. Every holomorphic torus is obtained from some $\widetilde V=(V,\mathcal J)$ in this way.
\end{example}

\begin{proposition}
A complex moment-angle manifold $\mathcal Z_K$ is non-K\"ahler, unless it is a holomorphic torus.
\end{proposition}
\begin{proof}
Suppose that $\mathcal Z_K$ is K\"ahler of complex dimension~$d$. Then the cohomology class $[\omega]\in H^2(\mathcal Z_K)$ of the K\"ahler form satisfies $[\omega]^d\ne0$, as $\mathcal Z_K$ is compact.
 
If $\mathcal K$ does not have ghost vertices, then it is easy to see using~\eqref{madec} and the appropriate cell decomposition that $\mathcal Z_K$ is 2-connected (see~\cite[Proposition~4.3.5]{bu-pa15}). In general, $\mathcal Z_K$ is diffeomorphic to a product of a 2-connected manifold and a torus (see Construction~\ref{addgv}). Hence, it can have a 2-dimensional cohomology class satisfying $[\omega]^d\ne0$ only if the 2-connected factor is a point, so $\mathcal K$ consists entirely of ghost vertices and $\mathcal Z_K$ is a $d$-dimensional holomorphic torus.
\end{proof}

When $\Sigma$ is the normal fan of a polytope, the quotient $U_\C(\sK)/\widetilde V\cong\zk$ can be realised as a nondegenerate intersection of \emph{Hermitian} quadrics:

\begin{theorem}\label{hquadrics}
Let 
\[
  P=\{\mb w\in W\colon\langle \mb a_i,\mb w\rangle+b_i\ge0,\quad i=1,\ldots,m\}
\]  
be a generic polyhedron with normal simplicial fan $\Sigma=\{\cone \mathrm A_I\colon I\in\mathcal K_P\}$. Let $\Gamma=\{\gamma_1,\ldots,\gamma_m\}$ be the Gale dual configuration of $\mathrm A=\{\mb a_1,\ldots,\mb a_m\}$. Put $\delta=\sum_{i=1}^mb_i\gamma_i$. 

The holomorphic exponential action $\widetilde V\times U_\C(\sK_P)\to U_\C(\sK_P)$ is free and proper, and the quotient $U_\C(\sK_P)/\widetilde V$ is diffeomorphic to the nondegenerate intersection of $k$ Hermitian quadrics in~$\C^m$:
\[
  U_\C(\sK_P)/\widetilde V\cong\{(z_1,\ldots,z_m)\in\C^m\colon
  \gamma_1 |z_1|^2+\cdots+\gamma_m|z_m|^2=\delta\}.
\]
\end{theorem}
\begin{proof}
The proof is similar to the real case (Theorem~\ref{quadrics}). We consider the map
\[
  \mu_\Gamma\colon\C^m\to V^*, \quad 
  (z_1,\ldots,z_m)\mapsto |z_1|^2\gamma_1+ \cdots+|z_m|^2\gamma_m,
\]
show that $\delta$ is a regular value of $\mu_\Gamma$ and $\mu_\Gamma^{-1}(\delta)$ is a nondegenerate intersection of Hermitian quadrics that meets every 
$\widetilde V$-orbit of $U_\C(\sK_P)$ at a single point.
\end{proof}

\section{Gale duality for rational configurations}\label{secgr}
A vector configuration $\widehat\Gamma=\{\widehat\gamma_1,\ldots,\widehat\gamma_m\}$ in a real vector space~$Q^*$ is called \emph{rational} if its 
$\Z$-span is a discrete subset of~$Q^*$. We continue to assume that the $\R$-span of $\widehat\Gamma$ is~$Q^*$. If $\widehat\Gamma$ is a rational configuration, then $L^*=\Z\langle\widehat\Gamma\rangle$ is a full-rank lattice in $Q^*$, that is, a discrete finitely generated free abelian subgroup of rank equal to $\dim Q^*$.

There is a Gale duality between rational configurations, which involves lattices alongside with vector spaces. All  related notation is conveniently included in a pair of diagrams of dual short exact sequences
\[
\begin{tikzcd}[row sep=small]
  0 \ar{r} & U \ar{r}{\widehat{A}^*} 
  & \R^m  \ar{r}{\widehat\varGamma} & Q^* \ar{r} & 0\\
  0 \ar{r} & M \ar{r} \ar[hook]{u} & \Z^m  \ar{r}\ar[hook]{u}  
  & L^* \ar{r}\ar[hook]{u} & 0
\end{tikzcd}
\]
and
\[
\begin{tikzcd}[row sep=small]
  0 \ar{r} & Q \ar{r}{\widehat{\varGamma}^*} 
  & \R^m  \ar{r}{\widehat A} & U^* \ar{r} & 0\\
  0 \ar{r} & L \ar{r} \ar[hook]{u} & \Z^m  \ar{r}\ar[hook]{u}  
  & M^* \ar{r}\ar[hook]{u} & 0
\end{tikzcd}
\]
Here $\widehat\varGamma$ takes $\mb e_i$ to $\widehat\gamma_i$ and $U=\Ker\widehat\varGamma$. The Gale dual of $\widehat\Gamma$ is a rational vector configuration $\widehat{\mathrm A}=\{\widehat{\mb a}_1,\ldots,\widehat{\mb a}_m\}$ in $U^*$, and $M^*=\Z\langle\widehat{\mathrm A}\rangle$ is a full-rank lattice in~$U^*$.

For rational configurations, we have the following analogue of Proposition~\ref{galeindspan}.

\begin{proposition}\label{rgaleindspan}
For any $I\subset[m]$, the rational subconfiguration $\widehat{\mathrm{A}}_I$ is part of a basis of the lattice $M^*$ if and only if $\widehat{\Gamma}_{\widehat I}$ spans the lattice~$L^*$.  
\end{proposition}
\begin{proof}
Consider the coordinate sublattices $\Z^I=\Z\langle\mb e_i\colon i\in I\rangle$ and
$\Z^{\widehat I}=\Z\langle\mb e_j\colon j\notin I\rangle$. Then 
$\widehat{\mathrm{A}}_I$ is part of a basis of $M^*$ if and only if the composite 
$\Z^I\to\Z^m\xrightarrow{\widehat A} M^*$ is split injective. This is equivalent to 
$M\xrightarrow{{\widehat A}^*}\Z^m\to\Z^I$ being surjective. Now consider the diagram
\[
\begin{tikzcd}[row sep=small]
  &&0 \ar{d}\\
  &&\Z^{\widehat I} \ar{d}\\
  0\ar{r} & M \ar{r}{{\widehat A}^*} & \Z^m \ar{d} \ar{r}{\widehat{\varGamma}}
  & L^*\ar{r} & 0\\
  &&\Z^I \ar{d}\\
  &&0
\end{tikzcd}
\]
A simple diagram chase (see~\cite[Lemma~1.2.5]{bu-pa15}) shows that $M\xrightarrow{{\widehat A}^*}\Z^m\to\Z^I$ is surjective if and only if
$\Z^{\widehat I}\to\Z^m\xrightarrow{\widehat\varGamma}L^*$ is surjective. The latter is equivalent to the condition that $\widehat{\Gamma}_{\widehat I}$ spans~$L^*$.
\end{proof}

Now we return to the general setup with a pair of Gale dual vector configurations 
$\mathrm A=\{\mb a_1,\ldots,\mb a_m\}$ and $\Gamma=\{\gamma_1,\ldots,\gamma_m\}$ in $W^*$ and $V^*$, respectively. Recall that $V$ can be thought of as the space of linear relations between the vectors $\mb a_1,\ldots,\mb a_m$, and $\gamma_k\in V^*$ maps a linear relation to its $k$th coefficient. This defines the map $\varGamma^*\colon V\to\R^m$, $\mb v\mapsto(\langle\gamma_1,\mb v\rangle,\ldots,\langle\gamma_m,\mb v\rangle)$.

The ``degree of  irrationality'' of a configuration $\Gamma$ can be measured by the following construction.

\begin{construction}\label{coras}
A subspace $Q\subset V$ is called \emph{rational} (with respect to~$\Gamma$) if it is generated by relations with integer coefficients, that is, if $Q\subset\R\langle(\varGamma^{*})^{-1}(\Z^m)\rangle$. Here
$(\varGamma^{*})^{-1}(\Z^m)$ is a lattice in $V$ of rank between $0$ and $k=\dim V$, and $\R\langle(\varGamma^{*})^{-1}(\Z^m)\rangle$ is the largest rational subspace in~$V$.

A rational subspace $Q\subset V$ contains a full-rank lattice
\begin{equation}\label{Llatt}
  L=(\varGamma^{*})^{-1}(\Z^m)\cap Q=
  \{\mb q\in Q\colon\langle\gamma_k,\mb q\rangle\in\Z\;
  \text{ for }k=1,\ldots,m\}.
\end{equation}

The image of $\Gamma$ under the map of dual spaces $V^*\to Q^*$ defines a rational vector configuration $\widehat\Gamma=\{\widehat\gamma_1,\ldots,\widehat\gamma_m\}$ in~$Q^*$. The $\Z$-span of $\widehat\Gamma$ is a full-rank lattice in $Q^*$, namely, $\Z\langle\widehat\Gamma\rangle=L^*$. 
We have $\widehat{\varGamma}^*=\varGamma^*|_Q$, and the Gale dual rational configuration $\widehat{\mathrm A}$ is defined by the middle line in the commutative diagram
\[
\begin{tikzcd}[row sep=small]
  0 \ar{r} & L \ar{r} \ar[hook]{d} & \Z^m  \ar{r}\ar[hook]{d}  & 
  M^* \ar{r} \ar[hook]{d} & 0\\
  0 \ar{r} & Q \ar{r}{\widehat{\varGamma}^*} \ar[hook]{d} & 
  \R^m \ar{r}{\widehat A} \ar[equal]{d} & U^* \ar{r} \ar{d} & 0\\
  0 \ar{r} & V \ar{r}{\varGamma^*} & \R^m \ar{r}{A} & W^* \ar{r} & 0
\end{tikzcd}
\]
namely, $\widehat{\mb a}_i=\widehat A(\mb e_i)$. The dual commutative diagram is
\[
\begin{tikzcd}[row sep=small]
  0 \ar{r} & M \ar{r} \ar[hook]{d} & \Z^m  \ar{r}\ar[hook]{d}  & 
  L^* \ar{r} \ar[hook]{d} & 0\\
  0 \ar{r} & U \ar{r}{\widehat{A}^*}  & 
  \R^m \ar{r}{\widehat\varGamma} & Q^* \ar{r} & 0\\
  0 \ar{r} & W \ar{r}{A^*} \ar[hook]{u} & \R^m \ar{r}{\varGamma}\ar[equal]{u} 
  & V^* \ar{r}\ar{u} & 0
\end{tikzcd}
\]
\end{construction}

A cone in a vector space $U^*$ with a full-rank lattice $M^*$ is called \emph{rational} (with respect to~$M^*$) if it is generated by a set of vectors in~$M^*$. A rational cone in \emph{nonsingular} if it is generated by part of a basis of~$M^*$. A nonsingular cone is necessarily simplicial. A fan is \emph{rational} (respectively, \emph{nonsingular}) if it is consists of rational (respectively, nonsingular) cones. 

A configuration $\mathrm A=\{\mb a_1,\ldots,\mb a_m\}$ of vectors in $U^*$ is \emph{rational} if each $\mb a_i$ belongs to the lattice~$M^*$.
A triangulated configuration $\{\mathcal K,\mathrm A\}$ is \emph{nonsingular} if $\{\mb a_i\colon i\in I\}$ is part of a basis of $M^*$ for any $I\in\mathcal K$. A nonsingular triangulated configuration $\{\mathcal K,\mathrm A\}$ defines a nonsingular fan $\Sigma=\{\cone\mathrm A_I\colon I\in\mathcal K\}$ in which $\mb a_i$ is a primitive generator whenever $\{i\}\in\mathcal K$.

\begin{proposition}\label{ahatfan}
Let $\Gamma$ and $\mathrm A$ be a pair of Gale dual vector configurations,  $Q\subset V$ a rational subspace, $\widehat\Gamma$ and $\widehat{\mathrm A}$ the corresponding rational configurations in $Q^*$ and $U^*$, and $\sK$ a simplicial complex on~$[m]$. Assume that $\{\sK,\mathrm A\}$ is a triangulated configuration defining a simplicial fan~$\Sigma$ in~$W^*$. Then 
\begin{itemize}
\item[(1)]
$\{\sK,\widehat{\mathrm A}\}$ is a rational triangulated configuration defining a rational simplicial fan $\widehat\Sigma$ in~$U^*$. 

\item[(2)] The surjections $\R^m\xrightarrow{\widehat A} U^*\to W^*$
induce maps of fans $\Sigma_\sK\to\widehat\Sigma\to\Sigma$ which restrict to linear isomorphisms on each cone.
\end{itemize}
\end{proposition}
\begin{proof}
Since $\{\sK,\mathrm A\}$ is a triangulated configuration, we have $(\relint\cone\mathrm A_I)\cap(\relint\cone\mathrm A_J)=\varnothing$ for any $I,J\in\sK$, $I\ne J$. By Proposition~\ref{positive}, the surjective map $U^*\to W^*$ takes $\relint\cone\widehat{\mathrm A}_I$ to $\relint\cone\mathrm A_I$. Hence, $(\relint\cone\widehat{\mathrm A}_I)\cap(\relint\cone
\widehat{\mathrm A}_J)=\varnothing$ for any $I,J\in\sK$, $I\ne J$. It follows that
the data $\{\sK,\widehat{\mathrm A}\}$ is also a triangulated configuration, proving~(1).  Statement~(2) is clear, as all the fans involved are simplicial.
\end{proof}

For a subset $I\subset[m]$, define a subgroup of~$Q$ by
\[
  L_I
  =\{\mb q\in Q\colon\langle\gamma_k,\mb q\rangle\in\Z\;\text{ for }k\notin I\}.
\]
Note that $L_I\supset L$ and $L_{\varnothing}=L$, see~\eqref{Llatt}.

\begin{proposition}\label{nsdual}
A triangulated configuration $\{\sK,\widehat{\mathrm A}\}$ is nonsingular if and only if $L_I=L$ for any $I\in\sK$.
\end{proposition}
\begin{proof}
We may write $L_I=\{\mb q\in Q\colon\langle\widehat\gamma_k,\mb q\rangle\in\Z\;\text{ for }k\notin I\}$. 

Assume that $L_I=L$  for any $I\in\sK$. Then $\widehat\Gamma_{\widehat I}=\{\widehat\gamma_k\colon k\notin I\}$ spans $Q^*$. Indeed, otherwise there is $\mb q\in Q$ such that $\langle\widehat\gamma_k,\mb q\rangle=0$ for $k\notin I$, so that $L_I$ contains the subspace $\R\langle\mb q\rangle$. On the other hand, $L$ is a lattice and cannot contain a subspace. Now $\R\langle\widehat\Gamma_{\widehat I}\rangle=Q^*$ implies that $\Z\langle\widehat\Gamma_{\widehat I}\rangle$ is a full-rank lattice in $Q^*$ whose dual lattice is~$L_I$. We obtain 
$\Z\langle\widehat\Gamma_{\widehat I}\rangle
=L_I^*=L^*$. This implies, by Proposition~\ref{rgaleindspan}, that $\widehat{\mathrm A}_I$ is part of basis of $M^*$ for any $I\in\sK$. Hence, the triangulated configuration $\{\sK,\widehat{\mathrm A}\}$ is nonsingular.

Now assume that $\{\sK,\widehat{\mathrm A}\}$ is a nonsingular triangulated configuration. Then 
$\Z\langle\widehat\Gamma_{\widehat I}\rangle
=L^*=\Z\langle\widehat\Gamma\rangle$ for any $I\in\sK$  by Proposition~\ref{rgaleindspan}. It follows that each $\widehat\gamma_i$, $i=1,\ldots,m$, is an integral linear combination of $\{\widehat\gamma_k\colon k\notin I\}$, so that $L_I=L$.
\end{proof}

\section{Partial quotients and torus-exponential actions}\label{sectea}
A \emph{partial quotient} of a moment-angle manifold $\zk$ is the quotient $\zk/H$ by a freely acting subtorus $H\subset T^m$. When $\{\mathcal K,\mathrm A\}$ is triangulated configuration and the subtorus $H$ is defined by a rational subspace $Q\subset V$ of even codimension, the partial quotient $\zk/H$ admits a complex structure in which the quotient torus $T^m/H$ acts by holomorphic transformations. Such a holomorphic partial quotient is the quotient of the open subset $U_\C(\sK)$ by a \emph{torus-exponential action} of the product of an algebraic torus and a complex space. In addition to complex moment-angle manifolds, this construction includes LVMB manifolds and complete nonsingular toric varieties as special cases.

\subsection{Torus-exponential actions and their quotients}
The lattice $L$ \eqref{Llatt} in a rational subspace $Q\subset V$ defines an algebraic torus 
\[
  \C^\times_L=L\otimes_\Z\C^\times\cong Q_\C/(2\pi i L),
\]  
where $Q_\C=Q\otimes_\R\C=Q\oplus iQ$ denotes the complexification of~$Q$,  and a compact torus
\[
  T_L=L\otimes_\Z S^1 \cong Q/(2\pi L).
\]  
Both $Q$ and $T_L$ are Lie subgroups of $\C^\times_L$ (given by $L\otimes_\Z\R$ and $L\otimes_\Z S^1$, respectively) and we have a product decomposition $\C^\times_L=Q\times T_L$ as Lie groups.

Given a vector configuration $\Gamma=\{\gamma_1,\ldots,\gamma_m\}$ in $V$,
the map $\varGamma^*|_L\otimes_\Z\C^\times$ embeds $\C^\times_L$ as a closed  subgroup in~$(\C^\times)^m$:
\begin{multline}\label{cqsub}
  (\varGamma^*|_L\otimes_\Z\C^\times)
  (\C^\times_L)=\exp\varGamma^*_\C(Q_\C)\\
  =
  \bigl\{(e^{\langle\gamma_1,\mb v\rangle+i\langle\gamma_1,\mb u\rangle},
  \ldots,e^{\langle\gamma_m,\mb v\rangle+i\langle\gamma_m,\mb u\rangle})
  \in(\C^\times)^m  
  \colon \mb v+i\mb u\in Q_\C\bigr\},
\end{multline}
where $\varGamma^*_\C\colon V_\C\to\C^m$ is the complexification of~$\varGamma^*$ and $\exp\colon\C^m\to(\C^\times)^m$ is the component-wise exponential map. Similarly, $\varGamma^*|_L\otimes_\Z S^1$ embeds $T_L$ as a closed subgroup in~$T^m=(S^1)^m$:
\[
  (\varGamma^*|_L\otimes S^1)(T_L)  = 
  \bigl\{(e^{i\langle\gamma_1,\mb q\rangle},
  \ldots,e^{i\langle\gamma_m,\mb q\rangle})
  \in T^m  \colon \mb q\in Q\bigr\}.
\]
The subgroup of $T^m$ on the right hand side above is closed precisely because $Q$ is a rational subspace.  

We define the Lie group $\C^\times_L\times_Q V$ as the quotient of 
$\C^\times_L\times V$ by the subgroup consisting of pairs $(j(-\mb q),\mb q)$, 
where $\mb q\in Q$ and $j\colon Q\to\C^\times_L$ is the inclusion. There are two product decompositions:
\[
  \C^\times_L\times_Q V\cong T_L\times V\cong\C^\times_L\times(V/Q).
\]
Here the first isomorphism is canonical, while the isomorphism 
$\C^\times_L\times(V/Q)\cong\C^\times_L\times_Q V$ is defined by a splitting $V/Q\to V$.

Both $\C^\times_L$ and $V$ embed as closed subgroups in~$(\C^\times)^m$, as 
$\exp\varGamma^*_\C(Q_\C)$ and $\exp\varGamma^*(V)$, respectively, see~\eqref{cqsub}. The intersection of these two subgroups is
$\exp\varGamma^*(Q)$, so $\C^\times_L\times_Q V$ also embeds as a closed subgroup in~$(\C^\times)^m$. We therefore obtain an action of 
$\C^\times_L\times_Q V$ on $\C^m$ by restricting the coordinate-wise action 
of~$(\C^\times)^m$. Using the canonical decomposition $\C^\times_L\times_Q V\cong T_L\times V$, this action can be written explicitly as
\begin{equation}\label{teact}
\begin{aligned}
  (T_L\times V)\times\C^m&\longrightarrow\C^m\\
  (\mb q,\mb v,\mb z)&\mapsto (\mb q,\mb v)\cdot\mb z=
  \bigl(e^{\langle\gamma_1,\mb v\rangle+i\langle\gamma_1,\mb q\rangle}z_1,
  \ldots,e^{\langle\gamma_m,\mb v\rangle+i\langle\gamma_m,\mb q\rangle}
  z_m\bigr).
\end{aligned}
\end{equation}
We refer to it as the \emph{torus-exponential action}.

\begin{proposition}\label{stabt}
For a point $\mb z=(z_1,\ldots,z_m)\in\C^m$, let $I=\{i\in[m]\colon z_i=0\}$.
The stabiliser of $\mb z$ under the action of~$T_L$ given by~\eqref{teact} is 
$(T_L)_{\mb z}=2\pi(L_I/L)$. 
\end{proposition}
\begin{proof}
For $\mb q\in T_L=Q/(2\pi L)$, we have
$
  \mb q\cdot\mb z=(e^{i\langle\gamma_1,\mb q\rangle}z_1,
  \ldots,e^{i\langle\gamma_m,\mb q\rangle}z_m)
$.
Hence, $\mb q\cdot\mb z=\mb z$ if and only if $\langle\gamma_k,\mb q\rangle\in 2\pi\Z$ for $k\notin I$, that is, $\mb q\in 2\pi(L_I/L)$.
\end{proof}

An action of a Lie group is \emph{almost free} if all stabilisers are discrete subgroups.

\begin{theorem}\label{teproper}
Let $\Gamma$ and $\mathrm A$ be a pair of Gale dual vector configurations,  $Q\subset V$ a rational subspace, and $\sK$ a simplicial complex on~$[m]$. The torus-exponential action $(\C^\times_L\times_Q V)\times U_\C(\sK)\to U_\C(\sK)$ is proper if and only if $\{\sK,\mathrm A\}$ is a triangulated configuration, in which case the action is almost free with finite stabilisers.
\end{theorem}
\begin{proof}
Since $T_Q$ is a compact group, the action of $\C^\times_L\times_Q V\cong T_Q\times V$ is proper if and only if the action of $V$ is proper. Hence, the assertion about the properness follows from Theorem~\ref{proper}. It also gives that the action of $V$ is free. For the action of $T_L$, the stabiliser of $\mb z\in U_\C(\sK)$ is $2\pi(L_I/L)$. Since $I\in\sK$, the set $\{\gamma_k\colon k\notin I\}$ spans~$V^*$, so $L_I$ is a full-rank lattice in~$Q$. It follows that $L_I/L$ is a finite group.
\end{proof}

%The group $T_Q\times V$ can be also decomposed as $\C^\times_Q\times V/Q$.

\begin{proposition}
Let $\Gamma$ and $\mathrm A$ be a pair of Gale dual vector configurations,  $Q\subset V$ a rational subspace, and $\sK$ a simplicial complex on~$[m]$. If 
$\{\sK,\mathrm A\}$ is a triangulated configuration, then the action of $\C^\times_L$ on $U_\C(\sK)$ is proper.
\end{proposition}
\begin{proof}
As a Lie group, $\C^\times_L$ is the product $Q\times T_L$ and $T_L$ is compact. Hence, the action of $\C^\times_L$ is proper if and only if the action of $Q$ on $U_\C(\sK)$ given by
\begin{equation}\label{qact}
  (\mb q,\mb z)\mapsto \mb q\cdot\mb z=
  \bigl(e^{\langle\gamma_1,\mb q\rangle}z_1,
  \ldots,e^{\langle\gamma_m,\mb q\rangle}z_m\bigr).
\end{equation}
is proper. This is the exponential action defined by $\widehat\Gamma=\{\gamma_1,\ldots,\gamma_m\}$ viewed as a vector configuration in the quotient space $Q^*$ of~$V^*$. Let $\widehat{\mathrm A}$ be the Gale dual configuration. Then $\{\sK,\widehat{\mathrm A}\}$ is a triangulated configuration by Proposition~\ref{ahatfan}. It follows that the action is proper, by the criterion of Theorem~\ref{proper}.
\end{proof}

The quotient $U_\C(\sK)/\C^\times_L$ by a proper algebraic action of 
$\C^\times_L$ is the \emph{toric variety} $X_{\widehat\Sigma}$ corresponding to the rational simplicial fan~$\widehat\Sigma$ defined by $\{\sK,\widehat{\mathrm A}\}$~\cite[\S5.1]{c-l-s11}. 

The quotient $U_\C(\sK)/(T_L\times V)$ by a proper torus-exponential action is identified with the quotient $\zk/T_L$, where $\zk$ is the moment-angle manifold (see Section~\ref{compzk}). If the action of  $T_L$ is free, then $\zk/T_L$ is a partial quotient; in general it has orbifold singularities.
On the other hand, the quotient $U_\C(\sK)/(T_L\times V)=U_\C(\sK)/(\C^\times_L\times V/Q)$ is identified with the quotient of the toric variety $X_{\widehat\Sigma}$ by the proper action of~$V/Q$. This is described by the commutative square of quotient maps:
\[
\begin{tikzcd}
  U_\C(\sK) \ar{r}{/\C^\times_L} \ar{d}{/V} 
   & X_{\widehat\Sigma} \ar{d}{/(V/Q)}\\
  \zk \ar{r}{/T_L} & \zk/T_L=X_{\widehat\Sigma}/(V/Q)
\end{tikzcd}
\]

\subsection{Holomorphic torus-exponential actions}
The holomorphic version of the torus-exponential action is defined when the dimension of $V/Q$ is even. Let $\mathcal J\colon V/Q\to V/Q$ be a complex structure, and denote the complex space $(V/Q,\mathcal J)$ by $\widetilde{V/Q}$. Choose a splitting $s\colon V/Q\to V$ and consider the dual projection $s^*\colon V^*\to\Ann Q$, where $\Ann Q\cong(V/Q)^*$ denotes the annihilator of $Q$ in~$V^*$. For a spanning vector configuration $\Gamma=\{\gamma_1,\ldots,\gamma_m\}$ in $V^*$, let $\gamma'_i=s^*(\gamma_i)$.

The space $(V/Q)^*$ with the conjugate complex structure $\mathcal J'$ is denoted by~$(\widetilde{V/Q})^*$. There is a $\C$-bilinear pairing $(\widetilde{V/Q})^*
\times\widetilde{V/Q}\to\C$ given by
\[
  \langle\gamma',\mb v'\rangle_\C=\langle\gamma',\mb v'\rangle_\R-
  i\langle\gamma',\mathcal J\mb v'\rangle_\R,
\]
and the holomorphic exponential action
\[
  \widetilde{V/Q}\times\C^m\longrightarrow\C^m,\qquad
  (\mb v',\mb z)\mapsto \bigl(e^{\langle\gamma'_1,\mb
  v'\rangle_\C}z_1,\ldots,e^{\langle\gamma'_m,\mb v'\rangle_\C}z_m\bigr).
\]
This together with the holomorphic action of $\C^\times_L$ defines the \emph{holomorphic torus-exponential action} of $\C^\times_L\times\widetilde{V/Q}$.

\begin{proposition}\label{holpq1}
Let $\Gamma$ and $\mathrm A$ be a pair of Gale dual vector configurations,  $Q\subset V$ a rational subspace, $\widehat\Gamma$ and $\widehat{\mathrm A}$ the corresponding rational configurations in $Q^*$ and $U^*$, and $\sK$ a simplicial complex on~$[m]$. Suppose that 
\begin{itemize}
\item[(a)]
$\{\sK,\mathrm A\}$ is a triangulated configuration defining a simplicial fan~$\Sigma$;

\item[(b)]
$\{\sK,\widehat{\mathrm A}\}$ is a nonsingular triangulated configuration defining a nonsingular simplicial fan $\widehat\Sigma$ in~$U^*$;

\item[(c)]
a complex structure is chosen on~$V/Q$ (in particular, its dimension is even).
\end{itemize}
Then the partial quotient $\zk/T_L$ has a structure of a compact complex manifold as the quotient $U_\C(\sK)/(\C^\times_L\times\widetilde{V/Q})=
X_{\widehat\Sigma}/(\widetilde{V/Q})$.
\end{proposition}
\begin{proof}
This follows by considering the diagram of quotient maps
\[
\begin{tikzcd}
  U_\C(\sK) \ar{r}{/\C^\times_L} \ar{d}{/V} 
   & X_{\widehat\Sigma} \ar{d}{/(\widetilde{V/Q})}\\
  \zk \ar{r}{/T_L} & \zk/T_L\cong X_{\widehat\Sigma}/(\widetilde{V/Q})
\end{tikzcd}
\]
Condition (a) implies that the action of $\C^\times_L\times\widetilde{V/Q}$ on $U_\C(\sK)$ is proper. Condition~(b) together with Propositions~\ref{nsdual} and~\ref{stabt} implies that the action of $T_L$ on $U_\C(\sK)$ is free, hence, the action of $\C^\times_L\times\widetilde{V/Q}$ is also free. Finally, the homeomorphism in the bottom right corner is proved in the same way as Theorem~\ref{quotzk} (which corresponds to the case $Q=\{\mathbf 0\}$).
\end{proof}

The torus-exponential action can be defined more canonically when a complex structure $\mathcal J$ is chosen on the whole of $V$, with $Q$ being a complex subspace. This, of course, induces a complex structure on $V/Q$. As usual, we denote $\widetilde V=(V,\mathcal J)$ and $\widetilde Q=(Q,\mathcal J)$. The \emph{holomorphic torus-exponential action} 
\[
  (\C^\times_L\times_{\widetilde Q}\widetilde V)\times\C^m\to\C^m
\]  
is defined as the product of holomorphic actions of $\C^\times_L$ and $\widetilde V$, which agree on the common holomorphic subgroup~$\widetilde Q$. Choosing a 
$\C$-linear splitting $\widetilde{V/Q}\to\widetilde V$ allows one to identify this action with the action of $\C^\times_L\times\widetilde{V/Q}$ considered above.

The torus $T_L$ acquires a holomorphic structure as the quotient $\C^\times_L/\widetilde Q$, which we denote by $\widetilde T_L$, see Example~\ref{exhtorus}. We have a commutative diagram of the quotient maps for holomorphic group actions:
\begin{equation}\label{diaghpq}
\begin{tikzcd}
  U_\C(\sK) \ar{r}{/\C^\times_L} \ar{d}{/\widetilde V} 
   & X_{\widehat\Sigma} \ar{d}{/(\widetilde{V/Q})}\\
  \zk \ar{r}{/\widetilde T_L} & \zk/\widetilde T_L\cong X_{\widehat\Sigma}/(\widetilde{V/Q})
\end{tikzcd}
\end{equation}
If the rational triangulated configuration $\{\sK,\widehat{\mathrm A}\}$ is nonsingular, then the actions of 
$\C^\times_L$ and $\widetilde T_L$ are free, and $\zk\to\zk/\widetilde T_L$ is a holomorphic torus principal bundle over the holomorphic partial quotient. More generally, $\zk\to\zk/\widetilde T_L$ is a holomorphic \emph{Seifert principal bundle}~\cite[\S5.1]{orli72} over the orbifold $\zk/\widetilde T_L\cong X_{\widehat\Sigma}/(\widetilde{V/Q})$, where the latter the quotient of the toric orbifold $X_{\widehat\Sigma}$ by a free proper action of~$\widetilde{V/Q}$.

\subsection{Maximal torus actions}
The quotient torus $T^m/T_L$ acts on a holomorphic partial quotient $\zk/T_L$ by holomorphic transformations. 

An action of a torus $T$ on a connected smooth manifold $M$ is \emph{maximal} if there is a point $x\in M$ such that $\dim T+\dim T_x=\dim M$, where $T_x$ is the stabiliser subgroup of~$x$. Examples of maximal torus actions are the action of a half-dimensional compact torus on a smooth complete toric variety, the free action of $T^m$ on itself and the action of $T^m$ on a moment-angle manifold~$\zk$. These are particular cases of the action of $T^m/T_L$ on a partial quotient~$\zk/T_L$:

\begin{proposition}
The action of the torus $T^m/T_L$ on a partial quotient~$\zk/T_L$ is maximal.
\end{proposition}
\begin{proof}
Write $\zk=\bigcup_{I\in\sK}(D^2,S^1)^I$ as in~\eqref{madec}, where 
$(D^2,S^1)^I=\prod_{i\in I}D^2\times\prod_{i\notin I}S^1$. Choose a maximal simplex $J\in\sK$, $|J|=n$, and a point $x\in\prod_{j\in J}0\times\prod_{j\notin J}S^1\subset\zk$. Writing $\zk\cong U_\C(\sK)/V$, we obtain that $x$ corresponds to the $V$-orbit of $\mb z\in U_\C(\sK)$ with $z_j=0$ for $j\in J$. The stabiliser of $x$ for the $T^m$-action on $\zk$ is the coordinate subtorus $T^J=\prod_{j\in J} S^1\subset \prod_{i=1}^m S^1= T^m$.

Let $G=T^m/T_L$ and let $\overline x\in\zk/T_L$ be the image of $x\in\zk$ in the quotient. We have $T^J\cap T_L=\{1\}$, since $T_L$ acts freely 
on~$\zk$. Hence, the stabiliser $G_{\overline x}$ is isomorphic to $T^J$ and is $n$-dimensional. We obtain $\dim G+\dim G_{\overline x}=m-\dim Q+n=\dim\zk-\dim Q=\dim\zk/T_L$, so the $G$-action on $\zk/T_L$ is maximal.
\end{proof}

Compact complex manifolds $M$ with maximal torus actions by holomorphic transformations are classified by Ishida in~\cite{ishi19}. The main result~\cite[Theorem~1.5]{ishi19} implies that every such $M$ is equivariantly biholomorphic to a holomorphic partial quotient $\zk/T_L$. (For the latter formulation, see~\cite[Theorem~4.8]{usti14}.) It also implies that any complex manifold $M$ with a maximal torus action is the base of a holomorphic torus principal bundle with total space a complex moment-angle manifold.

\smallskip

Two cases of holomorphic partial quotients $\zk/T_L$, corresponding to a one-dimensional rational subspace $Q\subset V$ and the whole of $V$ being rational, respectively, are of particular importance. They are considered in the next two sections.

\section{LVM and LVMB manifolds}\label{LVMBsec}
The study of holomorphic exponential actions led to the construction of a new class of non-K\"ahler complex manifolds in the works of Lopez de Medrano and Verjovsky~\cite{lo-ve97}, Meersseman~\cite{meer00} and Bosio~\cite{bosi01}. These manifolds are defined as the leaf spaces of exponential holomorphic foliations on the \emph{projectivisations} of open subsets $U_\C(\sK)\subset\C^m$ defined by polytopal fans (in the case of LVM manifolds) or complete simplicial fans (in the case of LVMB manifolds). Therefore, an LVMB manifold can be described as the projectivisation of a complex moment-angle manifold, or the partial quotient of a moment-angle manifold by the diagonal circle action. We describe this link here and draw some consequences.

We start with a spanning vector configuration $\Gamma=\{\gamma_1,\ldots,\gamma_m\}$ in $V^*$ for which there exists $\mb v_1\in V$ such that $\langle\gamma_i,\mb v_1\rangle=1$ for $i=1,\ldots,m$. Equivalently, the Gale dual configuration $\mathrm A=\{\mb a_1,\ldots,\mb a_m\}$ satisfies the relation 
$\sum_{k=1}^m\mb a_i=\bf0$. In particular, the conditions of Proposition~\ref{posit} are met. Let $R=\R\langle\mb v_1\rangle$.
% be the one-dimensional subspace of $V$ spanned by~$\mb v_1$. 
Thinking in terms of relations between the vectors of $\mathrm A$, the subspace $R$ is generated by a single relation with all coefficients~$1$, so it is rational. The lattice~\eqref{Llatt} is the 
$\Z$-span of~$\mb v_1$. Let $V'=V/R$ and choose a splitting $s\colon V'\to V$ as in Section~\ref{convlink}. Such a splitting is defined by a choice of $\delta\in V^*$ satisfying $\langle\delta,\mb v_1\rangle=1$, then $s(V')=\Ann\delta$. Let
$\Gamma'=\{\gamma'_1,\ldots,\gamma'_m\}$ be the corresponding point configuration in $(V')^*=\Ann R$, where $\gamma'_i=s^*(\gamma_i)=\gamma_i-\delta$.

Assume further that $\dim V'=k-1$ is even, and let $\mathcal J'\colon V'\to V'$ be a complex structure. Denote the resulting dimensional complex space $(V',\mathcal J')$ by~$\widetilde V'$. An LVMB-manifold is the quotient of $U_\C(\sK)$ by a proper holomorphic torus-exponential action of $\C^\times_d\times\widetilde V'$, where 
$\C^\times_d$ is the one-dimensional diagonal algebraic subtorus 
of~$(\C^\times)^m$. The original construction~\cite{bosi01} is as follows.

\begin{construction}[LVMB manifold]
Let $\Gamma'$ be a point configuration that affinely spans $(V')^*$, and
let $\mathcal E$ be a nonempty collection of subsets of $[m]$ such that for each $J\in\mathcal E$ the subconfiguration $\Gamma'_J$ is a minimal affine generating set of $(V')^*$ (in particular, each $J$ has cardinality~$k$).

A pair $(\mathcal E,\Gamma')$ is called an \emph{LVMB datum} if the two conditions are satisfied:
\[
\begin{tabular}{l p{0.65\textwidth}}
[\emph{imbrication}]& for any $I,J\in\mathcal E$, the polytopes 
$\conv\Gamma'_I$ and $\conv\Gamma'_J$ have a common interior point;\\[2pt]

[\emph{substitute existence}]& for any $J\in\mathcal E$ and $i\in[m]$ there exists $j\in J$ such that $(J\setminus\{j\})\cup\{i\}\in\mathcal E$.
\end{tabular}
\]

Define the following open subset in the projective space $\C P^{m-1}$:
%\begin{equation}\label{projE}
\[
\begin{aligned}
  P(\mathcal E)&=\bigl\{[z_1:\cdots: z_m]\in\C P^{m-1}\colon
  \text{there is $J\in\mathcal E$ such that }
  \{j\colon z_j\ne0\}\supset J\bigr\}\\
  &=\bigcup_{J\in\mathcal E}\bigl\{[z_1:\cdots: z_m]\in\C P^{m-1}\colon
   z_j\ne0\text{ for }j\in J\bigr\}.
\end{aligned}
\]
%\end{equation}
and the holomorphic exponential action
\begin{equation}\label{bact}
\begin{aligned}
  \widetilde V'\times P(\mathcal E)&\longrightarrow P(\mathcal E)\\
  (\mb v',\mb z)&\longmapsto \mb v\cdot\mb z=\bigl[e^{\langle\gamma'_1,\mb
  v'\rangle_\C}z_1:\cdots:e^{\langle\gamma'_m,\mb v'\rangle_\C}z_m\bigr].
\end{aligned}
\end{equation}
Bosio proves in~\cite[Theorem~1.4]{bosi01} that this action is proper if and only if 
$(\mathcal E,\Gamma')$ satisfies the imbrication condition, such a proper action is free, 
and the quotient $P(\mathcal E)/\widetilde V'$ of a proper action is compact if and only if the substitute existence condition is satisfied. 

The quotient  $\mathcal N=P(\mathcal E)/\widetilde V'$ defined by an LVMB datum 
$(\mathcal E,\Gamma')$ is called an \emph{LVMB manifold}. It is a compact complex manifold.
\end{construction}

To link the definition of LVMB manifold with the holomorphic exponential actions considered in Section~\ref{compzk}, we relate LVMB data $\{\mathcal E,\Gamma'\}$ to fan data (triangulated vector configurations). Let $\sK$ be the pure $(m-k-1)$-dimensional simplicial complex whose maximal simplices are complements to $J\in\mathcal E$, that is,
\begin{equation}\label{sclvmb}
  \sK=\{I\subset [m]\colon I\subset\widehat J\text{ for some }J\in\mathcal E\}.
\end{equation}
An element $i\in[m]$ is called \emph{indispensable} for $\mathcal E$ if it belongs to every $J\in\mathcal E$. Clearly, $i$ is indispensable if and only $\{i\}$ is a ghost vertex of~$\sK$.

\begin{proposition}\label{LVMBfan}\
\begin{itemize}
\item[(1)] $P(\mathcal E)\subset\C P^{m-1}$ is the projectivisation of $U_\C(\sK)\subset\C^m$;

\item[(2)] $\{\mathcal E,\Gamma'\}$ is an LVMB datum if and only if 
$\{\mathcal K,\mathrm A\}$ is a triangulated configuration defining a complete simplicial fan;

\item[(3)] the LVMB manifold $\mathcal N=P(\mathcal E)/\widetilde V'$ is the quotient of the moment-angle manifold $\zk$ by the diagonal free action of the circle~$S^1$, and the complex structure is the one of the quotient of $U_\C(\sK)$ by the holomorphic torus-exponential action of $\C^\times\times\widetilde V'$.
\end{itemize}
\end{proposition}
\begin{proof}
Statement (1) follows by comparing the definitions of $P(\mathcal E)$ and $U_\C(\sK)$, see~\eqref{UKC} and~\eqref{affui}.

To prove (2), first observe that the imbrication condition is equivalent to the condition that $\cone\Gamma_{\widehat I}$ and $\cone\Gamma_{\widehat J}$ have a common interior point for any $I,J\in\sK$ (compare Theorem~\ref{proper'}). The latter condition is equivalent to 
$\{\mathcal K,\mathrm A\}$ being a triangulated configuration by Proposition~\ref{galefan}. Now the fan is complete if and only if the substitute existence condition is satisfied, see Proposition~\ref{galecomp}.

For (3), observe that $P(\mathcal E)=U_\C(\sK)/\C^\times$ and $\mathcal N=P(\mathcal E)/\widetilde V'$. Hence, $\mathcal N=U_\C(\sK)/(\C^\times\times\widetilde V')$. The latter is the quotient~$\zk/S^1$ by Proposition~\ref{holpq1}.
\end{proof}

%\begin{proposition}
%Assume that $\{\sK,\mathrm A\}$ define a complete simplicial fan~$\Sigma$, and let $\{\mathcal E,\Gamma'\}$ be the corresponding LVMB datum. Then the corresponding LVMB manifold $\mathcal N$ is the quotient of the moment-angle manifold $\zk$ by the diagonal free action of the circle~$S^1$, i.\,e. there is a smooth principal $S^1$-bundle $\zk\to\mathcal N$.
%\end{proposition}
%\begin{proof}
%Similarly to~\eqref{ract'}, consider the action
%\[
%\begin{aligned}
%  (\R_>\times\widetilde V')\times U_\C(\sK)&\longrightarrow U_\C(\sK)\\
%  (r,\mb v',\mb z)&\mapsto \bigl(re^{\langle\gamma'_1,\mb
%  v'\rangle}z_1,\ldots,re^{\langle\gamma'_m,\mb v'\rangle}z_m\bigr),
%\end{aligned}
%\]
%where $\langle\gamma',\mb v'\rangle$ denotes the complex pairing. One can see in the same way as in Theorem~\ref{quotzk} that the quotient space of this action is identified with~$\zk$ (this time $k$ is odd, so we have a holomorphic action of $\widetilde V'$ and an action of $\R_>$ by dilations).  By 
%Proposition~\ref{LVMBfan}~(a), $P(\mathcal E)$ is the quotient of $U_\C(\sK)$ by the diagonal action of~$\C^\times$, so we obtain
%\[
%  \mathcal N=P(\mathcal E)/\widetilde V'\cong
% \bigl(U_\C(\sK)/(\R_>\times\widetilde V')\bigr)/(\C^\times/\R_>)\cong\zk/S^1.
%  \qedhere
%\]
%\end{proof}

By Proposition \ref{LVMBfan}, every LVMB manifold is the base of a principal $S^1$-bundle with total space a moment-angle manifold. There is also an entirely holomorphic version of this construction:

\begin{proposition}
Let $\mathcal N=P(\mathcal E)/\widetilde V'$ be an LVMB manifold and $\sK$ the corresponding simplicial complex~\eqref{sclvmb} on~$[m]$. Then there is a holomorphic principal bundle $\mathcal Z_{\sK_+}\to\mathcal N$ with fibre a one-dimensional complex torus and total space a complex moment-angle manifold corresponding to $\sK$ with one added ghost vertex.
\end{proposition}
\begin{proof}
Let $\{\sK,\mathrm A\}$ be the triangulated configuration corresponding to~$\mathcal N$. We add a ghost vertex to $\sK$. The triangulated configuration $\{\sK_+,\mathrm A_+\}$, where $\mathrm A_+=\{\mb a_1,\ldots,\mb a_m,\bf 0\}$, defines the same fan,  see Construction~\ref{addgv}. We have $V_+=V\oplus\R\langle\mb v_+\rangle=V'\oplus R\oplus\R\langle\mb v_+\rangle$. Then $Q=R\oplus\R\langle\mb v_+\rangle$ is a two-dimensional rational subspace of~$V_+$. It has a complex structure, which together with $\mathcal J'\colon V'\to V'$ defines a complex structure on~$V_+$. Denote the resulting complex spaces by~$\widetilde V_+=\widetilde V'\oplus\widetilde Q$. Then we have a diagram of holomorphic quotients~\eqref{diaghpq}:
\[
\begin{tikzcd}
  U_\C(\sK_+) \ar{r}{/\C^\times_L=\C^\times \times\C^\times} 
  \ar{d}{/\widetilde V_+=\widetilde V'\oplus\widetilde Q} 
   & X_{\widehat\Sigma}=P(\mathcal E) \ar{d}{/\widetilde V'}\\
  \mathcal Z_{\sK_+} \ar{r}{/\widetilde T_L} &  
  \mathcal Z_{\sK_+}/\widetilde T_L=P(\mathcal E)/\widetilde V',
\end{tikzcd}
\]
which describes the LVMB manifold $\mathcal N=P(\mathcal E)/\widetilde V'$ as the holomorphic quotient of the moment-angle manifold $\mathcal Z_{\sK_+}$ by a one-dimensional complex torus~$\widetilde T_L$.
\end{proof}

An \emph{LVM manifold} is an LVMB manifold whose corresponding data 
$\{\sK,\mathrm A\}$ define a normal fan. The original construction~\cite{meer00} is as follows.

\begin{construction}[LVM manifold]
Let $\Gamma'=\{\gamma'_1,\ldots,\gamma'_m\}$ be a point configuration in a complex vector space $(\widetilde V')^*$ of real dimension~$k-1$ satisfying the two conditions:
\[
\begin{tabular}{l p{0.65\textwidth}}
[\emph{Siegel condition}]& $\mathbf{0}\in\conv\Gamma'$;\\[2pt]

[\emph{weak hyperbolicity}]& if $\mathbf{0}\in\conv\Gamma'_J$, then $|J|\ge k$.
\end{tabular}
\]
Let
\[
  \mathcal E=\{J\subset[m]\colon \mathbf{0}\in\conv\Gamma'_J,\; |J|=k\}.
\]
It is shown in~\cite{meer00} that the holomorphic exponential action $\widetilde V'\times P(\mathcal E)\to P(\mathcal E)$ given by~\eqref{bact} is free and the quotient $P(\mathcal E)/\widetilde V'$ is identified with the nondegenerate intersection of projective quadrics
\begin{equation}\label{LVM}
  \mathcal N=\bigl\{[z_1\colon\cdots\colon z_m]\in\C P^{m-1}\colon
  \gamma'_1|z_1|^2+\cdots+\gamma'_m|z_m|^2=0\bigr\}.
\end{equation}
More precisely $\mathcal N\in P(\mathcal E)$, and each orbit of the $\widetilde V'$-action on $P(\mathcal E)$ intersects $\mathcal N$ transversely at a single point. Therefore, the compact manifold $\mathcal N$ acquires a complex structure as the holomorphic quotient $P(\mathcal E)/\widetilde V'$. It is called the \emph{LVM manifold} corresponding to the point configuration $\Gamma'=\{\gamma'_1,\ldots,\gamma'_m\}$ satisfying the Siegel and weak hyperbolicity conditions.
\end{construction}

Now let $V=V'\oplus\R$, $\gamma_i=\begin{pmatrix}\gamma'_i\\1\end{pmatrix}$, let $\mathrm A=\{\mb a_1,\ldots,\mb a_m\}$ be the Gale dual configuration to $\Gamma=\{\gamma_1,\ldots,\gamma_m\}$, and let $\sK$ be the simplicial complex given by~\eqref{sclvmb}. By the holomorphic versions of Theorems~\ref{quadrics1} and~\ref{quadrics2}, the data $\{\mathcal K,\mathrm A\}$ define the normal fan of a simple polytope $P$ given by~\eqref{ptope}, and and the quotient $U_\C(\sK)/(\R_>\times\widetilde V')$ is diffeomorphic to the nondegenerate intersection of Hermitian quadrics
\begin{equation}\label{Hlink}
  \left\{\begin{array}{lrcl}
  (z_1,\ldots,z_m)\in\C^m\colon&
  \gamma'_1|z_1|^2+\cdots+\gamma'_m|z_m|^2&=&\mathbf{0},\\[1mm]
  &|z_1|^2+\cdots+|z_m|^2&=&1.
  \end{array}\right\}
\end{equation}
The latter is known as the \emph{link of a system of special Hermitian quadrics}. The manifold~\eqref{LVM} is obtained by taking further quotient by the diagonal circle. It follows that an LVM manifold is an LVMB manifold whose defining fan data $\{\sK,\mathrm A\}$ is that of a normal fan.

LVM manifolds can also be characterised within LVMB manifolds in terms of the canonical foliation $\mathcal F$ on an LVMB manifold, introduced by Cupit-Foutou and Zaffran in~\cite{cu-za07}. The foliation~$\mathcal F$ turns into a holomorphic fibration over a toric orbifold with fibres compact tori in the rational case. Ishida~\cite{ishi17} proves that the canonical foliation~$\mathcal F$ is transversely K\"ahler if and only if the LVMB manifold is LVM.

Complex geometry of moment-angle manifolds and their partial quotients, including LVM- and LVMB-manifolds, is very rich and far from being completely understood. Analytic subsets and invariants of the complex structure of these manifolds, including the Hodge numbers and Dolbeault cohomology, holomorphic automorphism groups and moduli spaces of complex structures are currently the subject of active research. We refer to~\cite{me-ve04,bo-me06,pa-us12,tamb12,batt13,cai15,p-u-v16,kr-pa21,i-k-p22} for the aspects of complex geometry of these manifolds.

\section{Toric varieties and irrational deformations}\label{sectoric}

Complete nonsingular toric varieties are defined by nonsingular rational fans and can be obtained as partial quotients $\zk/T_V$; this corresponds to the case when the entire space $V$ is rational with respect to the vector configuration~$\Gamma$. By deforming the configuration $\Gamma$ we obtain holomorphic foliated manifolds $(\zk,\mathcal F)$ which model irrational deformations of simplicial toric varieties.

\subsection{Toric varieties as partial quotients}
Assume that the whole space $V$ is rational with respect to a spanning vector configuration $\Gamma=\{\gamma_1,\ldots,\gamma_m\}$, see Construction~\ref{coras}. Then 
\begin{equation}\label{Llatice}
   L=
  \{\mb v\in V\colon\langle\gamma_k,\mb v\rangle\in\Z\;
  \text{ for }k=1,\ldots,m\}.
\end{equation}
is a full rank lattice in $V$, and $L^*=\Z\langle\gamma_1,\ldots,\gamma_m\rangle$ is the dual lattice. The lattice $L$ defines the algebraic torus 
$\C^\times_L=L\otimes_\Z\C^\times$ and the compact torus $T_L=L\otimes_\Z S^1$ of dimension $k=\dim V$ with a product decomposition $\C^\times_L=V\times T_L$ as Lie groups.

The Gale dual rational vector configuration $\mathrm A=\{\mb a_1,\ldots,\mb a_m\}$ in $W^*$ spans a full rank lattice $M^*$, which we denote by~$N$.

The torus $\C^\times_L$ embeds in~$(\C^\times)^m$ as the closed holomorphic subgroup
\begin{equation}\label{cvsub}
 G= (\varGamma^*|_L\otimes_\Z\C^\times)
  (\C^\times_L)=\exp\varGamma^*_\C(V_\C)\\
  =\exp(\Ker A_\C\colon\C^m\to W^*_\C).
\end{equation}
This subgroup is in fact algebraic, as it can be written as
\begin{equation}\label{Ggrou}
  G=\Bigl\{(t_1,\ldots,t_m)\in(\C^\times)^m\colon
  \prod_{i=1}^m t_i^{\langle\mb a_i,\mb w\rangle}=1,\quad
  \mb w \in M\Bigr\}.
\end{equation}

The torus-exponential action~\eqref{teact} becomes the algebraic action of 
$\C^\times_L$ on $\C^m$. Given a simplicial complex $\sK$ on $[m]$, we obtain an algebraic action of $\C^\times_L$ on $U_\C(\sK)$ by restriction. By Theorem~\ref{teproper}, this action is proper if and only if $\{\sK,\mathrm A\}$ is a triangulated configuration, in which case the action is almost free with finite stabilisers.

When the vectors of the configuration $\mathrm A=\{\mb a_1,\ldots,\mb a_m\}$ are primitive in the lattice $N$ spanned by~$\mathrm A$, the (geometric) quotient $U_\C(\sK)/\C^\times_L$ by a proper algebraic action of 
$\C^\times_L$ is the \emph{toric variety} $X_{\Sigma}$ corresponding to the rational simplicial fan~$\Sigma$ defined by 
$\{\sK,\mathrm A\}$~\cite[\S5.1]{c-l-s11}. The algebraic torus $(\C^\times)^m/\C^\times_L=N\otimes_\Z\C^\times\cong(\C^\times)^n$ acts on $X_\Sigma=U_\C(\sK)/\C^\times_L$ with a Zariski open orbit. The toric variety $X_\Sigma$ is nonsingular if $\{\sK,\mathrm A\}$ is a nonsingular triangulated configuration (this implies that the action of $\C^\times_L$ is free and the vectors 
$\mb a_i$ are primitive in the lattice~$N$).

More generally, when the vectors $\mb a_i$ are not necessarily primitive in the lattice~$N$, the quotient $U_\C(\sK)/\C^\times_L$ is still the toric variety $X_{\Sigma}$ corresponding to the fan~$\Sigma$ defined by $\{\sK,\mathrm A\}$, yet with the non-standard orbifold structure. This phenomenon is treated in~\cite[\S4]{me-ve04}, see also~\cite{ba-za17}. We illustrate it on a simple example.

\begin{example}
Consider the two-vector configuration $\mathrm A=\{2,-1\}$ in one-dimensional space $W^*=\R$ and $\mathcal K=\{\varnothing,\{1\},\{2\}\}$. Then the triangulated configuration $\{\mathcal K,\mathrm A\}$ defines a complete nonsingular fan in $\R$, but the vector $\mb a_1=2$ is non-primitive in the lattice $N=\Z$ spanned by~$\mathrm A$. The Gale dual configuration is $\Gamma=\{1,2\}$ and the group~\eqref{cvsub} is given by
\[
  G=\{(e^v,e^{2v})=(t,t^2)\in (\C^\times)^2\}.
\]
The $G$-action on $U_\C(\mathcal K)=\C^2\setminus\{0\}$ has a nontrivial stabiliser of order $2$ at the point $(0,1)$.
The quotient $\C^2\setminus\{0\}/G$ is the projective line $\C P^1$ as a variety, but 
but comes equipped with an orbifold structure with one singular point of index~$2$.
\end{example}

When $V$ is entirely rational with respect to $\Gamma$, the commutative diagram of holomorphic quotients~\eqref{diaghpq} takes the form
\[
\begin{tikzcd}
  & U_\C(\sK) \ar{dd}{/\C^\times_L} \ar{ld}[swap]{/\widetilde V}\\
  \zk%\cong U_\C(\sK)/\widetilde V 
  \ar{rd}{/\widetilde T_L}\\
   & X_\Sigma
 \end{tikzcd}
\]
Here $\zk$ is the moment-angle manifold with the complex structure defined by the holomorphic quotient $U_\C(\sK)/\widetilde V$, and $\zk\to X_\Sigma$ is a holomorphic (Seifert) principal bundle with fibre a holomorphic compact torus $\widetilde T_L=\C^\times_L/\widetilde V$.

A complex structure on $\zk$ is defined by a triangulated configuration $\{\sK,\mathrm A\}$ and a choice of complex structure on~$V$, see \S\ref{compzk}. When the fan $\Sigma$ defined by $\{\sK,\mathrm A\}$ is rational, there is a holomorphic torus fibre bundle $\zk\to X_\Sigma$ over the corresponding toric variety. A perturbation of the vector configuration $\mathrm A$ destroys the rationality, the subgroup~\eqref{cvsub} ceases to be closed in  $(\C^\times)^m$, the closed holomorphic tori in the fibres of the bundle $\zk\to X_\Sigma$ ``open up'', and the fibre bundle turns into a \emph{holomorphic foliation} $\mathcal F$ of $\zk$ with noncompact leaves $G/\widetilde V$. Then $\mathcal F$ is an example of the \emph{canonical foliation} on a complex manifold, considered in~\cite{is-ka19} (see also~\cite{i-k-p22}). The holomorphic foliated manifolds $(\zk,\mathcal F)$ therefore model \emph{irrational deformations} of toric varieties, which have recently been studied from several different perspectives in 
the works~\cite{ba-pr01, ba-za15, ra-zu19, k-l-m-v21, i-k-p22, ba-pr23}.

On the other hand, the construction of toric varieties as quotients takes as input a lattice $N$ and a rational fan $\Sigma$ in the space $N\otimes_\Z\R$, which is~$W^*$ in our notation. The vectors $\mb a_1,\ldots,\mb a_m$ are the primitive lattice vectors generating the one-dimensional cones of~$\Sigma$. If the fan $\Sigma$ is simplicial, the toric variety $X_\Sigma$ can be defined as the geometric quotient $U_\C(\sK)/G$, where $G$ is the algebraic subgroup of $(\C^\times)^m$ given by~\eqref{Ggrou}. It may turn out that the vectors $\mb a_1,\ldots,\mb a_m$ do not span the lattice $N$; in particular, $\Z\langle \mb a_1,\ldots,\mb a_m\rangle$ may be a sublattice of finite index in~$N$. In the latter case the group~\eqref{Ggrou} is a product of an algebraic torus and a finite group $N/\Z\langle \mb a_1,\ldots,\mb a_m\rangle$. When we take as input a pair of Gale dual rational vector configurations $\Gamma$ and $\mathrm A$, the only natural choice for a lattice $N$ is $\Z\langle \mb a_1,\ldots,\mb a_m\rangle$, so the group~\eqref{Ggrou} does not have finite factors and is isomorphic to the algebraic torus~$\C^\times_L$. To take account of all simplicial toric varieties, one can follow the approach of Prato~\cite{prat01} to \emph{toric quasifolds} and include a finitely generated abelian subgroup $N$ of $W^*$ (a \emph{quasilattice}) containing the vectors $\mb a_1,\ldots,\mb a_m$ as part of the input data. In this way, all complete simplicial toric varieties (orbifolds), as well as toric quasifolds, can be obtained as quotients of LVMB manifolds~\cite[Theorem~3.4]{ba-za17}.

\subsection{The nef and ample cone of a toric variety}
A toric variety $X_\Sigma$ is defined by a lattice $N\cong\Z^n$ and a fan $\Sigma$ in $N\otimes_\Z\R=W^*$ which is rational with respect to~$N$. Let $\{\mathcal C,\mathrm A\}$ be the fan data of~$\Sigma$, where 
$\mathrm A=\{\mb a_1,\ldots,\mb a_m\}$ is the set of primitive generators of one-dimensional cones of~$\Sigma$ and $\mathcal C$ is an $\mathrm A$-closed collection of subsets of~$[m]$, see Section~\ref{secfan}. We continue assuming that $\mathrm A$ spans~$W^*$. The quotient construction~\cite[\S5.1]{c-l-s11} identifies $X_\Sigma$ with the good categorical quotient~$U_\C(\mathcal K)/\!\!/G$, where $\sK$ is the smallest simplicial complex containing~$\mathcal C$ and $G$ is given by~\eqref{Ggrou}. The categorical quotient~$U_\C(\mathcal K)/\!\!/G$ is geometric if and only if the fan $\Sigma$ is simplicial (so that $\mathcal C=\mathcal K$).

Let $M=N^*$ be the dual lattice in $W$. The map $A^*\colon W\to\R^m$ takes 
$\mb w\in W$ to $(\langle\mb a_1,\mb w\rangle,\ldots,\langle\mb a_m,\mb w\rangle)$. It can be interpreted in terms of torus-invariant Weil divisors on~$X_\Sigma$. Namely, we identify the $i$th standard basis vector $\mb e_i$ of $\Z^m$ with the irreducible torus-invariant divisor $D_i$ corresponding the the $i$th ray of $\Sigma$ (with primitive generator~$\mb a_i$). Then $A^*$ takes $\mb w\in M$ to the principal divisor $\sum_{i=1}^m\langle\mb a_i,\mb w\rangle D_i$ of the character $\chi^{\mb w}$ of the torus $\C^\times_N=N\otimes_\Z\C^\times$ corresponding to $\mb w\in M=\Hom(\C^\times_N,\C^\times)$. The quotient $\Z^m/A^*(M)$ is the \emph{class group} $\mathop\mathrm{Cl}(X_\Sigma)$, and we have an exact sequence of abelian groups
\[
  0\longrightarrow M \stackrel{A^*}\longrightarrow \Z^m 
  \longrightarrow \mathop\mathrm{Cl}(X_\Sigma) \longrightarrow 0.
\]
If $N=\Z\langle\mb a_1,\ldots,\mb a_m\rangle$, then the map $A^*\colon M\to\Z^m$ is split injective and $\mathop\mathrm{Cl}(X_\Sigma)$ is identified with the lattice $L^*$ spanned by the Gale dual rational configuration $\Gamma=\{\gamma_1,\ldots,\gamma_m\}$ in~$V^*$. In general, $\mathop\mathrm{Cl}(X_\Sigma)$ is isomorphic to the sum of $L^*$ and the finite group $N/\Z\langle \mb a_1,\ldots,\mb a_m\rangle$. In any case, we can identify $\gamma_i$ with the divisor class $[D_i]$ in $\mathop\mathrm{Cl}(X_\Sigma)\otimes_\Z\R=V^*$. Tensoring the exact sequence above with $\R$ we obtain
\begin{equation}\label{Rclas}
  0\longrightarrow W \stackrel{A^*}\longrightarrow \R^m 
  \stackrel{\varGamma}\longrightarrow V^* \longrightarrow 0.
\end{equation}

%We have a commutative diagrams with exact rows
%\[
%\begin{tikzcd}[row sep=small]
%  0 \ar{r} & M \ar{r}{A^*} \ar[hook]{d} & \Z^m  \ar{r}{\varGamma}\ar[hook]{d}  & L^* \ar{r} \ar[hook]{d} & 0\\
%  0 \ar{r} & W \ar{r}{A^*}  & 
%  \R^m \ar{r}{\varGamma} & V^* \ar{r} & 0
%\end{tikzcd}
%\]

In what follows we assume that the support $|\Sigma|=\bigcup_{\sigma\in\Sigma}\sigma\subset W^*$ of the fan $\Sigma$ is convex of full dimension. Torus-invariant Cartier divisors on $X_\Sigma$ correspond to piecewise linear functions $\varphi$ on $|\Sigma|$ which are linear on the cones of $\Sigma$ and take integer values on the lattice~$N$. The restriction of such a function $\varphi$ to a cone 
$\sigma\in\Sigma$ is given by $\varphi(\mb a)=\langle \mb a,\mb w_\sigma\rangle$ for $\mb a\in\sigma$ and some $\mb w_\sigma\in M$. A piecewise linear function $\varphi\colon|\Sigma|\to\R$ corresponds to the Cartier divisor $D=-\sum_{i=1}^m\varphi(\mb a_i)D_i$. Conversely, a torus-invariant Cartier divisor $D=\sum_{i=1}^m b_iD_i$ which is given by a collection of characters $\chi^{\mb w_\sigma}$ on the affine pieces $X_\sigma\subset X_\Sigma$ (the \emph{Cartier data}) corresponds to the piecewise linear function $\varphi_D \colon|\Sigma|\to\R$ given by 
$\varphi_D(\mb a)=\langle \mb a,\mb w_\sigma\rangle$ for $\mb a\in\sigma$.

The \emph{Picard group} $\mathop{\mathrm{Pic}}(X_\Sigma)\subset\mathop{\mathrm{Cl}}(X_\Sigma)$ consists of classes of Cartier divisors. It has a finite index in $\mathop{\mathrm{Cl}}(X_\Sigma)$ if and only if the fan $\Sigma$ is simplicial~\cite[Proposition~4.2.7]{c-l-s11}. When $\Sigma$ simplicial, we have
$\mathop\mathrm{Pic}(X_\Sigma)\otimes_\Z\R=\mathop\mathrm{Cl}(X_\Sigma)\otimes_\Z\R=V^*$, while in general $\mathop\mathrm{Pic}(X_\Sigma)\otimes_\Z\R$ is a proper subspace in~$V^*$.

A piecewise linear function $\varphi\colon|\Sigma|\to\R$ given by $\varphi(\mb a)=\langle \mb a,\mb w_\sigma\rangle$ for $\mb a\in\sigma$ is convex if and only if $\varphi(\mb a)\le\langle \mb a,\mb w_\sigma\rangle$ for all $\mb a\in|\Sigma|$, and $\varphi$ is called \emph{strictly convex} if $\varphi(\mb a)<\langle \mb a,\mb w_\sigma\rangle$ for all $\mb a\in|\Sigma|\setminus\sigma$. A fan $\Sigma$ with convex support admits a strictly convex piecewise linear function $\varphi\colon|\Sigma|\to\R$ if and only if $\Sigma=\Sigma_P$ is the normal fan of a convex polyhedron~$P$. Namely, for a strictly convex $\varphi\colon|\Sigma|\to\R$ one has that $\Sigma$ is the normal fan of the polyhedron
\begin{align*}
  P&=\{\mb w\in W\colon \langle\mb a,\mb w\rangle\ge\varphi(\mb a)
  \text{ for }\mb a\in|\Sigma|\}\\
  &=
  \{\mb w\in W\colon \langle\mb a_i,\mb w\rangle-\varphi(\mb a_i)\ge0,\quad 
  i=1,\ldots,m\}.
\end{align*}
Conversely, if $\Sigma=\Sigma_P$, then the \emph{support function} of $P$ given by $\varphi_P(\mb a)=\min_{\mb w\in P}\langle\mb a,\mb w\rangle$ is defined on 
$|\Sigma_P|$, linear on the cones of $\Sigma_P$ and strictly convex. 

A torus-invariant Cartier divisor $D$ on $X_\Sigma$ is basepoint free if and only if the corresponding function $\varphi_D\colon|\Sigma|\to\R$ is convex~\cite[Theorem~6.1.7]{c-l-s11}. When the fan $\Sigma$ is complete, a torus-invariant Cartier divisor $D$ on $X_\Sigma$ is ample if and only if the corresponding function $\varphi_D\colon|\Sigma|\to\R$ is strictly convex~\cite[Theorem~6.1.14]{c-l-s11}. In particular, a toric variety $X_\Sigma$ is projective if and only if $\Sigma$ is the normal fan of a convex polytope~$P$.

The classes of basepoint free divisors on $X_\Sigma$ span a cone in $V^*=\mathop\mathrm{Cl}(X_\Sigma)\otimes_\Z\R$, called the \emph{nef cone} and denoted by $\mathop\mathrm{Nef}(X_\Sigma)$ (a Cartier divisor on a toric variety $X_\Sigma$ is numerically effective if and only if it is basepoint free). The set of positive multiples of ample divisor classes on a complete toric variety $X_\Sigma$ is called the \emph{ample cone} and denoted by $\mathop\mathrm{Ample}(X_\Sigma)$. The variety $X_\Sigma$ is projective if and only if $\mathop\mathrm{Ample}(X_\Sigma)\ne\varnothing$. (Note that $\mathop\mathrm{Ample}(X_\Sigma)$ is not precisely a cone; it is a subset in $V^*$ whose closure is a polyhedral cone or empty.)  

The nef cone and the ample cone can be described effectively in terms of the Gale dual cones of~$\Sigma$. The proposition below is related to the criterion for a fan to be a normal fan in Theorem~\ref{galefannormgen} (see also~\cite[Proposition~2.4.2.6]{a-d-h-l15}):

\begin{proposition}\label{nefcone}
Let $\Sigma=\{\cone\mathrm A_I\colon I\in\mathcal C\}$ be a rational fan in $W^*$ with full-dimensional convex support, and let $\Gamma=\{\gamma_1,\ldots,\gamma_m\}$ be the Gale dual rational configuration in~$V^*$. 
\begin{itemize}
\item[(1)] The nef cone of the toric variety $X_\Sigma$ is given by
\[
  \mathop\mathrm{Nef}(X_\Sigma)=\bigcap_{I\in\mathcal C}\cone\Gamma_{\widehat I}.
\]

\item[(2)] If $\Sigma$ is complete, then the ample cone of $X_\Sigma$ is given by
\[
  \mathop\mathrm{Ample}(X_\Sigma)=\bigcap_{I\in\mathcal C}
  \relint\cone\Gamma_{\widehat I}
\]
and is nonempty if and only if $\Sigma$ is the normal fan of a polytope.

\smallskip

\item[(3)] If $\Sigma$ is complete and simplicial, then $\mathop\mathrm{Ample}(X_\Sigma)$ is the interior of~$\mathop\mathrm{Nef}(X_\Sigma)$, and $X_\Sigma$ is projective if and only if $\mathop\mathrm{Nef}(X_\Sigma)$ has full dimension in~$V^*$.
\end{itemize}
\end{proposition}
\begin{proof}
Recall that $\gamma_i\in V^*$ is identified with the divisor class $[D_i]$.

(1) Let $[D]=\sum_{i=1}^m b_i[D_i]\in V^*$ be a class of base free divisor, so that $b_i=-\varphi(\mb a_i)$ for a convex piecewise linear function $\varphi\colon|\Sigma|\to\R$. Since $\varphi$ is linear on any $\cone\mathrm A_I$ for $I\in\mathcal C$, we get $\varphi(\mb a_i)=\langle\mb a_i,\mb w_I\rangle$ for some $\mb w_I\in W$ and $i\in I$. Then
\[
  [D]=-\sum_{i=1}^m\varphi(\mb a_i)[D_i]+
  \sum_{i=1}^m\langle\mb a_i,\mb w_I\rangle[D_i]=
  \sum_{i\in\widehat I}(\langle\mb a_i,\mb w_I\rangle-\varphi(\mb a_i))\gamma_i
  \in\cone\Gamma_{\widehat I}
\]
for any $I\in\mathcal C$. Here the first identity above holds since $\sum_{i=1}^m\langle\mb a_i,\mb w_I\rangle D_i$ is a principal divisor (and $\sum_{i=1}^m\langle\mb a_i,\mb w_I\rangle\gamma_i=0$ by Gale duality), and the inclusion holds because $\langle\mb a_i,\mb w_I\rangle-\varphi(\mb a_i)\ge0$ for a convex~$\varphi$.  It follows that $[D]\in\bigcap_{I\in\mathcal C}\cone\Gamma_{\widehat I}$.

To prove the opposite inclusion, it is enough to show that any rational vector 
$\gamma\in\bigcap_{I\in\mathcal C}\cone\Gamma_{\widehat I}$ is a positive multiple of a basepoint free divisor class. We write $\gamma=\sum_{i=1}^m b_i[D_i]$ with rational~$b_i$. Take $I\in\mathcal C$. Since $\gamma\in\cone\Gamma_{\widehat I}$, we can also write 
$\gamma=\sum_{i\in\widehat I}b'_i[D_i]$ with rational $b'_i\ge0$. Now 
$\sum_{i\in\widehat I}b'_iD_i-\sum_{i=1}^m b_iD_i\in\Ker\varGamma$ (see~\eqref{Rclas}), so that there is a rational $\mb w_I\in W$ such that
\begin{equation}\label{bb'}
  \sum_{i=1}^m\langle\mb a_i,\mb w_I\rangle D_i=\sum_{i\in\widehat I}b'_iD_i-\sum_{i=1}^m b_iD_i.
\end{equation}
The data $\{\mb w_I\colon I\in\mathcal C\}$ define a piecewise linear function $\varphi\colon|\Sigma|\to\R$ given by $\varphi(\mb a)=\langle\mb a,\mb w_I\rangle$ for $\mb a\in\cone\mathrm A_I$. From~\eqref{bb'} we obtain $\varphi(\mb a_i)=-b_i$ for $i\in I$. Since each $\mb a_i$ belongs to a cone, we obtain $\varphi(\mb a_i)=-b_i$ for any~$i$. Identity~\eqref{bb'} also implies that
\[
  \langle\mb a_i,\mb w_I\rangle-\varphi(\mb a_i)=
  \langle\mb a_i,\mb w_I\rangle+b_i=b'_i\ge0
  \quad\text{for }i\in\widehat I,
\]
so that $\varphi$ is convex. Since $\varphi(\mb a_i)$ is rational for any~$i$, multiplying by a common denominator we achieve that $\varphi$ is integer on~$N$. It follows that $\gamma=\sum_{i=1}^m b_i[D_i]=-\sum_{i=1}^m\varphi(\mb a_i)[D_i]$ is a basepoint free divisor class up to a positive multiple.

\smallskip

The proof of (2) is similar: non-strict inequalities like $\langle\mb a_i,\mb w_I\rangle-\varphi(\mb a_i)\ge0$ and $b'_i\ge0$ are replaced by strict ones and the function $\varphi$ is now strictly convex. The fact that $\bigcap_{I\in\mathcal C}
  \relint\cone\Gamma_{\widehat I}$ is nonempty if and only if $\Sigma$ is the normal fan of a polytope is proved in Theorem~\ref{galefannormgen}.
  
\smallskip

(3) When $\Sigma$ is simplicial, each $\cone\Gamma_{\widehat I}$ is a full-dimensional cone in~$V^*$ for $I\in\mathcal C$, and its relative interior is the interior. Then
\[
  \mathop\mathrm{Ample}(X_\Sigma)=\bigcap_{I\in\mathcal C}
  \mathop\mathrm{int}\cone\Gamma_{\widehat I}=\mathop\mathrm{int}  
  \bigcap_{I\in\mathcal C}\cone\Gamma_{\widehat I}=
  \mathop\mathrm{int}\mathop\mathrm{Nef}(X_\Sigma),
\]
and this is nonempty if and only if $\mathop\mathrm{Nef}(X_\Sigma)$ has full dimension.
\end{proof}

When $X_\Sigma$ is not simplicial, $\mathop\mathrm{Nef}(X_\Sigma)$ cannot have full dimension in $V^*$, as $\mathop\mathrm{Nef}(X_\Sigma)\subset\mathop\mathrm{Pic}(X_\Sigma)\otimes_\Z\R$ and $\mathop\mathrm{Pic}(X_\Sigma)\otimes_\Z\R$ is a proper subspace in~$V^*$. However, $\mathop\mathrm{Nef}(X_\Sigma)$ can have full dimension in $\mathop\mathrm{Pic}(X_\Sigma)\otimes_\Z\R$. When $X_\Sigma$ is projective, $\mathop\mathrm{Ample}(X_\Sigma)$ is the relative interior of $\mathop\mathrm{Nef}(X_\Sigma)$ (or the interior of $\mathop\mathrm{Nef}(X_\Sigma)$ in $\mathop\mathrm{Pic}(X_\Sigma)\otimes_\Z\R$). On the other hand, $\mathop\mathrm{Nef}(X_\Sigma)$ may have full dimension in $\mathop\mathrm{Pic}(X_\Sigma)\otimes_\Z\R$ with $X_\Sigma$ still not being projective. To illustrate this, we elaborate on an example due to Fujino~\cite{fuji05}, see also~\cite[Example 6.3.26]{c-l-s11}.

\begin{example}
Consider the same vector configurations $\mathrm A$ and $\Gamma$ as in Example~\ref{prism1}, given by the columns of the matrices
\[
  A=\begin{pmatrix}1&0&-1&1&0&-1\\0&1&-1&0&1&-1\\
  1&1&1&-1&-1&-1\end{pmatrix},\quad
    \varGamma=\begin{pmatrix}1&-1&0&-1&1&0\\0&1&-1&0&-1&1\\
  1&1&1&1&1&1\end{pmatrix}.
\]   
We subdivide the boundary of the triangular prism $\conv(\mb a_1,\ldots,\mb a_6)$
as shown in Fig.~\ref{figFujino}, left, and consider the complete fan 
$\Sigma$ with apex at $\bf 0$ over the faces of this subdivision. It has $6$ three-dimensional cones $\cone\mathrm A_I$, $I\in\mathcal C$, of which two are not simplicial, namely $\cone(\mb a_1,\mb a_3,\mb a_4,\mb a_6)$ and $\cone(\mb a_2,\mb a_3,\mb a_5,\mb a_6)$. The vectors
$\gamma_i=\begin{pmatrix}\gamma'_i\\1\end{pmatrix}$ lie in the affine plane~$x_3=1$, and the corresponding $2$-dimensional point configuration 
$\Gamma'=(\gamma'_1,\ldots,\gamma'_6)$ is shown in Fig.~\ref{figFujino}, right.

The four triangles $\conv\Gamma'_{\widehat I}$ corresponding to three-dimensional simplicial cones of $\Sigma$ are shown in Fig.~\ref{figFujino}, right, together with the two segments $\conv(\gamma'_2,\gamma'_5)$ and $\conv(\gamma'_1,\gamma'_4)$ corresponding to non-simplicial cones. Their intersection is a single point, while the intersection of the relative interiors is empty. Therefore, $\mathop\mathrm{Nef}(X_\Sigma)$ is a one-dimensional cone with generator $\gamma_1+\gamma_4$, while $\mathop\mathrm{Ample}(X_\Sigma)$ is empty, so $\Sigma$ is not a normal fan and $X_\Sigma$ is not projective.
\begin{figure}[h]
\begin{tikzpicture}[scale=0.8]
  \draw[thick] (0,2)--(0,5)--(6,5)--(6,2)--(1.5,0.5)--cycle;
  %\draw[thick] (1.5,0.5)--(0,5)--(1.5,3.5)--cycle;
  \draw[thick] (1.5,3.5)--(0,5);
  %\draw[thick] (6,2)--(1.5,3.5)--(6,5);
  \draw[thick] (1.5,0.5)--(1.5,3.5)--(6,5)--cycle;
  %\draw[dashed] (6,2)--(0,2)--(6,5);
  \draw[dashed] (0,2)--(6,2);
  \draw[fill] (1.5,3.5) circle (1.6pt) node[anchor=south]{$\ \mb a_1$};
  \draw[fill] (6,5) circle (1.6pt) node[anchor=west]{$\mb a_2$};
  \draw[fill] (0,5) circle (1.6pt) node[anchor=east]{$\ \mb a_3$};
  \draw[fill] (1.5,0.5) circle (1.6pt) node[anchor=north]{$\mb a_4$};
  \draw[fill] (6,2) circle (1.6pt) node[anchor=west]{$\mb a_5$};
  \draw[fill] (0,2) circle (1.6pt) node[anchor=east]{$\mb a_6$};
  \draw[fill] (2.5,3) circle (1.6pt) node[anchor=north]{$\bf0$};
  \draw[thick] (8,3)--(8,6)--(11,6)--(14,3)--(14,0)--(11,0)--cycle;
%356
  \draw[fill=yellow!40, opacity=0.5] (11,0)--(14,0)--(11,6)--cycle;
%136
  \draw[fill=blue!40, opacity=0.5] (14,3)--(11,0)--(11,6)--cycle;
%456  
  \draw[fill=red!20, opacity=0.5] (8,3)--(14,0)--(11,6)--cycle;
%123  
  \draw[fill=cyan!20, opacity=0.5] (14,3)--(8,6)--(11,0)--cycle;
%25
  \draw (8,6)--(14,0);
%14
  \draw (14,3)--(8,3);  
  \draw[fill] (11,3) circle (1.6pt);  
  \draw[fill] (14,3) circle (1.6pt) node[anchor=west]{$\gamma_1$};
  \draw[fill] (8,6) circle (1.6pt) node[anchor=south east]{$\gamma_2$};
  \draw[fill] (11,0) circle (1.6pt) node[anchor=north]{$\gamma_3$};
  \draw[fill] (8,3) circle (1.6pt) node[anchor=east]{$\gamma_4$};
  \draw[fill] (14,0) circle (1.6pt) node[anchor=north west]{$\gamma_5$};
  \draw[fill] (11,6) circle (1.6pt) node[anchor=south]{$\gamma_6$};  
\end{tikzpicture}
\caption{A non-projective complete toric variety with the nef cone of full dimension in the Picard group}
\label{figFujino}
\end{figure}

It can be shown that $\mathop\mathrm{Pic}(X_\Sigma)=\Z\langle3(\gamma_1+\gamma_4)\rangle$, so $\mathop\mathrm{Nef}(X_\Sigma)$ has full dimension in $\mathop\mathrm{Pic}(X_\Sigma)\otimes_\Z\R$, yet $3(\gamma_1+\gamma_4)$ is a Cartier divisor class in the interior of $\mathop\mathrm{Nef}(X_\Sigma)$ which is not ample.
\end{example}

\subsection{The K\"ahler cone and symplectic reduction}\label{parsr}
A smooth projective toric variety $X_\Sigma$ is a K\"ahler manifold, so in particular it has a symplectic structure, obtained by restricting the Fubini--Study form via an embedding into a projective space. The algebraic torus action on $X_\Sigma$ restricts to the action of the compact torus $T_N=N\otimes_\Z S^1\cong T^n$ which is  Hamiltonian with respect to the symplectic structure on~$X_\Sigma$. On the other hand, by the Delzant theorem~\cite{delz88}, any $2n$-dimensional compact connected symplectic manifold $M$ with an effective Hamiltonian action of $T^n$ (a \emph{Hamiltonian toric manifold}) is determined up to an equivariant symplectomorphism by the image of its moment map $\mu\colon M\to\mathop\mathrm{Lie}(T^n)^*=W$. This image is a convex polytope $P$ with nonsingular rational normal fan~$\Sigma_P$. Therefore, any Hamiltonian toric manifold is the underlying symplectic manifold of a smooth projective toric variety. It follows that a smooth complete toric variety is projective if and only if it has a symplectic structure.

When $X_\Sigma$ is complete and smooth, the interior of $\mathop\mathrm{Nef}(X_\Sigma)$ consists of K\"ahler classes in $V^*=H^2(X_\Sigma;\R)$ and is known as the \emph{K\"ahler cone}. It is empty when $X_\Sigma$ is nonprojective. The space $V^*$ contains the lattice $L^*=\Z\langle\gamma_1,\ldots,\gamma_m\rangle=\mathop\mathrm{Pic}(X_\Sigma)=H^2(X_\Sigma;\Z)$, and lattice points in the interior of $\mathop\mathrm{Nef}(X_\Sigma)$ correspond to K\"ahler classes coming from ample divisors on~$X_\Sigma$.

A Hamiltonian toric manifold $M$ can be reconstructed from its moment polytope $P$ using \emph{symplectic reduction}. We outline the construction below; the details can be found in~\cite[Chapter~6]{audi91} and~\cite[Chapter~1]{guil94}.

\begin{construction}[Hamiltonian toric manifolds as symplectic quotients]
The symplectic structure on $\C^m$ with coordinates $(z_1,\ldots,z_m)$, $z_k=x_k+iy_k$, is given by the $2$-form
\[
  \omega=-i\sum_{i=1}^m dz_k\wedge d\overline z_k=-2\sum_{i=1}^m dx_k\wedge dy_k.
\]  
The coordinatewise action of the torus $T^m$ on $\C^m$ is Hamiltonian with the moment map given by
\[
  \mu\colon\C^m\to\mathop\mathrm{Lie}(T^m)^*=\R^m,\quad
  (z_1,\ldots,z_m)\mapsto(|z_1|^2,\ldots,|z_m|^2).
\]

A rational vector configuration $\Gamma=\{\gamma_1,\ldots,\gamma_m\}$ in $V^*\cong\R^k$ with $L^*=\Z\langle\gamma_1,\ldots,\gamma_m\rangle\cong\Z^k$ and the dual lattice~\eqref{Llatice} defines a torus $T_L=L\otimes_\Z S^1$ and an inclusion of tori $\varGamma^*\otimes_\Z S^1\colon T_L\to T^m$, so that
\begin{equation}\label{Ktorus}
  K=(\varGamma^*\otimes_\Z S^1)(T_L)=
  \bigl\{\bigl(e^{2\pi i\langle\gamma_1,\mb v\rangle},\ldots,
  e^{2\pi i\langle\gamma_m,\mb v\rangle}\bigr)\colon \mb v\in L  \bigr\}
\end{equation}
is a $k$-dimensional subtorus in $T^m$. 

The moment map for the Hamiltonian action of $K$ on $\C^m$ is obtained by restriction and is given by the composite $\C^m\stackrel\mu{\longrightarrow}\mathop\mathrm{Lie}(T^m)^*\to\mathop\mathrm{Lie}(T_L)^*$. The second map here is identified with $\varGamma\colon\R^m\to V^*$, so we obtain
\[
  \mu_\Gamma\colon\C^m\to V^*, \quad (z_1,\ldots,z_m)\mapsto     
  |z_1|^2\gamma_1+ \cdots+|z_m|^2\gamma_m.
\]
The real version of this map was considered in Subsection~\ref{secqua} and its regular values are described in terms of $\Gamma$ in Theorems~\ref{quadrics} and~\ref{quadrics1}.

If $\delta\in V^*$ is a regular value of the moment map $\mu_\Gamma$, then
the intersection of Hermitian quadrics
\[
  \mu_\Gamma^{-1}(\delta)=\{(z_1,\ldots,z_m)\in\C^m\colon
  \gamma_1 |z_1|^2+\cdots+\gamma_m|z_m|^2=\delta\}
\]
is nondegenerate, so it is a smooth submanifold of dimension $2m-k=m+n$. The action of $K\cong T_L$ on $\mu_\Gamma^{-1}(\delta)$ is almost free. The stabiliser of a point $\mb z\in\mu_\Gamma^{-1}(\delta)$ is a finite group $L_I/L$ (see Proposition~\ref{stabt}), where $I=\{i\in[m]\colon z_i=0\}$ and
\[
  L_I
  =\{\mb v\in V\colon\langle\gamma_k,\mb v\rangle\in\Z\;\text{ for }k\notin I\}.
\]

The restriction of the symplectic form $\omega$ to $\mu_\Gamma^{-1}(\delta)$ degenerates precisely along the orbits of the $K$-action.
If the action of $K$ on $\mu_\Gamma^{-1}(\delta)$ is free, then the quotient $\mu_\Gamma^{-1}(\delta)/K$ is a manifold, and there is a unique symplectic form $\omega'$ on $\mu_\Gamma^{-1}(\delta)/K$ satisfying $p^*\omega'=i^*\omega$, where $i\colon\mu_\Gamma^{-1}(\delta)\hookrightarrow\C^m$ is the inclusion and $p\colon \mu_\Gamma^{-1}(\delta)\to\mu_\Gamma^{-1}(\delta)/K$ is the projection. The symplectic manifold $M_{\Gamma,\delta}=\mu_\Gamma^{-1}(\delta)/K$ is referred to as the \emph{symplectic quotient} of $\C^m$ by~$K$.  

The $2n$-dimensional manifold $M_{\Gamma,\delta}=\mu_\Gamma^{-1}(\delta)/K$ has a residual Hamiltonian action of the quotient $n$-torus $T_N=T^m/K$. The image of the corresponding moment map $\mu_N\colon M_{\Gamma,\delta}\to\mathop\mathrm{Lie}(T_N)^*=W$ is the convex polytope given by
\begin{equation}\label{ptope3}
  P=\{\mb w\in W\colon\langle \mb a_i,\mb w\rangle+b_i\ge0,\quad i=1,\ldots,m\},
\end{equation}
where $\mathrm A=\{\mb a_1,\ldots,\mb a_m\}$ is the Gale dual configuration of~$\Gamma$ and $\delta=\sum_{i=1}^mb_i\gamma_i$, see Theorem~\ref{quadrics}. (Note that $P$ is defined by $(\Gamma,\delta)$ up to a translation in~$W$.) The normal fan $\Sigma_P$ of $P$ is defined by the triangulated configuration $(\mathcal K_P,\mathrm A)$, and it is nonsingular if and only if the action of $K$ on $\mu_\Gamma^{-1}(\delta)$ is free, see Proposition~\ref{nsdual}. A polytope $P$ with a nonsingular normal fan is called \emph{Delzant}. (Note that the vertices of a Delzant polytope do not necessarily lie in the lattice $M=N^*\subset W$.) 

Conversely, given a Delzant polytope $P$, the Gale dual configuration $\Gamma$ defines a $k$-dimensional subtorus~\eqref{Ktorus} in~$T^m$, and the regular value $\delta=\sum_{i=1}^mb_i\gamma_i$ of the moment map $\mu_\Gamma$, as in Theorem~\ref{quadrics}. The level set $\mu_\Gamma^{-1}(\delta)$ is identified with the moment-angle manifold $\mathcal Z_{\mathcal K_P}$ (Theorem~\ref{hquadrics}).

By a theorem of Delzant~\cite{delz88}, every Hamiltonian toric manifold $M$ can be obtained as a symplectic quotient $M_{\Gamma,\delta}$ of $\C^m$ by a torus, so $M$ is determined up to an equivariant symplectomorphism by its moment image~$P$.
\end{construction}

The symplectic and algebraic quotient constructions are related as follows.

\begin{theorem}\label{algsym}
Let $P$ be a polytope~\eqref{ptope3} with nonsingular normal fan $\Sigma_P$. The inclusion $\mu_\Gamma^{-1}(\delta)\subset U_\C(\mathcal K_P)$ induces a $T_N$-equivariant diffeomorphism
\[
  M_{\Gamma,\delta}=\mu_\Gamma^{-1}(\delta)/K\stackrel\cong\longrightarrow
  U_\C(\mathcal K_P)/G=X_{\Sigma_P}
\]
between the corresponding Hamiltonian toric manifold and smooth projective toric variety.

Furthermore, if the polytope $P$ has vertices in the lattice $M\subset W$, then the symplectic structure on $X_{\Sigma_P}$ defined by the ample divisor corresponding to the support function $\varphi_P$ coincides with that of $M_{\Gamma,\delta}$.
\end{theorem}

\begin{proof}
To prove the first statement we need to check that each $G$-orbit of $U_\C(\mathcal K_P)$ intersects $\mu_\Gamma^{-1}(\delta)$ at a single $K$-orbit. Using the product decomposition $G=\C^\times_L=V\times T_L=V\times K$, we need to the check that each $V$-orbit of $U(\mathcal K_P)$ intersects $\mu_\Gamma^{-1}(\delta)$ at a single point. This is proved in Theorem~\ref{quadrics}~(4). The second statement follows from the Delzant theorem, as a $T_N$-equivariant symplectic structure on $X_{\Sigma_P}$ is determined by its moment polytope~$P$ uniquely.
\end{proof}

The complex structure on the algebraic quotient $U_\C(\mathcal K_P)/G$ is compatible with the symplectic structure on the symplectic quotient 
$\mu_\Gamma^{-1}(\delta)/K$, so that the symplectic form $\omega'$ on 
$\mu_\Gamma^{-1}(\delta)/K$ is K\"ahler. See~\cite[Appendix~1, Theorem~1.5]{guil94}.

There is a generalisation of Theorem~\ref{algsym} to irrational simple polytopes $P$~\cite[Theorem~3.2]{ba-pr01}, which relates the complex and symplectic \emph{quasifolds} associated with $P$ and a quasilattice containing the normal configuration~$\mathrm A$.

\end{document}